\pdfoutput=1
\RequirePackage{ifpdf}
\ifpdf 
\documentclass[pdftex]{sigma}
\else
\documentclass{sigma}
\fi

\usepackage[mathscr]{eucal} 
\usepackage{epic}
\usepackage{eepic}
\usepackage{array}

\usepackage{bm}
\usepackage[all]{xy}
\usepackage{tikz}
\usetikzlibrary{calc}

\newcommand\qarrow[2]{\draw[->,shorten >=2pt,shorten <=2pt] (#1) -- (#2) [thick];} 
\newcommand\qdarrow[2]{\draw[->,dashed,shorten >=2pt,shorten <=2pt] (#1) -- (#2) [thick];} 

\newcommand{\Z}{{\mathbb Z}}
\newcommand{\R}{{\mathbb R}}
\newcommand{\C}{{\mathbb C}}

\newcommand{\bb}{{\mathfrak b}}
\newcommand{\e}{{\mathrm e}}

\def\ve{{\varepsilon}}

\newcommand\trop{{\mathrm{trop}}}
\newcommand\rY{{\mathsf{Y}}}
\newcommand\Ad{{\mathrm{Ad}}}
\newcommand\bY{Y}

\newcommand{\uu}{{\mathsf u}}
\newcommand{\ww}{{\mathsf w}}

\numberwithin{equation}{section}

\newtheorem{Theorem}{Theorem}[section]
\newtheorem{Corollary}[Theorem]{Corollary}
\newtheorem{Lemma}[Theorem]{Lemma}
\newtheorem{Proposition}[Theorem]{Proposition}
 { \theoremstyle{definition}
\newtheorem{Example}[Theorem]{Example}
\newtheorem{Remark}[Theorem]{Remark} }

\begin{document}

\allowdisplaybreaks

\newcommand{\arXivNumber}{2403.08814}

\renewcommand{\PaperNumber}{113}

\FirstPageHeading

\ShortArticleName{Solutions of Tetrahedron Equation from Quantum Cluster Algebra}

\ArticleName{Solutions of Tetrahedron Equation\\ from Quantum Cluster Algebra Associated\\with Symmetric Butterfly Quiver}

\Author{Rei INOUE~$^{\rm a}$, Atsuo KUNIBA~$^{\rm b}$, Xiaoyue SUN~$^{\rm c}$, Yuji TERASHIMA~$^{\rm d}$ and Junya YAGI~$^{\rm e}$}

\AuthorNameForHeading{R.~Inoue, A.~Kuniba, X.~Sun, Y.~Terashima and J.~Yagi}

\Address{$^{\rm a)}$~Department of Mathematics and Informatics, Faculty of Science, Chiba University,\\
\hphantom{$^{\rm a)}$}~Chiba, 263-8522, Japan}
\EmailD{\href{mailto:reiiy@math.s.chiba-u.ac.jp}{reiiy@math.s.chiba-u.ac.jp}}

\Address{$^{\rm b)}$~Institute of Physics, Graduate School of Arts and Sciences, University of Tokyo,\\
\hphantom{$^{\rm b)}$}~Komaba, Tokyo 153-8902, Japan}
\EmailD{\href{mailto:atsuo.s.kuniba@gmail.com}{atsuo.s.kuniba@gmail.com}}

\Address{$^{\rm c)}$~Department of Mathematical Sciences and Yau Mathematical Sciences Center,\\
\hphantom{$^{\rm c)}$}~Tsinghua University, Haidian District, Beijing, 100084, P.R.~China}
\EmailD{\href{mailto:sunxy20@mails.tsinghua.edu.cn}{sunxy20@mails.tsinghua.edu.cn}}

\Address{$^{\rm d)}$~Graduate School of Science, Tohoku University, 6-3, Aoba, Aramaki-aza, Aoba-ku,\\
\hphantom{$^{\rm d)}$}~Sendai, 980-8578, Japan}
\EmailD{\href{mailto:yujiterashima@tohoku.ac.jp}{yujiterashima@tohoku.ac.jp}}

\Address{$^{\rm e)}$~Yau Mathematical Sciences Center, Tsinghua University,\\
\hphantom{$^{\rm e)}$}~Haidian District, Beijing, 100084, P.R.~China}
\EmailD{\href{mailto:junyagi@tsinghua.edu.cn}{junyagi@tsinghua.edu.cn}}

\ArticleDates{Received March 26, 2024, in final form December 05, 2024; Published online December 21, 2024}

\Abstract{We construct a new solution to the tetrahedron equation by further pursuing the quantum cluster algebra approach in our previous works. The key ingredients include a symmetric butterfly quiver attached to the wiring diagrams for the longest element of type $A$ Weyl groups and the implementation of quantum $Y$-variables through the $q$-Weyl algebra. The solution consists of four products of quantum dilogarithms. By exploring both the coordinate and momentum representations, along with their modular double counterparts, our solution encompasses various known three-dimensional (3D) $R$-matrices. These include those obtained by Kapranov--Voevodsky (1994) utilizing the quantized coordinate ring, Bazhanov--Mangazeev--Sergeev (2010) from a quantum geometry perspective, Kuniba--Matsuike--Yoneyama (2023) linked with the quantized six-vertex model, and Inoue--Kuniba--Terashima (2023) associated with the Fock--Goncharov quiver. The 3D $R$-matrix presented in this paper offers a unified perspective on these existing solutions, coalescing them within the framework of quantum cluster algebra.}


\Classification{82B23; 81R12; 13F60}

\section{Introduction}

The tetrahedron equation \cite{Z80} is a generalization of the
Yang--Baxter equation \cite{B82} to three-dimensional systems. A
fundamental form of the equation in the so-called vertex formulation
reads
$R_{124}R_{135}R_{236}R_{456} = R_{456}R_{236}R_{135}R_{124}$,
where $R$ is a linear operator on $V^{\otimes 3}$ for some vector
space $V$, and the indices specify the tensor components in
$V^{\otimes 6}$ on which it acts non-trivially.\looseness=-1

In this paper, we construct a new solution to the tetrahedron equation
by an approach based on quantum cluster algebras \cite{FG09,FG09b}.
This method, initiated in \cite{SY22} and further developed in~\cite{IKT1, IKT2},
commences with the Weyl group of type $A$ and employing wiring diagrams
to represent reduced expressions of the longest element in a standard format.
One introduces a specific {\em quiver} and the corresponding quantum cluster algebra linked to these wiring diagrams.
The pivotal element of the approach is the {\em cluster transformation} $\hat{R}$
serving as a counterpart of the cubic Coxeter relation.
It acts on quantum $Y$-variables through a sequence of mutations and permutations.
From the consideration about the embedding $A_2 \hookrightarrow A_3$,
$\hat{R}$ is shown to satisfy the tetrahedron equation.
Apart from the monomial part, $\hat{R}$ is described as an adjoint action of quantum dilogarithms.
The next key step is to devise a realization of the quantum
$Y$-variables in terms of a direct product of $q$-Weyl algebras which
is an exponential version of the algebra of canonical coordinates
$\uu_i$ and momenta $\ww_i$ with relations
$\e^{\uu_i} \e^{\ww_j} = q^{\delta_{ij}}\e^{\ww_j} \e^{\uu_i}$,
$\e^{\uu_i} \e^{\uu_j} = \e^{\uu_j} \e^{\uu_i}$ and
$\e^{\ww_i} \e^{\ww_j} = \e^{\ww_j} \e^{\ww_i}$. It allows for the cluster
transformation, including its monomial part, to be fully expressed in
the adjoint form $\hat{R}= \mathrm{Ad}(R)$. It is this $R$ which has
many interesting features connected to existing solutions. The
operator $R$ can be endowed with several ``spectral parameters'' and
satisfies the tetrahedron equation on its own including these
parameters.

We execute the above program for the {\em symmetric butterfly} (SB)
quiver, which is a symmetrized version of the butterfly quiver
introduced in \cite{SY22}. The vertices of an SB quiver are placed
{\em both} on the vertices of the wiring diagram and within its
domains. This contrasts with the Fock--Goncharov (FG) and the square
quivers studied in \cite{IKT1} and \cite{IKT2,SY22}, respectively.
In the former, vertices are assigned to the domains of the wiring
diagrams, while in the latter, they are assigned to the edges.

Apart from $q=\e^\hbar$, our $R$-matrix $R=R_{123}$ involves
parameters $C_1, \dots, C_8$ subject to $C_5+C_6=C_7+C_8$.
(See Remark \ref{re:X}.)
Up to normalization, it is given by
\begin{gather}
R =
\Psi_q\bigl(\e^{2C_7+\uu_1+\uu_3+\ww_1-\ww_2+\ww_3}\bigr)^{-1}
\Psi_q\bigl(\e^{2C_5+\uu_1-\uu_3+\ww_1-\ww_2+\ww_3}\bigr)^{-1}\nonumber
\\
\phantom{R =}{} \times P
\Psi_q\bigl(\e^{2C_2+2C_3-2C_6+2C_8+\uu_1-\uu_3+\ww_1-\ww_2+\ww_3}\bigr)
\Psi_q\bigl(\e^{2C_2+2C_3+\uu_1+\uu_3+\ww_1-\ww_2+\ww_3}\bigr),\nonumber
\\
P= \e^{\tfrac{1}{\hbar}(\uu_3-\uu_2)\ww_1}
\e^{\tfrac{1}{\hbar}\lambda_0(-\ww_1-\ww_2+\ww_3)}
\e^{\tfrac{1}{\hbar}(\lambda_1\uu_1+\lambda_2\uu_2+\lambda_3\uu_3)}\rho_{23},\label{Rint}
\end{gather}
where $\Psi_q$ is the quantum dilogarithm \eqref{Psiq}, $\lambda_i$'s
are linear combinations of $C_1, \dots, C_8$ in \eqref{lac}, and
$\rho_{23}$ is the permutation $(\uu_2,\ww_2) \leftrightarrow (\uu_3,\ww_3)$.

The result \eqref{Rint} is universal within the current approach based on the SB quiver.
In fact, one can project it onto various representations of the canonical variables.
Our final result for the matrix elements $\langle \mathbf{n}| R\bigl| \mathbf{n}'\bigr\rangle$
(up to normalization) with bases labeled by
$\mathbf{n} = (n_1, n_2, n_3)$ and~${\mathbf{n}' = \bigl(n'_1, n'_2, n'_3\bigr) \in \Z^3}$ reads
\begin{align}
\bigl\langle \mathbf{n}| R| \mathbf{n}'\bigr\rangle
={}& \delta^{n_1+n_2}_{n'_1+n'_2} \delta^{n_2+n_3}_{n'_2+n'_3}
\e^{\lambda_1 n'_1+\lambda_2n'_3+\lambda_3 n'_2}
\bigl(\e^{2C_5}q^{n_1+g_3}\bigr)^{n_3+g_3}q^{n'_2+g_2}\nonumber
\\
& \times
\oint\frac{{\rm d}z}{2\pi{\rm i} z^{n'_2+g_2+1}}
\frac{\bigl(-z\e^{-2C_8}q^{2+n'_1+n'_3};q^2\bigr)_\infty\bigl(-z\e^{-2C_7}q^{-n_1-n_3};q^2\bigr)_\infty}
{\bigl(-z\e^{-2C_6}q^{n'_1-n'_3};q^2\bigr)_\infty\bigl(-z\e^{-2C_5}q^{n'_3-n'_1};q^2\bigr)_\infty}\label{co}
\end{align}
in the coordinate representation of the $q$-Weyl algebras where $\uu_i$ is diagonal (see Theorem \ref{th:re}),
and
\begin{gather}
\bigl\langle \mathbf{n}| R| \mathbf{n}'\bigr\rangle
 = q^{\psi_0}
\bigl(-\e^{-2C_7}\bigr)^{\frac{m_1}{2}}
\bigl(\e^{2C_8-2C_3}\bigr)^{\frac{m_2}{2}}
\bigl(\e^{-C_1-C_2-2C_3-C_4}\bigr)^{\frac{m_3}{2}}
\bigl(\e^{C_1-C_2-2C_3-C_4}\bigr)^{\frac{m_4}{2}}\label{mo}
\\
\phantom{\bigl\langle \mathbf{n}| R| \mathbf{n}'\bigr\rangle =}{} \times
\frac{\bigl(\e^{2C_3};q^2\bigr)_{\frac{m_1}{2}}
\bigl(\e^{-2C_2-2C_8};q^2\bigr)_{\frac{m_2}{2}}
\bigl(\e^{2C_1-2C_3+2C_5};q^2\bigr)_{\frac{m_3}{2}}
\bigl(\e^{-2C_1-2C_3+2C_6};q^2\bigr)_{\frac{m_4}{2}}}
{\bigl(\e^{-4C_3+2C_5+2C_6};q^2\bigr)_{\frac{m_3+m_4}{2}}}\nonumber
\end{gather}
in the momentum representation where $\ww_i$ is diagonal (see Theorem \ref{th:ew}).
Here {\samepage \[(g_1,g_2,g_3)=\frac{1}{\hbar}(C_7-C_6,-C_4,C_7-C_5),\]
$m_i$ and $\psi_0$ are linear and quadratic forms of
$\mathbf{n}$ and $\mathbf{n}'$ as given in~\eqref{mdef} and~\eqref{p0def}.}

Let us write \eqref{co} and \eqref{mo} as
\smash{$R^{n_1, n_2, n_3}_{n'_1,n'_2,n'_3}$} and
\smash{$S^{n_1, n_2, n_3}_{n'_1,n'_2,n'_3}$}, respectively. We have also
evaluated the matrix elements in the modular double setting with the
corresponding results \smash{$\mathcal{R}^{n_1, n_2, n_3}_{n'_1,n'_2,n'_3}$}
(see~Theorem~\ref{th:Rm}) and
\smash{$\mathcal{S}^{n_1, n_2, n_3}_{n'_1,n'_2,n'_3}$} (see Theorem \ref{th:ft}).
They are expressed in terms of the non-compact quantum dilogarithm
\eqref{ncq}. When the parameters are specialized appropriately, our
$R$-matrices yield those obtained in \cite{KV94} as the intertwiner of
the quantized coordinate ring of ${\rm SL}_3$ (see also~\cite{BS06,K22}), in
\cite{BMS10} from a quantum geometry consideration, in \cite{KMY23}
from a quantized six-vertex model, and in~\cite{IKT1} from the quantum
cluster algebra associated with the FG quiver. These results are
summarized in Table \ref{tab:R}.

\begin{table}[h]\centering
\begin{tabular}{c|c|c|c|c}
 & $\begin{matrix} \text{relevant}\\ \text{quantum} \\
 \text{dilogarithm} \end{matrix}$ & coordinate rep. & momentum rep.
 &
$\begin{matrix} \text{specialization}\\ \text{adapted} \\
\text{to the FG quiver} \end{matrix}$\\
\hline
$q$-dilog $R$ & $\Psi_q$ &
$\begin{matrix}R^{n_1, n_2, n_3}_{n'_1,n'_2,n'_3},\\ \text{Theorem~\ref{th:re}}\end{matrix}$
&
$\begin{matrix}S^{n_1, n_2, n_3}_{n'_1,n'_2,n'_3},\\ \text{Theorem~\ref{th:ew}}\end{matrix}$
&
\\
\eqref{RL12} & \eqref{Psiq} & \cite{BS06,KV94}, Remark~\ref{re:R} & \cite{KMY23}, Remark~\ref{re:kmy0}
 &
 \cite{IKT1}, Theorem~\ref{th:fg}
\\
\hline
modular $\mathcal{R}$ & $\Phi_\bb$ &
$\begin{matrix}\mathcal{R}^{n_1, n_2, n_3}_{n'_1,n'_2,n'_3},\\ \text{Theorem~\ref{th:Rm}}\end{matrix}$
 &
$\begin{matrix}\mathcal{S}^{n_1, n_2, n_3}_{n'_1,n'_2,n'_3},\\ \text{Theorem~\ref{th:ft}}\end{matrix}$
&
\\
\eqref{RRf} & \eqref{ncq} & \cite{BMS10}, Remark~\ref{re:bms} & \cite{KMY23}, Remark~\ref{re:kon}
 &
\cite{IKT1}, Proposition~\ref{pr:lim}
\end{tabular}
\caption{$R$-matrices in this paper. Relations to those in the
 literature and the relevant remarks or statements are given in the
 second line within each box.}
\label{tab:R}
\end{table}

In \cite{IKT2}, the $R$-matrix in \cite{S99} was reproduced in a
parallel story based on the {\em square} quiver. This solution also
involves four quantum dilogarithms, but it differs from the one in
this paper. In fact, even the special case of our solution mentioned
in Remark \ref{re:bms} is related to \cite{S99} only through a highly
non-trivial transformation called vertex-IRC (interaction round cube)
duality~\cite{S10}. Along with the current results obtained from the
SB quiver, the quantum cluster algebra approach has successfully
captured most of the significant solutions of the tetrahedron equation
known to date for a generic $q$. Additionally, this approach has been
extended to the 3D reflection equations \cite{IK97,K22}, as previously
demonstrated with the FG quiver in \cite{IKT1}. In this paper we
assume that $q$ is generic throughout. We hope to explore the $q$
root-of-unity case elsewhere.

The layout of the paper is as follows. In Section \ref{s:qca}, we
recall basic facts about quantum cluster algebras necessary in this
paper. In Section \ref{s:ct}, we introduce the SB quiver and study the
cluster transformation $\hat{R}$. In Section \ref{s:qw}, we realize
the quantum $Y$-variables by $q$-Weyl algebras and extract $R$ such
that $\hat{R}=\mathrm{Ad}(R)$. The contents of Sections \ref{s:ct}
and \ref{s:qw} are parallel with \cite{IKT2}. The matrix elements of
$R$ are calculated in Sections \ref{s:me} and \ref{s:md}. In Section
\ref{s:6v}, we explain that the $R$-matrix in this paper satisfies the
so-called $RLLL=LLLR$ relation for the $L$-operator which can be
regarded as a quantized six-vertex model \cite{BMS10,KMY23}. It
implies that the matrix elements obey linear recursion relations. In
Section \ref{s:red}, we explain that the $R$-matrix for the FG quiver
previously obtained in \cite{IKT1} arises as a special limit of the
$R$-matrix in this paper. Appendix~\ref{ap:sup} is a supplement to
Section \ref{ss:ms}. Appendix \ref{app:ruw} provides another formula
for $R$ corresponding to a~different choice of signs labeling the
decomposition of mutations into monomial and automorphism parts.
Appendix~\ref{ap:nc} contains integral formulas for non-compact
quantum dilogarithm. Appendix~\ref{app:rlll} is a list of explicit
forms of the $RLLL=LLLR$ relations.

\section{Quantum cluster algebra}\label{s:qca}

\subsection{Mutation}\label{ss:mutation}

Let us recall the definition of quantum cluster mutation following \cite{FG06}.
For a finite set $I$, set~${B = (b_{ij})_{i,j \in I}}$ with
$b_{ij} = -b_{ji} \in \Z/2$. We call $B$ the {\it exchange matrix}.
In this article we will only encounter skew-symmetric exchange matrices
with $b_{ij} \in \bigl\{\pm 1, \pm \frac{1}{2},0\bigr\}$. An exchange matrix
will be depicted as a {\em quiver}. It is an oriented graph
with vertices labeled with the elements of $I$ and a solid arrow
(resp.\ dotted arrow) from $i$ to $j$ when $b_{ij}=1$ \big(resp.\
$b_{ij}=\frac{1}{2}$\big).

Let $\mathcal{Y}(B)$ be a skew field generated by $q$-commuting variables
$\bY = (Y_i)_{i \in I}$ under the relations
\begin{align}\label{yyq}
 Y_i Y_j = q^{2b_{ij}} Y_j Y_i.
\end{align}
The data $(B, \bY)$ will be called a quantum $y$-seed and $Y_i$ a (quantum) $Y$-variable.
We assume that the parameter $q$ is generic throughout.
For $(B, \bY)$ and for $k \in I$ such that
$b_{ki} \neq \pm \frac{1}{2}$, the mutation $\mu_k$ transforms
$(B, \bY)$ to $\bigl(B', \bY'\bigr) := \mu_k (B,\bY)$, where
\begin{align}\label{muB}
 &b_{ij}' =
 \begin{cases}
 -b_{ij}, & i=k \text{ or } j=k,
 \\
 \displaystyle{b_{ij} + \frac{|b_{ik}| b_{kj} + b_{ik} |b_{kj}|}{2}},
 & \text{otherwise},
 \end{cases}
 \\ \label{muY}
 &Y_{i}' =
 \begin{cases}
 Y_k^{-1}, & i=k,
 \\
 \displaystyle{Y_i \prod_{j=1}^{|b_{ik}|}\bigl(1 + q^{2j-1} Y_k^{-\mathrm{sgn}(b_{ik})}\bigr)^{-\mathrm{sgn}(b_{ik})}}, & i \neq k.
 \end{cases}
\end{align}
The mutations are involutive, $\mu_k \mu_k = \mathrm{id.}$, and
commutative, $\mu_k \mu_j = \mu_j \mu_k$ if $b_{jk}=b_{kj}=0$. The
mutation $\mu_k$ induces an isomorphism of skew fields
$\mu_k^{\ast}\colon \mathcal{Y}\bigl(B'\bigr) \to \mathcal{Y}(B)$, where
$\mathcal{Y}\bigl(B'\bigr)$ is a skew field generated by the variables
$Y'=\bigl(Y'_i\bigr)_{i\in I}$ under the relations
\smash{$Y'_i Y'_j = q^{2b'_{ij}} Y'_j Y'_i$}.

The map $\mu_k^{\ast}$ is decomposed into two parts, a monomial part
and an automorphism part~\cite{FG09}, in two ways \cite{Ke11}. To explain it,
let us introduce an isomorphism $\tau_{k,\varepsilon}$ of the
skew fields for~${\varepsilon \in \{+,-\}}$ by
\begin{align}\label{tauy}
\tau_{k,\ve} \colon\ \mathcal{Y}\bigl(B'\bigr) \to \mathcal{Y}(B)
; \qquad Y'_i \mapsto
\begin{cases}
 Y_k^{-1}, & i= k,
 \\
 q^{-b_{ik}[\ve b_{ik}]_+}Y_i Y_k^{[\ve b_{ik}]_+}, & i \neq k,
\end{cases}
\end{align}
where $[a]_+ := \max[0,a]$.
The adjoint action $\mathrm{Ad}_{k,\ve}$ on $\mathcal{Y}(B)$ is defined by
$\mathrm{Ad}_{k,+} := \mathrm{Ad}(\Psi_{q}(Y_k)) $, \smash{$
\mathrm{Ad}_{k,-} := \mathrm{Ad}\bigl(\Psi_{q}\bigl(Y_k^{-1}\bigr)^{-1}\bigr)$},
where $\mathrm{Ad}(Y)(X) = YXY^{-1}$.
The symbol $\Psi_q(Y)$ appearing here denotes the quantum dilogarithm
\begin{align}\label{Psiq}
\Psi_q(Y) = \frac{1}{\bigl(-qY; q^2\bigr)_\infty}, \qquad (z;q)_\infty = \prod_{n=0}^\infty (1-zq^n).
\end{align}
One has the expansions
\begin{align}\label{expa}
\Psi_q(Y) = \sum_{n = 0}^\infty \frac{(-qY)^n}{\bigl(q^2;q^2\bigr)_n},
\qquad
\Psi_q(Y)^{-1} = \sum_{n = 0}^\infty \frac{q^{n^2} Y^n}{\bigl(q^2;q^2\bigr)_n},
\end{align}
where $\bigl(z;q^2\bigr)_n = \bigl(z;q^2\bigr)_\infty/\bigl(zq^{2n};q^2\bigr)_\infty$ for any $n$.
Basic properties of the quantum dilogarithm are{\samepage
\begin{align}
\label{Prec}
&\Psi_q\bigl(q^2 U\bigr) \Psi_q(U)^{-1} = 1+qU,
\\
&\Psi_q(U)\Psi_q(W) = \Psi_q(W) \Psi_q\bigl(q^{-1}UW\bigr) \Psi_q(U) \qquad \text{if}\quad UW = q^2WU,\nonumber
\end{align}
where the second one is called the pentagon identity.}

Now the decomposition of $\mu^\ast_k$ in two ways mentioned in the above
is given as
\begin{align}\label{mud}
\mu_k^{\ast} = \mathrm{Ad}_{k,+} \circ \tau_{k,+}
= \mathrm{Ad}_{k,-} \circ \tau_{k,-}.
\end{align}
Namely, one has the following diagram for both choices $\ve=+, -$:
\begin{align*}
\xymatrix{
\mathcal{Y}\bigl(B'\bigr) \ar[r]^{\mu_k^\ast} \ar[dr]_{\tau_{k,\ve}} & \mathcal{Y}(B)
\\
& \mathcal{Y}(B) \ar[u]_{\mathrm{Ad}_{k,\ve}}.
}
\end{align*}

\begin{Example}
 Let $I = \{1,2\}$ and the 2-by-2 exchange matrix be given by
 $B = \left(\begin{smallmatrix}0 & 1 \\ -1 & 0 \end{smallmatrix}\right)$, which implies
 $Y_1Y_2 = q^2Y_2Y_1$. Consider the mutation
 $\mu_2(B,\bY) = \bigl(B', \bY'\bigr)$, where $\bY = (Y_1, Y_2)$ and~${\bY' = (Y'_1, Y'_2)}$. Then $B'=-B$ from \eqref{muB} and
 \smash{$Y'_1 = Y_1\bigl(1+qY^{-1}_2\bigr)^{-1}$} from \eqref{muY}. On the other hand,
 the same result is obtained also in the form
 \smash{$Y'_1 \rightarrow Y_1\bigl(1+qY^{-1}_2\bigr)^{-1}$} in two ways according to
 \eqref{mud} as follows:
\begin{gather*}
Y'_1 \overset{\tau_{2,+}}{\longrightarrow} q^{-1}Y_1Y_2
\overset{\mathrm{Ad}_{2,+}}{\longrightarrow} q^{-1}\Psi_q(Y_2)Y_1Y_2\Psi_q(Y_2)^{-1}\\
\qquad= q^{-1}Y_1\Psi_q\bigl(q^{-2}Y_2\bigr)\Psi_q(Y_2)^{-1}Y_2
= q^{-1}Y_1\bigl(1+q^{-1}Y_2\bigr)^{-1}Y_2,
\\
Y'_1 \overset{\tau_{2,-}}{\longrightarrow} Y_1
\overset{\mathrm{Ad}_{2,-}}{\longrightarrow} \Psi_q\bigl(Y_2^{-1}\bigr)^{-1}Y_1 \Psi_q\bigl(Y_2^{-1}\bigr)
= Y_1\Psi_q\bigl(q^2Y_2^{-1}\bigr)^{-1}\Psi_q\bigl(Y_2^{-1}\bigr) =Y_1\bigl(1+qY^{-1}_2\bigr)^{-1}.
\end{gather*}
\end{Example}

For later use, we introduce the quantum torus algebra $\mathcal{T}(B)$ associated to $B$.
It is the $\mathbb{Q}(q)$-algebra generated by
non-commutative variables $\rY^\alpha$ $\bigl(\alpha \in \Z^{I}\bigr)$ satisfying the
relations
\begin{align}\label{qyy}
q^{\langle \alpha,\beta \rangle} \rY^\alpha \rY^\beta
= \rY^{\alpha + \beta},
\end{align}
where $\langle ~,~ \rangle$ is a skew-symmetric form defined by
$\langle \alpha,\beta \rangle = - \langle \beta,\alpha \rangle =
-\alpha \cdot B \beta$.
Let $e_i$ be the standard unit vector of
$\Z^I$. We write $\rY^{e_i}$ simply as $\rY_i$. Then
$\rY_i \rY_j = q^{2 b_{ij}} \rY_j \rY_i$ holds. We identify $\rY_i$ with
$Y_i$, which is consistent with \eqref{yyq}.

Let $\mathcal{FT}(B)$ be the fractional field of $\mathcal{T}(B)$.
The mutations $\mu_k^{\ast}$
and their decompositions
induce the morphisms for the fractional fields of the quantum torus algebras naturally.
In particular, the monomial part \eqref{tauy} of $\mu_k^{\ast}$ is written as
\begin{align}\label{taft}
\tau_{k,\ve} \colon\ \mathcal{FT}\bigl(B'\bigr) \to \mathcal{FT}(B); \qquad
\rY'_i \mapsto \begin{cases}
\rY_k^{-1}, & i= k,
\\
\rY^{e_i + e_k [\ve b_{ik}]_+}, & i \neq k
\end{cases}
\end{align}
under the identification $\rY_i = Y_i$, $\rY'_i = Y'_i$.
Hence $\mathcal{FT}(B)$ (resp. $\mathcal{FT}\bigl(B'\bigr)$) is identified with $\mathcal{Y}(B)$ (resp. $\mathcal{Y}\bigl(B'\bigr)$).

\subsection[Tropical y-variables and tropical sign]{Tropical $\boldsymbol{y}$-variables and tropical sign}

Let $\mathbb{P}(u)=\mathbb{P}_\trop(u_1,u_2,\dotsc,u_p) := \bigl\{\prod_{i=1}^{p} u_i^{a_i};\, a_i \in \Z \bigr\}$
be the tropical semifield of rank $p$, endowed with the addition $\oplus$ and multiplication $\cdot$ defined by
\[
 \prod_{i=1}^{p} u_i^{a_i} \oplus \prod_{i=1}^{p} u_i^{b_i}
 =
 \prod_{i=1}^{p} u_i^{\min(a_i,b_i)},
 \qquad
 \prod_{i=1}^{p} u_i^{a_i} \cdot \prod_{i=1}^{p} u_i^{b_i}
 =
 \prod_{i=1}^{p} u_i^{a_i+b_i}.
\]
For $s = \prod_{i \in I} u_i^{a_i} \in \mathbb{P}(u)$, we write $s = u^\alpha$
with $\alpha = (a_i)_{i \in I} \in \Z^{I}$.
We say that $s$ is positive if~${\alpha \in (\Z_{\geq 0})^{I}}$ and negative if
$\alpha \in (\Z_{\leq 0})^{I}$.

For a quiver $Q$ whose vertex set is $I$, let $\mathbb{P}(u)$ be a tropical semifield of rank $|I|$.
The data of the form $(B,y)$ with $B$ being the exchange matrix of $Q$ and
$y = (y_i)_{i \in I} \in \mathbb{P}(u)^{I}$ is called a~tropical $y$-seed.
For $k \in I$, the mutation\footnote{For simplicity, we use the same symbol $\mu_k$ to denote a mutation for
quantum $y$-seeds $(B,\bY)$ and tropical $y$-seeds $(B,y)$.}
 $\mu_k (B, y) =: \bigl(B', y'\bigr)$ is defined by
\eqref{muB} and
\begin{equation}\label{tmu}
y_{i}' =
 \begin{cases}
 y_k^{-1}, & i=k,
 \\
 y_i \bigl(1 \oplus y_k^{-\mathrm{sgn}(b_{ik})}\bigr)^{-b_{ik}}, & i \neq k.
 \end{cases}
\end{equation}
For a tropical $y$-variable $y_i' = u^{\alpha'}$, the vector $\alpha' \in \Z^I$
is called the {\it $c$-vector} of $y_i'$.
The following theorem states the {\em sign coherence} of the $c$-vectors.

\begin{Theorem}[\cite{FZ07,GHKK14}]
\label{thm:sign-coherence}
Let $\bigl(B',y'\bigr) = \mu_{i_L} \dotsm \mu_{i_2} \mu_{i_1}(B,u)$ be a
tropical $y$-seed with $y'=(y'_i)_{i \in I}$. For any sequence
$(i_1,\dotsc, i_L) \in I^L$, each $y'_i \in \mathbb{P}(u)$ is either
positive or negative.
\end{Theorem}

Based on Theorem \ref{thm:sign-coherence}, for any tropical $y$-seed $\bigl(B',y'\bigr)$ with
$y'=(y'_i)_{i \in I}$ obtained from $(B,u)$ by applying mutations, we
define the {\it tropical sign} $\ve_i'$ of $y_i'$ to be $+1$ (resp.\
$-1$) if $y_i'$ is positive (resp.\ $y_i$ is negative). We also write
$\ve_i'= \pm$ for $\ve_i'=\pm 1$ for simplicity.

\begin{Remark}\label{re:sgni}
For the mutation $\mu_k (B,y) = \bigl(B',y'\bigr)$ of a tropical $y$-seed,
let $c_i$, $c_i'$, $c_k$ be the $c$-vectors of $y_i$, $y_i'$, $y_k$, respectively,
and let $\ve_k$ be the tropical sign of $y_k$.
Then the tropical mutation \eqref{tmu} is expressed in terms of $c$-vectors as
\begin{align*}
 c_i' =
 \begin{cases}
 -c_k, & i=k,
 \\
 c_i + c_k [\ve_k b_{ik}]_+, & i \neq k.
\end{cases}
\end{align*}
This coincides with the transformation of quantum torus \eqref{taft}
on $\Z^I$ (i.e., the power of \eqref{taft}) when $\ve = \ve_k$.
\end{Remark}

\subsection{Sequence of mutations}

Let us describe the quantum $Y$-variables associated with the sequence
of mutations $\mu_{i_l} \mu_{i_{l-1}} \dots \allowbreak\mu_{i_2} \mu_{i_1}$:
 \begin{align}\label{bys}
\bigl(B^{(1)},\bY^{(1)}\bigr) \stackrel{\mu_{i_1}}{\longleftrightarrow}
\bigl(B^{(2)},\bY^{(2)}\bigr) \stackrel{\mu_{i_2}}{\longleftrightarrow}
\cdots
\stackrel{\mu_{i_l}}{\longleftrightarrow} \bigl(B^{(l+1)},\bY^{(l+1)}\bigr).
\end{align}
For $t=1, \dots, l+1$, let $\rY^\alpha(t)$ $\bigl(\alpha \in \Z^I\bigr)$ be the
generators of the quantum torus $\mathcal{T}\bigl(B^{(t)}\bigr)$ in the sense
explained around \eqref{qyy}. We set $\rY_i(t) =
\rY^{e_i}(t)$.
Especially for $t=1$, we use the simpler notations
$\rY^\alpha=\rY^\alpha(1)$ and $\rY_i=\rY_i(1)$. As in \eqref{taft},
we identify $\rY_i$ with \smash{$Y_i = Y^{(1)}_i$}, hence
\smash{$\mathcal{Y}\bigl(B^{(1)}\bigr)$} with \smash{$\mathcal{FT}\bigl(B^{(1)}\bigr)$}.
Then the quantum $Y$-variables \smash{$Y^{(t+1)}=\bigl(Y^{(t+1)}_i\bigr)_{i\in I}$}
($t=0, \dots, l$) appearing in~\eqref{bys} are expressed as
\begin{align}
Y_i^{(t+1)}
&= \Ad\bigl(\Psi_{q}\bigl(\rY_{i_1}(1)^{\delta_1}\bigr)^{\delta_1}\bigr) \tau_{i_1,\delta_1}\dotsm \Ad\bigl(\Psi_{q}\bigl(\rY_{i_{t}}(t)^{\delta_{t}}\bigr)^{\delta_{t}}\bigr) \tau_{i_{t},\delta_{t}}(\rY_i(t+1))\nonumber
\\
&=\Ad\bigl(\Psi_{q}\bigl(\rY^{\delta_1 \beta_1}\bigr)^{\delta_1} \dotsm \Psi_{q}\bigl(\rY^{\delta_{t} \beta_{t}}\bigr)^{\delta_{t}}\bigr) \tau_{i_1,\delta_1}\dotsm \tau_{i_{t},\delta_{t}}(\rY_i(t+1)).\label{yad}
\end{align}
This formula is valid for any choice of the signs $\delta_1, \dots,
\delta_l \in \{+, -\}$, on which the left-hand side is independent. Note that
\smash{$Y^{(t+1)}_i$} is in general a ``complicated'' element in
$\mathcal{Y}\bigl(B^{(1)}\bigr)$ generated from \smash{$\bigl(B^{(1)},Y^{(1)}\bigr)$} by applying
$\mu_{i_t}\dotsm \mu_{i_2}\mu_{i_1}$ according to \eqref{muY}.
On the other hand,
$\rY_i(t+1)$ is just a basis of \smash{$\mathcal{T}\bigl(B^{(t+1)}\bigr)$}. The first
line of \eqref{yad} says that \smash{$Y^{(t+1)}_i$} is also obtained as the
image of $\rY_i(t+1)$ under the composition
\smash{$\mu_{i_1}^\ast \dotsm \mu^\ast_{i_{t-1}} \mu^\ast_{i_t}$} which is an
isomorphism
$\mathcal{FT}\bigl(B^{(t+1)}\bigr) \rightarrow \mathcal{FT}\bigl(B^{(1)}\bigr) =
\mathcal{Y}\bigl(B^{(1)}\bigr)$. The second line is derived from the first line
by pushing $\tau_{i,\delta}$'s to the right. Thus we have $\beta_1 = e_{i_1}$,
and in general $\beta_r \in \Z^I$ is determined by
$\rY^{\beta_r} = \tau_{i_1,\delta_1} \dotsm \tau_{i_{r-1},
 \delta_{r-1}}(\rY_{i_r}(r))$.

\subsection{A useful theorem}

Let $\sigma_{r,s} \in \mathfrak{S}_I$ ($r, s \in I$) be a transposition.
We let it act on either classical $y$-seeds $(B,y)$ or quantum $y$-seeds $(B,Y)$
as the exchange of the indices $r$ and $s$. For quantum $y$-seeds, it
is given~by
\begin{equation}\label{qys}
((b_{ij})_{i,j\in I} , (Y_i)_{i \in I} ) \mapsto
((b_{\sigma_{r,s}(i), \sigma_{r,s}(j)})_{i,j\in I}, (Y_{\sigma_{r,s}(i)})_{i \in I} ),
\end{equation}
where $\sigma_{r,s}(r) = s$, $\sigma_{r,s}(s)=r$ and
$\sigma_{r,s}(i) = i$ for $i \neq r,s$. For classical $y$-seeds, the
rule is similar.

Let
\begin{equation}\label{nuL}
\nu = \nu_L\dotsm \nu_1 := \sigma_{r_m,s_m} \dotsm
\mu_{i_l} \dotsm \sigma_{r_1,s_1} \dotsm \mu_{i_1}, \qquad L = l+m,
\end{equation}
be a composition of $l$ mutations $\mu_{i_1}, \dots, \mu_{i_l}$ and
$m$ transpositions $\sigma_{r_1,s_1}, \dots, \sigma_{r_m, s_m}$ in
an arbitrary order. (So $\nu_L$ may actually be a mutation for
example.) For simplicity, we also call $\nu$ a~mutation sequence even
though a part of it may involve transpositions.

Consider the tropical $y$-seeds starting from $(B,y)$ and the quantum
$y$-seeds starting from $(B,Y)$ which are generated along the mutation
sequences $\nu=\nu_L\dotsm \nu_1$ and $\nu'= \nu'_{L'} \dotsm \nu'_1$
as follows:
\begin{align}
\label{seq0}
&(B,y) =:\bigl(B^{(1)},y^{(1)}\bigr) \stackrel{\nu_1}{\longleftrightarrow}
\bigl(B^{(2)},y^{(2)}\bigr) \stackrel{\nu_2}{\longleftrightarrow}
\dotsb
\stackrel{\nu_L}{\longleftrightarrow} \bigl(B^{(L+1)},y^{(L+1)}\bigr) = \nu(B,y),
\\
&(B,\bY) =:\bigl(B^{(1)},\bY^{(1)}\bigr) \stackrel{\nu_1}{\longleftrightarrow}
\bigl(B^{(2)},\bY^{(2)}\bigr) \stackrel{\nu_2}{\longleftrightarrow}
\dotsb
\stackrel{\nu_L}{\longleftrightarrow} \bigl(B^{(L+1)},\bY^{(L+1)}\bigr) = \nu(B,Y),
\\
&(B,y) =:\bigl(B^{(1)\prime },y^{(1)\prime}\bigr) \stackrel{\nu'_1}{\longleftrightarrow}
\bigl(B^{(2)\prime},y^{(2)\prime}\bigr) \stackrel{\nu'_2}{\longleftrightarrow}
\dotsb
\stackrel{\nu'_{L'}}{\longleftrightarrow} \bigl(B^{(L+1)\prime},y^{(L+1)\prime}\bigr) = \nu'(B,y),
\\
&(B,\bY) =:\bigl(B^{(1)\prime},\bY^{(1)\prime}\bigr) \!\stackrel{\nu'_1}{\longleftrightarrow}\!
\bigl(B^{(2)\prime},\bY^{(2)\prime}\bigr) \!\stackrel{\nu'_2}{\longleftrightarrow}
\dotsb
\stackrel{\nu'_{L'}}{\longleftrightarrow} \!\bigl(B^{(L+1)\prime},\bY^{(L+1)\prime}\bigr) = \nu'(B,Y).\!\!\!
\label{seq3}
\end{align}
The following theorem is established by combining the synchronicity \cite{N21}
among $x$-seeds, $y$-seeds and tropical $y$-seeds,
and the synchronicity between classical and quantum seeds
\cite[Lemma 2.22]{FG09b}, \cite[Proposition~3.4]{KN11}.

\begin{Theorem}\label{thm:piv}
 In the situation in \eqref{seq0}--\eqref{seq3}, the following two
 statements are equivalent:
\begin{itemize}\itemsep=0pt

\item[$(1)$]
The tropical $y$-seeds satisfy $\nu (B,y) = \nu'(B,y)$.

\item[$(2)$]
The quantum $y$-seeds satisfy $\nu (B,\bY) = \nu'(B,\bY)$.

\end{itemize}
\end{Theorem}

It is remarkable that (2) follows from (1) which is much simpler to check.
We will utilize this fact efficiently
in the subsequent arguments.

\section[Cluster transformation R]{Cluster transformation $\boldsymbol{\widehat{R}}$}\label{s:ct}

\subsection{Wiring diagram and symmetric butterfly quiver}

Let us fix our convention of the wiring diagrams and associated square quivers
using examples.
See also \cite[Section 3]{SY22}.
Let $W(A_n)$ be the Weyl group of $A_n$ generated by the simple
reflections~${s_1,\dotsc, s_n}$ obeying the Coxeter relations $s_i^2=1$,
$s_is_js_i = s_js_is_j$ ($|i-j|=1$) and~${s_is_j=s_js_i}$~($|i-j|\ge 2$). A reduced expression
$s_{i_1}\dotsm s_{i_l}$ of an element in $W(A_n)$ is identified with
the (reduced) word $i_1\ldots i_l \in [1,n]^l$. A wiring diagram is a
collection of $n$ wires which are horizontal except the vicinity of
crossings. In the aforementioned context, $i_k$ indicates that the
$k$-th crossing from the left takes place at the $i_k$-th level,
measured from the top. Crossings are required to occur at distinct
horizontal positions, although this restriction can be relaxed due to
the identification of topologically equivalent diagrams which are
transformable by $s_is_j=s_js_i$ ($|i-j|\ge 2$).
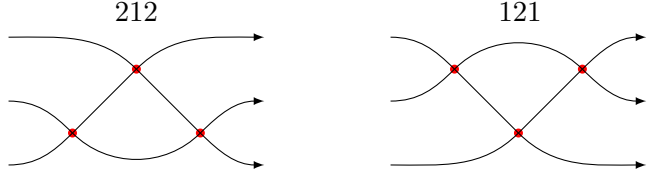
\begin{figure}[t]
\centering
\begin{tikzpicture}[scale=0.85]
\begin{scope}[>=latex,xshift=0pt]
\draw (2,2.1) circle(0pt) node[above]{212};
\draw (8,2.1) circle(0pt) node[above]{121};
{\color{red}
\fill (1,0.5) circle(2pt) coordinate(A) node[below]{};
\fill (2,1.5) circle(2pt) coordinate(B) node[above]{};
\fill (3,0.5) circle(2pt) coordinate(C) node[below]{};
}
\draw [-] (0,2) to [out = 0, in = 135] (B);
\draw [-] (B) -- (C);
\draw [->] (C) to [out = -45, in = 180] (4,0);
\draw [-] (0,1) to [out = 0, in = 135] (A);
\draw [-] (A) to [out = -45, in = -135] (C);
\draw [->] (C) to [out = 45, in = 180] (4,1);
\draw [-] (0,0) to [out = 0, in = -135] (A);
\draw [-] (A) -- (B);
\draw [->] (B) to [out = 45, in = 180] (4,2);
\coordinate (P1) at (4.5,1);
\coordinate (P2) at (5.5,1);
\end{scope}
\begin{scope}[>=latex,xshift=170pt]
{\color{red}
\fill (3,1.5) circle(2pt) coordinate(A) node[above]{};
\fill (2,0.5) circle(2pt) coordinate(B) node[below]{};
\fill (1,1.5) circle(2pt) coordinate(C) node[above]{};
}
\draw [-] (0,0) to [out = 0, in = -135] (B);
\draw [-] (B) -- (A);
\draw [->] (A) to [out = 45, in = 180] (4,2);
\draw [-] (0,1) to [out = 0, in = -135] (C);
\draw [-] (C) to [out = 45, in = 135] (A);
\draw [->] (A) to [out = -45, in = 180] (4,1);
\draw [-] (0,2) to [out = 0, in = 135] (C);
\draw [-] (C) -- (B);
\draw [->] (B) to [out = -45, in = 180] (4,0);
\end{scope}
\end{tikzpicture}

\caption{Wiring diagrams for the reduced words 212 and 121 of the
 longest element $s_2s_1s_2=s_1s_2s_1$ of $W(A_2)$.}
\end{figure}

Given a wiring diagram, the associated symmetric butterfly quiver has
the vertices in the domains and on the crossings of it. The vertices
are interconnected by elementary triangles which are oriented with
dotted arrows. A pair of dotted arrows pointing in the same (resp.\
opposite) direction are regarded as a solid arrow (resp.~none). We
choose the convention that quiver vertices on the crossings of the
wiring diagram become sources vertically and sinks horizontally.

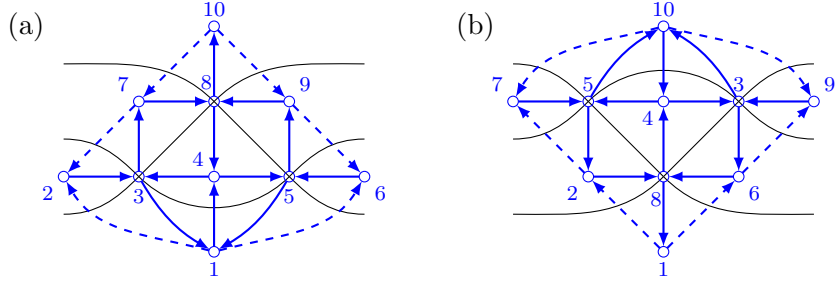
\begin{figure}[t]\centering
\begin{tikzpicture}
\begin{scope}[>=latex,xshift=0pt]
\draw (-0.5,2.5) node{(a)};
\coordinate (A) at (1,0.5);
\coordinate (B) at (2,1.5);
\coordinate (C) at (3,0.5);
%
\draw [-] (0,2) to [out = 0, in = 135] (B);
\draw [-] (B) -- (C);
\draw [-] (C) to [out = -45, in = 180] (4,0);
\draw [-] (0,1) to [out = 0, in = 135] (A);
\draw [-] (A) to [out = -45, in = -135] (C);
\draw [-] (C) to [out = 45, in = 180] (4,1);
\draw [-] (0,0) to [out = 0, in = -135] (A);
\draw [-] (A) -- (B);
\draw [-] (B) to [out = 45, in = 180] (4,2);
%
%
%
{\color{blue}
\draw (2,-0.5) circle(2pt) coordinate(1) node[below]{\scriptsize$1$};
\draw (0,0.5) circle(2pt) coordinate(2) node[below left]{\scriptsize$2$};
\draw (1,0.5) circle(2pt) coordinate(3) node[below=1pt]{\scriptsize$3$};
\draw (2,0.5) circle(2pt) coordinate(4) node[above left]{\scriptsize$4$};
\draw (3,0.5) circle(2pt) coordinate(5) node[below=1pt]{\scriptsize$5$};
\draw (4,0.5) circle(2pt) coordinate(6) node[below right]{\scriptsize$6$};
\draw (1,1.5) circle(2pt) coordinate(7) node[above left]{\scriptsize$7$};
\draw (2,1.5) circle(2pt) coordinate(8);
\draw (1.9,1.75) node {\scriptsize$8$};
\draw (3,1.5) circle(2pt) coordinate(9) node[above right]{\scriptsize$9$};
\draw (2,2.5) circle(2pt) coordinate(10) node[above]{\scriptsize$10$};
\draw[->,dashed, shorten >=2pt,shorten <=2pt] (1) to [out = 165, in = -60] (2) [thick];
\draw[->,shorten >=2pt,shorten <=2pt] (3) to [out = -60, in = 150] (1) [thick];
\qarrow{1}{4}
\draw[->,shorten >=2pt,shorten <=2pt] (5) to [out = -120, in = 30] (1) [thick];
\draw[->,dashed, shorten >=2pt,shorten <=2pt] (1) to [out = 15, in = -120] (6) [thick];
\qarrow{6}{5}
\qarrow{4}{5}
\qarrow{4}{3}
\qarrow{2}{3}
\qdarrow{7}{2}
\qarrow{3}{7}
\qarrow{8}{4}
\qarrow{5}{9}
\qdarrow{9}{6}
\qarrow{9}{8}
\qarrow{7}{8}
\qdarrow{10}{7}
\qarrow{8}{10}
\qdarrow{10}{9}
}
\end{scope}
\begin{scope}[>=latex,xshift=170pt]
\draw (-0.5,2.5) node{(b)};
\coordinate (A) at (3,1.5);
\coordinate (B) at (2,0.5);
\coordinate (C) at (1,1.5);
%
\draw [-] (0,0) to [out = 0, in = -135] (B);
\draw [-] (B) -- (A);
\draw [-] (A) to [out = 45, in = 180] (4,2);
\draw [-] (0,1) to [out = 0, in = -135] (C);
\draw [-] (C) to [out = 45, in = 135] (A);
\draw [-] (A) to [out = -45, in = 180] (4,1);
\draw [-] (0,2) to [out = 0, in = 135] (C);
\draw [-] (C) -- (B);
\draw [-] (B) to [out = -45, in = 180] (4,0);
{\color{blue}
\draw (2,-0.5) circle(2pt) coordinate(1) node[below]{\scriptsize$1$};
\draw (1,0.5) circle(2pt) coordinate(2) node[below left]{\scriptsize$2$};
\draw (3,1.5) circle(2pt) coordinate(3) node[above]{\scriptsize$3$};
\draw (2,1.5) circle(2pt) coordinate(4) node[below left]{\scriptsize$4$};
\draw (1,1.5) circle(2pt) coordinate(5) node[above]{\scriptsize$5$};
\draw (3,0.5) circle(2pt) coordinate(6) node[below right]{\scriptsize$6$};
\draw (0,1.5) circle(2pt) coordinate(7) node[above left]{\scriptsize$7$};
\draw (2,0.5) circle(2pt) coordinate(8);
\draw (1.9,0.2) node {\scriptsize$8$};
\draw (4,1.5) circle(2pt) coordinate(9) node[above right]{\scriptsize$9$};
\draw (2,2.5) circle(2pt) coordinate(10) node[above]{\scriptsize$10$};
\draw[->,dashed, shorten >=2pt,shorten <=2pt] (10) to [out = -165, in = 60] (7) [thick];
\draw[->,shorten >=2pt,shorten <=2pt] (5) to [out = 60, in = -150] (10) [thick];\qarrow{10}{4}
\draw[->,shorten >=2pt,shorten <=2pt] (3) to [out = 120, in = -30] (10) [thick];
\draw[->,dashed, shorten >=2pt,shorten <=2pt] (10) to [out = -15, in = 120] (9) [thick];
\qarrow{7}{5}
\qarrow{4}{5}
\qarrow{4}{3}
\qarrow{9}{3}
\qdarrow{2}{7}
\qarrow{8}{4}
\qarrow{5}{2}
\qarrow{3}{6}
\qdarrow{6}{9}
\qarrow{2}{8}
\qarrow{6}{8}
\qdarrow{1}{2}
\qarrow{8}{1}
\qdarrow{1}{6}
}
\end{scope}
\end{tikzpicture}

\vspace{-3mm}
\caption{Symmetric butterfly quivers (depicted in blue) associated
 with the wiring diagrams. Given the labels $1,\dotsc, 10$ of the
 quiver vertices in (a), those in (b) are determined following the
 mutation sequence in Figure \ref{fig:mus}.}\label{fig:3.2}\vspace{-2mm}
\end{figure}

\begin{Remark}\label{re:c1}
 Let $B$ be the exchange matrix corresponding to the symmetric
 butterfly quiver in Figure \ref{fig:3.2}\,(a). Then the skew filed
 $\mathcal{Y}(B)$ generated by $Y_1, \dots, Y_{10}$ has the center
 generated by
\[
Y^{-1}_1Y_7Y_8Y_9Y^2_{10},\qquad
Y_2Y^{-1}_4Y_6Y_{10},\qquad
Y_3Y^2_4Y^{-2}_6Y^2_7Y_8,\qquad
Y_5Y^2_6Y_8Y^2_9Y^2_{10}.
\]

\end{Remark}

\subsection[Cluster transformation R]{Cluster transformation $\boldsymbol{\widehat{R}}$}

Let $\bigl(B^{(1)}, \bY^{(1)}\bigr) = (B, \bY)$ and
$\bigl(B^{(6)}, \bY^{(6)}\bigr) = \bigl(B', \bY'\bigr)$ be the quantum $y$-seeds corresponding to Figure \ref{fig:3.2}\,(a) and (b), respectively.
We connect them by the following mutation sequence
\begin{align}
\bigl(B^{(1)}, \bY^{(1)}\bigr) &\underset{\varepsilon_1}{\overset{\mu_4}{\longleftrightarrow}}
\bigl(B^{(2)}, \bY^{(2)}\bigr) \underset{\varepsilon_2}{\overset{\mu_3}{\longleftrightarrow}}
\bigl(B^{(3)}, \bY^{(3)}\bigr)\nonumber
\\
&\underset{\varepsilon_3}{\overset{\mu_5}{\longleftrightarrow}}
\bigl(B^{(4)}, \bY^{(4)}\bigr) \underset{\varepsilon_4}{\overset{\mu_8}{\longleftrightarrow}}
\bigl(B^{(5)}, \bY^{(5)}\bigr) \overset{\sigma_{3,5}\sigma_{4,8}}{\longleftrightarrow}
\bigl(B^{(6)}, \bY^{(6)}\bigr),\label{mus}
\end{align}
where \smash{$\bY^{(t)}=\bigl(Y^{(t)}_1,\dotsc, Y^{(t)}_{10}\bigr)$}. The symbol
$\sigma_{ij}$ denotes the exchange of the indices $i$ and $j$ in the
exchange matrix and $Y$-variables. See \eqref{qys}. We have also
attached the signs $\varepsilon_i=\pm 1$ along which the decomposition
\eqref{mud} into the automorphism part and the monomial part will be
considered. See Figure \ref{fig:mus}.

\begin{figure}[t]\centering
\begin{tikzpicture}
\begin{scope}[>=latex,xshift=0pt]
\draw (-0.5,2.5) node{$B^{(1)}=B$};
{\color{blue}
\draw (2,-0.5) circle(2pt) coordinate(1) node[below]{\scriptsize$1$};
\draw (0,0.5) circle(2pt) coordinate(2) node[below left]{\scriptsize$2$};
\draw (1,0.5) circle(2pt) coordinate(3) node[below=1pt]{\scriptsize$3$};
\draw (2,0.5) circle(2pt) coordinate(4) node[above left]{\scriptsize$4$};
\draw (3,0.5) circle(2pt) coordinate(5) node[below=1pt]{\scriptsize$5$};
\draw (4,0.5) circle(2pt) coordinate(6) node[below right]{\scriptsize$6$};
\draw (1,1.5) circle(2pt) coordinate(7) node[above left]{\scriptsize$7$};
\draw (2,1.5) circle(2pt) coordinate(8) node[above left]{\scriptsize$8$};
\draw (3,1.5) circle(2pt) coordinate(9) node[above right]{\scriptsize$9$};
\draw (2,2.5) circle(2pt) coordinate(10) node[above]{\scriptsize$10$};
\draw[->,dashed, shorten >=2pt,shorten <=2pt] (1) to [out = 165, in = -60] (2) [thick];
\draw[->,shorten >=2pt,shorten <=2pt] (3) to [out = -60, in = 150] (1) [thick];
\qarrow{1}{4}
\draw[->,shorten >=2pt,shorten <=2pt] (5) to [out = -120, in = 30] (1) [thick];
\draw[->,dashed, shorten >=2pt,shorten <=2pt] (1) to [out = 15, in = -120] (6) [thick];
\qarrow{6}{5}
\qarrow{4}{5}
\qarrow{4}{3}
\qarrow{2}{3}
\qdarrow{7}{2}
\qarrow{3}{7}
\qarrow{8}{4}
\qarrow{5}{9}
\qdarrow{9}{6}
\qarrow{9}{8}
\qarrow{7}{8}
\qdarrow{10}{7}
\qarrow{8}{10}
\qdarrow{10}{9}
}
\draw[<->] (2,-1) -- (2,-1.8);
\draw (2,-1.4) node[right] {$\mu_4$};
\end{scope}
\begin{scope}[>=latex,xshift=190pt]
\draw (-0.5,2.5) node{$B^{(6)}=B'$};
{\color{blue}
\draw (2,-0.5) circle(2pt) coordinate(1) node[below]{\scriptsize$1$};
\draw (1,0.5) circle(2pt) coordinate(2) node[below left]{\scriptsize$2$};
\draw (3,1.5) circle(2pt) coordinate(3) node[above]{\scriptsize$3$};
\draw (2,1.5) circle(2pt) coordinate(4) node[below left]{\scriptsize$4$};
\draw (1,1.5) circle(2pt) coordinate(5) node[above]{\scriptsize$5$};
\draw (3,0.5) circle(2pt) coordinate(6) node[below right]{\scriptsize$6$};
\draw (0,1.5) circle(2pt) coordinate(7) node[above left]{\scriptsize$7$};
\draw (2,0.5) circle(2pt) coordinate(8) node[below left]{\scriptsize$8$};
\draw (4,1.5) circle(2pt) coordinate(9) node[above right]{\scriptsize$9$};
\draw (2,2.5) circle(2pt) coordinate(10) node[above]{\scriptsize$10$};
\draw[->,dashed, shorten >=2pt,shorten <=2pt] (10) to [out = -165, in = 60] (7) [thick];
\draw[->,shorten >=2pt,shorten <=2pt] (5) to [out = 60, in = -150] (10) [thick];\qarrow{10}{4}
\draw[->,shorten >=2pt,shorten <=2pt] (3) to [out = 120, in = -30] (10) [thick];
\draw[->,dashed, shorten >=2pt,shorten <=2pt] (10) to [out = -15, in = 120] (9) [thick];
\qarrow{7}{5}
\qarrow{4}{5}
\qarrow{4}{3}
\qarrow{9}{3}
\qdarrow{2}{7}
\qarrow{8}{4}
\qarrow{5}{2}
\qarrow{3}{6}
\qdarrow{6}{9}
\qarrow{2}{8}
\qarrow{6}{8}
\qdarrow{1}{2}
\qarrow{8}{1}
\qdarrow{1}{6}
}
\draw[<->] (2,-1) -- (2,-1.8);
\draw (2,-1.4) node[right] {$\sigma_{3,5} \sigma_{4,8}$};
\end{scope}
\begin{scope}[>=latex,xshift=0pt,yshift=-140pt]
\draw (-0.5,2.5) node{$B^{(2)}$};
{\color{blue}
\draw (2,-0.5) circle(2pt) coordinate(1) node[below]{\scriptsize$1$};
\draw (0,0.5) circle(2pt) coordinate(2) node[below left]{\scriptsize$2$};
\draw (1,0.5) circle(2pt) coordinate(3) node[below=1pt]{\scriptsize$3$};
\draw (2,0.5) circle(2pt) coordinate(4) node[above left]{\scriptsize$4$};
\draw (3,0.5) circle(2pt) coordinate(5) node[below=1pt]{\scriptsize$5$};
\draw (4,0.5) circle(2pt) coordinate(6) node[below right]{\scriptsize$6$};
\draw (1,1.5) circle(2pt) coordinate(7) node[above left]{\scriptsize$7$};
\draw (2,1.5) circle(2pt) coordinate(8) node[above left]{\scriptsize$8$};
\draw (3,1.5) circle(2pt) coordinate(9) node[above right]{\scriptsize$9$};
\draw (2,2.5) circle(2pt) coordinate(10) node[above]{\scriptsize$10$};
\draw[->,dashed, shorten >=2pt,shorten <=2pt] (1) to [out = 165, in = -60] (2) [thick];
\qarrow{4}{1}
\draw[->,dashed, shorten >=2pt,shorten <=2pt] (1) to [out = 15, in = -120] (6) [thick];
\qarrow{6}{5}
\qarrow{5}{4}
\qarrow{3}{4}
\qarrow{2}{3}
\qdarrow{7}{2}
\qarrow{3}{7}
\qarrow{8}{3}
\qarrow{4}{8}
\qarrow{8}{5}
\qarrow{5}{9}
\qdarrow{9}{6}
\qarrow{9}{8}
\qarrow{7}{8}
\qdarrow{10}{7}
\qarrow{8}{10}
\qdarrow{10}{9}
}
\draw[<->] (2,-1) -- (2,-1.8);
\draw (2,-1.4) node[right] {$\mu_3$};
\end{scope}
\begin{scope}[>=latex,xshift=190pt,yshift=-140pt]
\draw (-0.5,2.5) node{$B^{(5)}$};
{\color{blue}
\draw (2,-0.5) circle(2pt) coordinate(1) node[below]{\scriptsize$1$};
\draw (1,0.5) circle(2pt) coordinate(2) node[below left]{\scriptsize$2$};
\draw (3,1.5) circle(2pt) coordinate(3) node[above]{\scriptsize$5$};
\draw (2,1.5) circle(2pt) coordinate(4) node[below left]{\scriptsize$8$};
\draw (1,1.5) circle(2pt) coordinate(5) node[above]{\scriptsize$3$};
\draw (3,0.5) circle(2pt) coordinate(6) node[below right]{\scriptsize$6$};
\draw (0,1.5) circle(2pt) coordinate(7) node[above left]{\scriptsize$7$};
\draw (2,0.5) circle(2pt) coordinate(8) node[below left]{\scriptsize$4$};
\draw (4,1.5) circle(2pt) coordinate(9) node[above right]{\scriptsize$9$};
\draw (2,2.5) circle(2pt) coordinate(10) node[above]{\scriptsize$10$};
\draw[->,dashed, shorten >=2pt,shorten <=2pt] (10) to [out = -165, in = 60] (7) [thick];
\draw[->,shorten >=2pt,shorten <=2pt] (5) to [out = 60, in = -150] (10) [thick];\qarrow{10}{4}
\draw[->,shorten >=2pt,shorten <=2pt] (3) to [out = 120, in = -30] (10) [thick];
\draw[->,dashed, shorten >=2pt,shorten <=2pt] (10) to [out = -15, in = 120] (9) [thick];
\qarrow{7}{5}
\qarrow{4}{5}
\qarrow{4}{3}
\qarrow{9}{3}
\qdarrow{2}{7}
\qarrow{8}{4}
\qarrow{5}{2}
\qarrow{3}{6}
\qdarrow{6}{9}
\qarrow{2}{8}
\qarrow{6}{8}
\qdarrow{1}{2}
\qarrow{8}{1}
\qdarrow{1}{6}
}
\draw[<->] (2,-1) -- (2,-1.8);
\draw (2,-1.4) node[right] {$\mu_8$};
\end{scope}
\begin{scope}[>=latex,xshift=0pt,yshift=-280pt]
\draw (-0.5,2.5) node{$B^{(3)}$};
{\color{blue}
\draw (2,-0.5) circle(2pt) coordinate(1) node[below]{\scriptsize$1$};
\draw (1,0.5) circle(2pt) coordinate(2) node[below left]{\scriptsize$2$};
\draw (1,1.5) circle(2pt) coordinate(3) node[above]{\scriptsize$3$};
\draw (2,0.5) circle(2pt) coordinate(4) node[below right]{\scriptsize$4$};
\draw (3,0.5) circle(2pt) coordinate(5) node[below]{\scriptsize$5$};
\draw (4,0.5) circle(2pt) coordinate(6) node[below right]{\scriptsize$6$};
\draw (0,1.5) circle(2pt) coordinate(7) node[above left]{\scriptsize$7$};
\draw (2,1.5) circle(2pt) coordinate(8) node[above left]{\scriptsize$8$};
\draw (3,1.5) circle(2pt) coordinate(9) node[above right]{\scriptsize$9$};
\draw (2,2.5) circle(2pt) coordinate(10) node[above]{\scriptsize$10$};
\qdarrow{1}{2}
\qarrow{4}{1}
\draw[->,dashed, shorten >=2pt,shorten <=2pt] (1) to [out = 15, in = -120] (6) [thick];;
\qarrow{6}{5}
\qarrow{5}{4}
\qarrow{2}{4}
\qdarrow{2}{7}
\qarrow{3}{2}
\qarrow{4}{3}
\qarrow{8}{5}
\qarrow{5}{9}
\qdarrow{9}{6}
\qarrow{7}{3}
\qarrow{3}{8}
\qarrow{9}{8}
\draw[->,dashed, shorten >=2pt,shorten <=2pt] (10) to [out = -165, in = 60] (7) [thick];
\qarrow{8}{10}
\qdarrow{10}{9}
}
\draw[<->] (4.8,1) -- (5.8,1);
\draw (5.3,1.3) node {$\mu_5$};
\end{scope}
\begin{scope}[>=latex,xshift=190pt,yshift=-280pt]
\draw (-0.5,2.5) node{$B^{(4)}$};
{\color{blue}
\draw (2,-0.5) circle(2pt) coordinate(1) node[below]{\scriptsize$1$};
\draw (1,0.5) circle(2pt) coordinate(2) node[below left]{\scriptsize$2$};
\draw (3,1.5) circle(2pt) coordinate(3) node[above]{\scriptsize$5$};
\draw (2,1.5) circle(2pt) coordinate(4) node[below left]{\scriptsize$8$};
\draw (1,1.5) circle(2pt) coordinate(5) node[above]{\scriptsize$3$};
\draw (3,0.5) circle(2pt) coordinate(6) node[below right]{\scriptsize$6$};
\draw (0,1.5) circle(2pt) coordinate(7) node[above left]{\scriptsize$7$};
\draw (2,0.5) circle(2pt) coordinate(8) node[below left]{\scriptsize$4$};
\draw (4,1.5) circle(2pt) coordinate(9) node[above right]{\scriptsize$9$};
\draw (2,2.5) circle(2pt) coordinate(10) node[above]{\scriptsize$10$};
\draw[->,dashed, shorten >=2pt,shorten <=2pt] (10) to [out = -165, in = 60] (7) [thick];
\draw[->,dashed, shorten >=2pt,shorten <=2pt] (10) to [out = -15, in = 120] (9) [thick];
\qarrow{4}{10}
\qarrow{7}{5}
\qarrow{5}{4}
\qarrow{3}{4}
\qarrow{9}{3}
\qdarrow{2}{7}
\qarrow{4}{8}
\qarrow{8}{5}
\qarrow{8}{3}
\qarrow{5}{2}
\qarrow{3}{6}
\qdarrow{6}{9}
\qarrow{2}{8}
\qarrow{6}{8}
\qdarrow{1}{2}
\qarrow{8}{1}
\qdarrow{1}{6}
}
\end{scope}
\end{tikzpicture}
\caption{The quivers $B=B^{(1)}, \dots, B^{(6)}=B'$ and the
 mutations connecting them. We do not consider the wiring diagrams
 corresponding to the intermediate ones $B^{(2)}, \dots, B^{(5)}$.}
\label{fig:mus}
\end{figure}
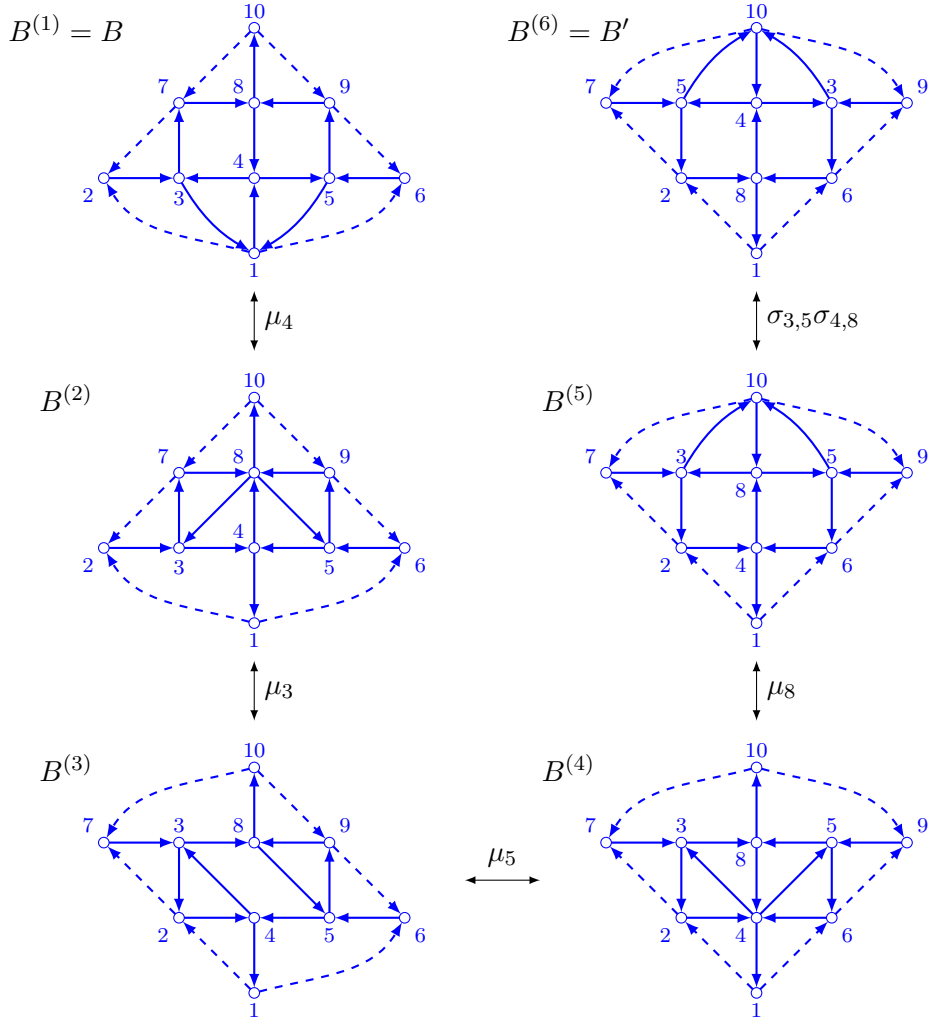

For simplicity, we identify \smash{$Y^{(t)}_i$} and $\rY_i(t)$ in the
description from now on. We introduce the cluster transformation
\smash{$\widehat{R}\colon \mathcal{Y}\bigl(B'\bigr) \rightarrow \mathcal{Y}(B)$}
corresponding to the mutation sequence \eqref{mus} by applying
\eqref{yad} as
\begin{align}
\widehat{R} ={}&
\mathrm{Ad}\bigl(\Psi_q\bigl(\bigl(Y^{(1)}_4\bigr)^{\varepsilon_1}\bigr)^{\varepsilon_1}\bigr)\tau_{4,\varepsilon_1}
\mathrm{Ad}\bigl(\Psi_q\bigl(\bigl(Y^{(2)}_3\bigr)^{\varepsilon_2}\bigr)^{\varepsilon_2}\bigr)\tau_{3,\varepsilon_2}\nonumber
\\
 & \times
\mathrm{Ad}\bigl(\Psi_q\bigl(\bigl(Y^{(3)}_5\bigr)^{\varepsilon_3}\bigr)^{\varepsilon_3}\bigr)\tau_{5,\varepsilon_3}
\mathrm{Ad}\bigl(\Psi_q\bigl(\bigl(Y^{(4)}_8\bigr)^{\varepsilon_4}\bigr)^{\varepsilon_4}\bigr)\tau_{8,\varepsilon_4}
\sigma_{3,5}\sigma_{4,8}.\label{rhat}
\end{align}
The selection of $(\varepsilon_1, \varepsilon_2, \varepsilon_3, \varepsilon_4) \in \{1,-1\}^4$
influences the expressions, but \smash{$\widehat{R}$} itself remains independent of it.
We set
\begin{align}\label{taue}
\tau_{\varepsilon_1, \varepsilon_2, \varepsilon_3, \varepsilon_4} =
\tau_{4,\varepsilon_1}
\tau_{3,\varepsilon_2}
\tau_{5,\varepsilon_3}
\tau_{8,\varepsilon_4} \sigma_{3,5}\sigma_{4,8}\colon\
\mathcal{Y}\bigl(B'\bigr) \rightarrow \mathcal{Y}(B),
 \end{align}
 and call it the monomial part of \smash{$\widehat{R}$}.

\begin{Example}\label{ex:1}
$\tau_{--++}$ and $\tau_{--++}^{-1}$ are given as follows:
\begin{align*}
&\tau_{--++}\colon\
\begin{cases}
Y'_1 \mapsto Y_1,\qquad
Y'_2 \mapsto Y_2, \qquad
Y'_3 \mapsto Y_8, \qquad
Y'_4 \mapsto Y^{-1}_4Y^{-1}_5Y^{-1}_8, \qquad
\\
Y'_5 \mapsto Y^{-1}_3Y_5Y_8, \qquad
Y'_6 \mapsto Y_4 Y_5 Y_6,\qquad
Y'_7 \mapsto Y_3Y_4Y_7,\qquad
Y'_8 \mapsto Y_3,\qquad
\\
Y'_9 \mapsto Y_9, \qquad
Y'_{10} \mapsto Y_{10},
\end{cases}
\\
&\tau_{--++}^{-1}\colon\
\begin{cases}
Y_1 \mapsto Y'_1, \qquad
Y_2 \mapsto Y'_2, \qquad
Y_3 \mapsto Y'_8, \qquad
Y_4 \mapsto Y'^{-1}_4Y'^{-1}_5Y'^{-1}_8,
\\
Y_5 \mapsto Y'^{-1}_3Y'_5Y'_8, \qquad
Y_6 \mapsto Y'_3Y'_4Y'_6, \qquad
Y_7 \mapsto Y'_4Y'_5Y'_7, \qquad
Y_8 \mapsto Y'_3,
\\
Y_9 \mapsto Y'_9, \qquad
Y_{10} \mapsto Y'_{10}.
\end{cases}
\end{align*}
By using them, $\widehat{R}$ in \eqref{rhat} for the choice
$(\varepsilon_1, \varepsilon_2, \varepsilon_3, \varepsilon_4)=(-,-,+,+)$
is expressed as
\begin{align}
\widehat{R}={}&
\mathrm{Ad}\bigl(\Psi_q\bigl(\bigl(Y^{(1)}_4\bigr)^{-1}\bigr)^{-1}\bigr)\tau_{4,-}
\mathrm{Ad}\bigl(\Psi_q\bigl(\bigl(Y^{(2)}_3\bigr)^{-1}\bigr)^{-1}\bigr)\tau_{3,-}
\nonumber \\
& \times
\mathrm{Ad}\bigl(\Psi_q\bigl(Y^{(3)}_5\bigr)\bigr)\tau_{5,+}
 \mathrm{Ad}\bigl(\Psi_q\bigl(Y^{(4)}_8\bigr)\bigr)\tau_{8,+}\sigma_{3,5}\sigma_{4,8}
\nonumber \\
={}&
\mathrm{Ad}\bigl(\Psi_q\bigl(Y_4^{-1}\bigr)^{-1}
\Psi_q\bigl(qY^{-1}_3Y^{-1}_4\bigr)^{-1}
\Psi_q\bigl(q^{-1}Y_4Y_5\bigr)
\Psi_q\bigl(Y_4Y_5Y_8\bigr)\bigr)\tau_{--++}
\label{dode0}\\
={}&
\mathrm{Ad}\bigl(
\Psi_q\bigl(Y_4^{-1}\bigr)^{-1}
\Psi_q\bigl(qY^{-1}_3Y^{-1}_4\bigr)^{-1}\bigr)
\tau_{--++}
\mathrm{Ad}\bigl(
\Psi_q\bigl(qY'^{-1}_3Y'^{-1}_4\bigr)
\Psi_q\bigl(Y'^{-1}_4\bigr)\bigr).
\label{dode}
\end{align}
The formula \eqref{dode} is derived from \eqref{dode0}
by moving $\tau_{--++}$ to the left by using $\tau_{--++}^{-1}$.
\end{Example}

\begin{Example}\label{ex:2}
$\tau_{-+-+}$ and $\tau_{-+-+}^{-1}$ are given as follows:
\begin{align*}
&\tau_{-+-+}\colon\
\begin{cases}
Y'_1 \mapsto Y_1,\qquad
Y'_2 \mapsto Y_2Y_3Y_4, \qquad
Y'_3 \mapsto Y_3Y^{-1}_5Y_8, \qquad
Y'_4 \mapsto q^2Y^{-1}_3Y^{-1}_4Y^{-1}_8,
\\
Y'_5 \mapsto Y_8, \qquad
Y'_6 \mapsto Y_6,\qquad
Y'_7 \mapsto Y_7,\qquad
Y'_8 \mapsto Y_5,
\\
Y'_9 \mapsto q^{-2}Y_4Y_5Y_9, \qquad
Y'_{10} \mapsto Y_{10},
\end{cases}
\\
& \tau_{-+-+}^{-1}\colon\
\begin{cases}
Y_1 \mapsto Y'_1, \qquad\!
Y_2 \mapsto Y'_2Y'_4Y'_5, \qquad\!
Y_3 \mapsto Y'_3Y'^{-1}_5Y'_8, \qquad\!
Y_4 \mapsto q^2Y'^{-1}_3Y'^{-1}_4Y'^{-1}_8,
\\
Y_5 \mapsto Y'_8, \qquad
Y_6 \mapsto Y'_6, \qquad
Y_7 \mapsto Y'_7, \qquad
Y_8 \mapsto Y'_5,
\\
Y_9 \mapsto q^2Y'_3Y'_4Y'_9, \qquad
Y_{10} \mapsto Y'_{10}.
\end{cases}
\end{align*}
By using them, $\widehat{R}$ in \eqref{rhat} for the choice
$(\varepsilon_1, \varepsilon_2, \varepsilon_3, \varepsilon_4)=(-,+,-,+)$
is expressed as
\begin{align}
\widehat{R}={}&
\mathrm{Ad}\bigl(\Psi_q\bigl(\bigl(Y^{(1)}_4\bigr)^{-1}\bigr)^{-1}\bigr)\tau_{4,-}
\mathrm{Ad}\bigl(\Psi_q\bigl(Y^{(2)}_3\bigr)\bigr)\tau_{3,+}
\nonumber \\
& \times
\mathrm{Ad}\bigl(\Psi_q\bigl(\bigl(Y^{(3)}_5\bigr)^{-1}\bigr)^{-1}\bigr)\tau_{5,-}
\mathrm{Ad}\bigl(\Psi_q\bigl(Y^{(4)}_8\bigr)\bigr)\tau_{8,+}\sigma_{3,5}\sigma_{4,8}
\nonumber \\
={}&
\mathrm{Ad}\bigl(\Psi_q\bigl(Y_4^{-1}\bigr)^{-1}
\Psi_q(qY_3Y_4)
\Psi_q\bigl(qY^{-1}_5Y^{-1}_4\bigr)^{-1}
\Psi_q\bigl(q^2Y_3Y_4Y_8\bigr)\bigr)\tau_{-+-+}
\label{dode1}\\
={}&
\mathrm{Ad}\bigl(
\Psi_q\bigl(Y_4^{-1}\bigr)^{-1}
\Psi_q(qY_3Y_4)^{-1}\bigr)
\tau_{-+-+}
\mathrm{Ad}\bigl(
\Psi_q(qY'_3Y'_4)^{-1}
\Psi_q\bigl(Y'^{-1}_4\bigr)\bigr).
\label{dode2}
\end{align}
\end{Example}

Performing a straightforward calculation using any one
of the formulas for $\widehat{R}$ in Examples~\ref{ex:1} and \ref{ex:2},
we get the following.

\begin{Proposition}\label{pr:RY}
The cluster transformation $\widehat{R}\colon \mathcal{Y}\bigl(B'\bigr) \rightarrow \mathcal{Y}(B)$ is given by
\begin{gather*}
Y'_1 \mapsto q\Lambda^{-1}_4Y_4Y_1,\qquad
Y'_2 \mapsto qY_3\Lambda_4\Lambda^{-1}_3Y_2,
\qquad
Y'_3 \mapsto q^2 \Lambda^{-1}_0Y_3Y_4Y_8,
\\
Y'_4 \mapsto q\Lambda^{-1}_4Y^{-1}_3Y^{-1}_4Y^{-1}_5Y^{-1}_8\Lambda_3\Lambda_5,
\qquad
Y'_5 \mapsto \Lambda^{-1}_0 Y_4Y_5Y_8,
\qquad
Y'_6 \mapsto q Y_5 \Lambda_4\Lambda^{-1}_5Y_6,
\\
Y'_7 \mapsto Y_7\Lambda_3,
\qquad
Y'_8 \mapsto Y^{-1}_4\Lambda_0,
\qquad
Y'_9 \mapsto Y_9\Lambda_5,
\qquad
Y'_{10} \mapsto Y_{10}\Lambda^{-1}_5\Lambda^{-1}_3\Lambda_0,
\end{gather*}
where $\Lambda_0$, $\Lambda_3$, $\Lambda_4$ and $\Lambda_5$ are given
as follows:
\begin{gather*}
\Lambda_0 = \Lambda_3\Lambda_5 + Y_4Y_3Y_5Y_8\Lambda_4,\qquad
\Lambda_3 = 1+qY_3+Y_4Y_3,\qquad
\Lambda_4 = 1 + q Y_4, \\
\Lambda_5 = 1+qY_5+Y_4Y_5.
\end{gather*}
\end{Proposition}

\begin{Remark}\label{re:c2}
$\widehat{R}$ preserves the following combinations:
\begin{gather*}
\widehat{R}\bigl(Y'_3Y'_8\bigr) = Y_3Y_8,
\qquad
\widehat{R}\bigl(Y'_5Y'_8\bigr) = Y_5Y_8,
\\
\widehat{R}\bigl(Y'_1Y'_2Y'_4Y'_6Y'_7Y'^{-1}_8Y'_9Y'_{10}\bigr) =
Y_1Y_2Y_4Y_6Y_7Y^{-1}_8Y_9Y_{10}.
\end{gather*}
\end{Remark}

\subsection[R satisfies the tetrahedron equation]{$\boldsymbol{\widehat{R}}$ satisfies the tetrahedron equation}

In the situation in Figure \ref{fig:R123}, $\widehat{R}$ is a
transformation of the 10 variables $\{Y'_1,\dotsc, Y'_{10}\}$ into
$\{Y_1,\dotsc, Y_{10}\}$ as in Proposition \ref{pr:RY}. We denote it
simply by $\widehat{R}_{123}$, where the indices $1$, $2$, $3$ are the
vertices $1$, $2$, $3$ of the wiring diagram (highlighted in red).

\begin{figure}[t]\centering
\begin{tikzpicture}
\begin{scope}[>=latex,xshift=0pt]
{\color{red}
\fill (1,0.5) circle(2pt) coordinate(A);
\draw (1.1,0.8) node {\scriptsize$3$};
\fill (2,1.5) circle(2pt) coordinate(B);
\draw (2.1,1.2) node {\scriptsize$2$};
\fill (3,0.5) circle(2pt) coordinate(C);
\draw (3.1,0.8) node {\scriptsize$1$};
}
\draw [-] (0,2) to [out = 0, in = 135] (B);
\draw [-] (B) -- (C);
\draw [-] (C) to [out = -45, in = 180] (4,0);
\draw [-] (0,1) to [out = 0, in = 135] (A);
\draw [-] (A) to [out = -45, in = -135] (C);
\draw [-] (C) to [out = 45, in = 180] (4,1);
\draw [-] (0,0) to [out = 0, in = -135] (A);
\draw [-] (A) -- (B);
\draw [-] (B) to [out = 45, in = 180] (4,2);
{\color{red}
\coordinate (P1) at (4.5,1);
\coordinate (P2) at (5.5,1);
\draw[->] (5.5,1) -- (4.5,1);
\draw (5,1.8) circle(0pt) node[below]{$\widehat{R}_{123}$};
}
{\color{blue}
\draw (2,-0.5) circle(2pt) coordinate(1) node[below]{\scriptsize$1$};
\draw (0,0.5) circle(2pt) coordinate(2) node[below left]{\scriptsize$2$};
\draw (1,0.5) circle(2pt) coordinate(3) node[below=1pt]{\scriptsize$3$};
\draw (2,0.5) circle(2pt) coordinate(4) node[above left]{\scriptsize$4$};
\draw (3,0.5) circle(2pt) coordinate(5) node[below=1pt]{\scriptsize$5$};
\draw (4,0.5) circle(2pt) coordinate(6) node[below right]{\scriptsize$6$};
\draw (1,1.5) circle(2pt) coordinate(7) node[above left]{\scriptsize$7$};
\draw (2,1.5) circle(2pt) coordinate(8);
\draw (1.9,1.75) node {\scriptsize$8$};
\draw (3,1.5) circle(2pt) coordinate(9) node[above right]{\scriptsize$9$};
\draw (2,2.5) circle(2pt) coordinate(10) node[above]{\scriptsize$10$};
%
\draw[->,dashed, shorten >=2pt,shorten <=2pt] (1) to [out = 165, in = -60] (2) [thick];
\draw[->,shorten >=2pt,shorten <=2pt] (3) to [out = -60, in = 150] (1) [thick];
\qarrow{1}{4}
\draw[->,shorten >=2pt,shorten <=2pt] (5) to [out = -120, in = 30] (1) [thick];
\draw[->,dashed, shorten >=2pt,shorten <=2pt] (1) to [out = 15, in = -120] (6) [thick];
\qarrow{6}{5}
\qarrow{4}{5}
\qarrow{4}{3}
\qarrow{2}{3}
\qdarrow{7}{2}
\qarrow{3}{7}
\qarrow{8}{4}
\qarrow{5}{9}
\qdarrow{9}{6}
\qarrow{9}{8}
\qarrow{7}{8}
\qdarrow{10}{7}
\qarrow{8}{10}
\qdarrow{10}{9}
}
\end{scope}
\begin{scope}[>=latex,xshift=170pt]
{\color{red}
\fill (3,1.5) circle(2pt) coordinate(A);
\draw (3.1,1.2) node {\scriptsize$3$};
\fill (2,0.5) circle(2pt) coordinate(B);
\draw (2.1,0.8) node {\scriptsize$2$};
\fill (1,1.5) circle(2pt) coordinate(C);
\draw (1.1,1.2) node {\scriptsize$1$};
}
\draw [-] (0,0) to [out = 0, in = -135] (B);
\draw [-] (B) -- (A);
\draw [-] (A) to [out = 45, in = 180] (4,2);
\draw [-] (0,1) to [out = 0, in = -135] (C);
\draw [-] (C) to [out = 45, in = 135] (A);
\draw [-] (A) to [out = -45, in = 180] (4,1);
\draw [-] (0,2) to [out = 0, in = 135] (C);
\draw [-] (C) -- (B);
\draw [-] (B) to [out = -45, in = 180] (4,0);
{\color{blue}
\draw (2,-0.5) circle(2pt) coordinate(1) node[below]{\scriptsize$1$};
\draw (1,0.5) circle(2pt) coordinate(2) node[below left]{\scriptsize$2$};
\draw (3,1.5) circle(2pt) coordinate(3) node[above]{\scriptsize$3$};
\draw (2,1.5) circle(2pt) coordinate(4) node[below left]{\scriptsize$4$};
\draw (1,1.5) circle(2pt) coordinate(5) node[above]{\scriptsize$5$};
\draw (3,0.5) circle(2pt) coordinate(6) node[below right]{\scriptsize$6$};
\draw (0,1.5) circle(2pt) coordinate(7) node[above left]{\scriptsize$7$};
\draw (2,0.5) circle(2pt) coordinate(8);
\draw (1.9,0.2) node {\scriptsize$8$};
\draw (4,1.5) circle(2pt) coordinate(9) node[above right]{\scriptsize$9$};
\draw (2,2.5) circle(2pt) coordinate(10) node[above]{\scriptsize$10$};
\draw[->,dashed, shorten >=2pt,shorten <=2pt] (10) to [out = -165, in = 60] (7) [thick];
\draw[->,shorten >=2pt,shorten <=2pt] (5) to [out = 60, in = -150] (10) [thick];\qarrow{10}{4}
\draw[->,shorten >=2pt,shorten <=2pt] (3) to [out = 120, in = -30] (10) [thick];
\draw[->,dashed, shorten >=2pt,shorten <=2pt] (10) to [out = -15, in = 120] (9) [thick];
\qarrow{7}{5}
\qarrow{4}{5}
\qarrow{4}{3}
\qarrow{9}{3}
\qdarrow{2}{7}
\qarrow{8}{4}
\qarrow{5}{2}
\qarrow{3}{6}
\qdarrow{6}{9}
\qarrow{2}{8}
\qarrow{6}{8}
\qdarrow{1}{2}
\qarrow{8}{1}
\qdarrow{1}{6}
}
\end{scope}
\end{tikzpicture}

\caption{Cluster transformation $\widehat{R}_{123}$, which acts on the
 $q$-Weyl variables attached to the vertices $1$, $2$, $3$ of the
 wiring diagram colored in red.}
\label{fig:R123}
\end{figure}
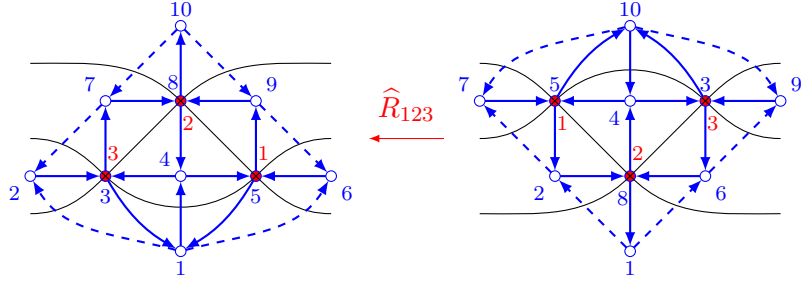

The following result is essentially due to \cite{SY22}.
\begin{Proposition}\label{pr:sy}
$\widehat{R}$ satisfies the tetrahedron equation twisted by permutations of $Y$-variables
\begin{align}\label{TE1}
\widehat{R} _{124}\widehat{R} _{135}\widehat{R} _{236} \widehat{R} _{456}\sigma_{7,12}=
\widehat{R} _{456}\widehat{R} _{236}\widehat{R} _{135}\widehat{R} _{124}\sigma_{7,14} .
\end{align}
\end{Proposition}
\begin{proof}
 For each reduced word for the longest element of the Weyl group
 $W(A_3)$, draw a wiring diagram and a symmetric butterfly quiver
 extending Figure \ref{fig:R123} naturally. The quivers and the
 crossings of the wiring diagrams (red vertices $1, \dots, 6$) are
 connected by the cluster transformations $\widehat{R}_{ijk}$ as in
 Figure \ref{fig:bTE}. 
In Figure \ref{fig:bTE},
let $\nu$ and $\nu'$ be the mutation sequences corresponding to the left path
\begin{gather*}
\bigl(B^{(1)},Y^{(1)}\bigr) \rightarrow
\bigl(B^{(6)},Y^{(6)}\bigr) \rightarrow
\bigl(B^{(11)},Y^{(11)}\bigr) \rightarrow
\bigl(B^{(16)},Y^{(16)}\bigr) \rightarrow
\bigl(B^{(22)},Y^{(22)}\bigr)\\
\qquad=\nu\bigl(B^{(1)},Y^{(1)}\bigr)
\end{gather*}
and the right path
\begin{gather*}
\bigl(B^{(1)\prime},Y^{(1)\prime}\bigr) \rightarrow
\bigl(B^{(6)\prime},Y^{(6)\prime}\bigr) \rightarrow
\bigl(B^{(11)\prime},Y^{(11)\prime}\bigr) \rightarrow
\bigl(B^{(16)\prime},Y^{(16)\prime}\bigr) \rightarrow
 \bigl(B^{(22)\prime},Y^{(22)\prime}\bigr)\\
 \qquad=\nu'\bigl(B^{(1)},Y^{(1)}\bigr),
\end{gather*}
 respectively.
Let \smash{$\nu\bigl(B^{(1)},y^{(1)}\bigr)$} and \smash{$\nu'\bigl(B^{(1)},y^{(1)}\bigr)$} be the
tropical $y$-seeds generated by the same mutation sequences.
It has been checked \cite[Section~A.2]{SY22}
that they satisfy the equality
$\nu\bigl(B^{(1)},y^{(1)}\bigr)= \nu'\bigl(B^{(1)},y^{(1)}\bigr)$.
Thus Theorem \ref{thm:piv} enforces the equality of quantum $y$-seeds
\smash{$\nu\bigl(B^{(1)}, Y^{(1)}\bigr) = \nu'\bigl(B^{(1)\prime}, Y^{(1)\prime}\bigr)$}.
In terms of cluster transformations, it means that the twisted tetrahedron equation
$\widehat{R}_{124}\widehat{R}_{135}\widehat{R}_{236}\widehat{R}_{456}\sigma_{7,12}
=\widehat{R}_{456}\widehat{R}_{236}\widehat{R}_{135}\widehat{R}_{124} \sigma_{7,14}$ is valid.
\end{proof}

\begin{figure}[t]\centering

\includegraphics[clip,scale=0.77]{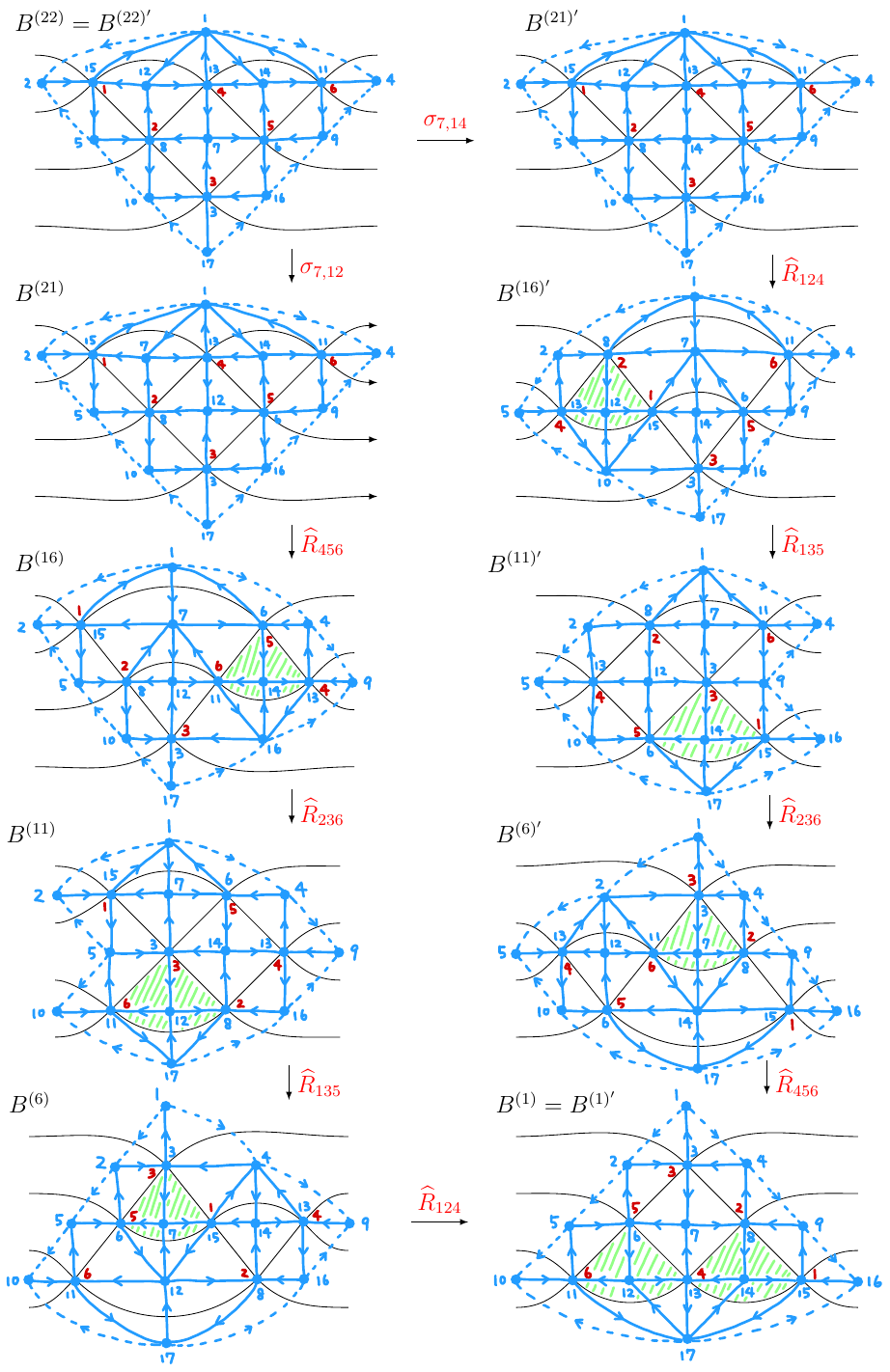}
\vspace{-2mm}
\caption{Cluster transformations $\widehat{R}_{ijk}$.
The wiring diagrams have 6 crossings (red).
The quivers (blue) have 17 vertices.
Triangles relevant to the image of $\hat{R}_{ijk}$ are hatched in green.
The seeds \smash{$\bigl(B^{(t)}, \bY^{(t)}\bigr)$} and~\smash{$\bigl(B^{(t)\prime}, \bY^{(t)\prime}\bigr)$} will be
explained in detail in Section \ref{ss:ms}.}\label{fig:bTE}\vspace{-1mm}
\end{figure}

\subsection{Monomial solutions to the tetrahedron equation}\label{ss:ms}

In this subsection, we provide additional details regarding Figure~\ref{fig:bTE} and Proposition~\ref{pr:sy}. Let
\smash{$\bigl(B^{(1)}, \bY^{(1)}\bigr) = \bigl(B^{(1)\prime}, \bY^{(1)\prime}\bigr)$} be the
initial quantum $y$-seed corresponding to the quiver at the bottom of
Figure~\ref{fig:bTE}. The quantum $y$-seeds $\bigl(B^{(t)}, \bY^{(t)}\bigr)$
and $\bigl(B^{(t)\prime}, \bY^{(t)\prime}\bigr)$ ($t=2, \dots, 21$), which
pertain to the left and the right paths are determined from it by the
mutation sequences, and we have just shown that the final results
coincide, i.e.,
$\bigl(B^{(22)}, \bY^{(22)}\bigr) = \bigl(B^{(22)\prime}, \bY^{(22)\prime}\bigr)$. Set~\smash{$\bY^{(t)}=\bigl(Y^{(t)}_1,\dotsc, Y^{(t)}_{17}\bigr)$} and
\smash{$\bY^{(t)\prime}=\bigl(Y^{(t)\prime}_1,\dotsc, Y^{(t)\prime}_{17}\bigr)$}.

The quantum $y$-seeds $\bigl(B^{(t)}, \bY^{(t)}\bigr)$ ($t=2,\dotsc, 21$) are
determined from the initial one $\bigl(B^{(1)},\allowbreak \bY^{(1)}\bigr) $ by
\begin{gather}
\bigl(B^{(1)}, \bY^{(1)}\bigr) \underset{\varepsilon_1}{\overset{\mu_{14}}{\longleftrightarrow}}
\bigl(B^{(2)}, \bY^{(2)}\bigr) \underset{\varepsilon_2}{\overset{\mu_{13}}{\longleftrightarrow}}
\bigl(B^{(3)}, \bY^{(3)}\bigr) \underset{\varepsilon_3}{\overset{\mu_{15}}{\longleftrightarrow}}
\bigl(B^{(4)}, \bY^{(4)}\bigr)\nonumber
\\
\phantom{\bigl(B^{(1)}, \bY^{(1)}\bigr)}{}
 \underset{\varepsilon_4}{\overset{\mu_{8}}{\longleftrightarrow}}
\bigl(B^{(5)}, \bY^{(5)}\bigr) \overset{\sigma_{13,15}\sigma_{8,14}}{\longleftrightarrow}
\bigl(B^{(6)}, \bY^{(6)}\bigr),\nonumber
\\
\bigl(B^{(6)}, \bY^{(6)}\bigr) \underset{\varepsilon_1}{\overset{\mu_{7}}{\longleftrightarrow}}
\bigl(B^{(7)}, \bY^{(7)}\bigr) \underset{\varepsilon_2}{\overset{\mu_{6}}{\longleftrightarrow}}
\bigl(B^{(8)}, \bY^{(8)}\bigr) \underset{\varepsilon_3}{\overset{\mu_{15}}{\longleftrightarrow}}
\bigl(B^{(9)}, \bY^{(9)}\bigr)\nonumber
\\
\phantom{\bigl(B^{(6)}, \bY^{(6)}\bigr)}{}
\underset{\varepsilon_4}{\overset{\mu_{3}}{\longleftrightarrow} }
\bigl(B^{(10)}, \bY^{(10)}\bigr) \overset{\sigma_{6,15}\sigma_{3,7}}{\longleftrightarrow}
\bigl(B^{(11)}, \bY^{(11)}\bigr),\nonumber
\\
\bigl(B^{(11)}, \bY^{(11)}\bigr) \underset{\varepsilon_1}{\overset{\mu_{12}}{\longleftrightarrow} }
\bigl(B^{(12)}, \bY^{(12)}\bigr) \underset{\varepsilon_2}{\overset{\mu_{11}}{\longleftrightarrow} }
\bigl(B^{(13)}, \bY^{(13)}\bigr) \underset{\varepsilon_3}{\overset{\begin{color}{red}\mu_{8}\end{color}}{\longleftrightarrow} }
\bigl(B^{(14)}, \bY^{(14)}\bigr)\nonumber
\\
\phantom{\bigl(B^{(11)}, \bY^{(11)}\bigr)}{}
 \underset{\varepsilon_4}{ \overset{\mu_{3}}{\longleftrightarrow}}
\bigl(B^{(15)}, \bY^{(15)}\bigr) \overset{\sigma_{8,11}\sigma_{3,12}}{\longleftrightarrow}
\bigl(B^{(16)}, \bY^{(16)}\bigr),\nonumber
\\
\bigl(B^{(16)}, \bY^{(16)}\bigr) \underset{\varepsilon_1}{\overset{\begin{color}{red}\mu_{14}\end{color}}{\longleftrightarrow}}
\bigl(B^{(17)}, \bY^{(17)}\bigr) \underset{\varepsilon_2}{\overset{\mu_{11}}{\longleftrightarrow}}
\bigl(B^{(18)}, \bY^{(18)}\bigr) \underset{\varepsilon_3}{\overset{\mu_{13}}{\longleftrightarrow}}
\bigl(B^{(19)}, \bY^{(19)}\bigr)\nonumber
\\
\phantom{\bigl(B^{(16)}, \bY^{(16)}\bigr) }{}
\underset{\varepsilon_4}{ \overset{\mu_{6}}{\longleftrightarrow}}
\bigl(B^{(20)}, \bY^{(20)}\bigr) \overset{\sigma_{11,13}\sigma_{6,14}}{\longleftrightarrow}
\bigl(B^{(21)}, \bY^{(21)}\bigr),\nonumber
\\
\bigl(B^{(21)}, \bY^{(21)}\bigr) \overset{\sigma_{7,12}}{\longleftrightarrow}
\bigl(B^{(22)}, \bY^{(22)}\bigr),\label{b21L}
\end{gather}
where the notation is parallel with \eqref{mus}. In particular the
choice of signs $\varepsilon_1, \dots, \varepsilon_4$ does not
influence the mutations themselves.
According to \eqref{rhat}, each line in the above corresponds to a~cluster transformation appearing in the left path of
Figure \ref{fig:bTE} as follows:
\begin{gather}
\widehat{R}_{124} =
\mathrm{Ad}\bigl(\Psi_q\bigl(\bigl(Y^{(1)}_{14}\bigr)^{\varepsilon_1}\bigr)^{\varepsilon_1}\bigr)\tau_{14,\varepsilon_1}
\mathrm{Ad}\bigl(\Psi_q\bigl(\bigl(Y^{(2)}_{13}\bigr)^{\varepsilon_2}\bigr)^{\varepsilon_2}\bigr)\tau_{13,\varepsilon_2}\nonumber
\\
\phantom{\widehat{R}_{124} =}{} \times
\mathrm{Ad}\bigl(\Psi_q\bigl(\bigl(Y^{(3)}_{15}\bigr)^{\varepsilon_3}\bigr)^{\varepsilon_3}\bigr)\tau_{15,\varepsilon_3}
\mathrm{Ad}\bigl(\Psi_q\bigl(\bigl(Y^{(4)}_{8}\bigr)^{\varepsilon_4}\bigr)^{\varepsilon_4}\bigr)\tau_{8,\varepsilon_4}
\sigma_{13,15}\sigma_{8,14},\nonumber
\\
\widehat{R}_{135} =
\mathrm{Ad}\bigl(\Psi_q\bigl(\bigl(Y^{(6)}_{7}\bigr)^{\varepsilon_1}\bigr)^{\varepsilon_1}\bigr)\tau_{7,\varepsilon_1}
\mathrm{Ad}\bigl(\Psi_q\bigl(\bigl(Y^{(7)}_{6}\bigr)^{\varepsilon_2}\bigr)^{\varepsilon_2}\bigr)\tau_{6,\varepsilon_2}\nonumber
\\
\phantom{\widehat{R}_{135} =}{} \times
\mathrm{Ad}\bigl(\Psi_q\bigl(\bigl(Y^{(8)}_{15}\bigr)^{\varepsilon_3}\bigr)^{\varepsilon_3}\bigr)\tau_{15,\varepsilon_3}
\mathrm{Ad}\bigl(\Psi_q\bigl(\bigl(Y^{(9)}_{3}\bigr)^{\varepsilon_4}\bigr)^{\varepsilon_4}\bigr)\tau_{3,\varepsilon_4}
\sigma_{6,15}\sigma_{3,7},\nonumber
\\
\widehat{R}_{236} =
\mathrm{Ad}\bigl(\Psi_q\bigl(\bigl(Y^{(11)}_{12}\bigr)^{\varepsilon_1}\bigr)^{\varepsilon_1}\bigr)\tau_{12,\varepsilon_1}
\mathrm{Ad}\bigl(\Psi_q\bigl(\bigl(Y^{(12)}_{11}\bigr)^{\varepsilon_2}\bigr)^{\varepsilon_2}\bigr)\tau_{11,\varepsilon_2}\nonumber
\\
\phantom{\widehat{R}_{236} =}{} \times
\mathrm{Ad}\bigl(\Psi_q\bigl(\bigl(Y^{(13)}_{8}\bigr)^{\varepsilon_3}\bigr)^{\varepsilon_3}\bigr)\tau_{8,\varepsilon_3}
\mathrm{Ad}\bigl(\Psi_q\bigl(\bigl(Y^{(14)}_{3}\bigr)^{\varepsilon_4}\bigr)^{\varepsilon_4}\bigr)\tau_{3,\varepsilon_4}
\sigma_{8,11}\sigma_{3,12},\nonumber
\\
\widehat{R}_{456} =
\mathrm{Ad}\bigl(\Psi_q\bigl(\bigl(Y^{(16)}_{14}\bigr)^{\varepsilon_1}\bigr)^{\varepsilon_1}\bigr)\tau_{14,\varepsilon_1}
\mathrm{Ad}\bigl(\Psi_q\bigl(\bigl(Y^{(17)}_{11}\bigr)^{\varepsilon_2}\bigr)^{\varepsilon_2}\bigr)\tau_{11,\varepsilon_2}\nonumber
\\
\phantom{\widehat{R}_{456} =}{} \times
\mathrm{Ad}\bigl(\Psi_q\bigl(\bigl(Y^{(18)}_{13}\bigr)^{\varepsilon_3}\bigr)^{\varepsilon_3}\bigr)\tau_{13,\varepsilon_3}
\mathrm{Ad}\bigl(\Psi_q\bigl(\bigl(Y^{(19)}_{6}\bigr)^{\varepsilon_4}\bigr)^{\varepsilon_4}\bigr)\tau_{6,\varepsilon_4}
\sigma_{11,13}\sigma_{6,14}.\label{r124a}
\end{gather}

The quantum $y$-seeds $\bigl(B^{(t)\prime}, Y^{(t)\prime}\bigr)$
($t=2, \dots, 21$) are determined from the initial one
$\bigl(B^{(1)\prime},\allowbreak Y^{(1)\prime}\bigr)$ by
\begin{gather}
\bigl(B^{(1)\prime}, \bY^{(1)\prime}\bigr) \underset{\varepsilon_1}{\overset{\mu_{12}}{\longleftrightarrow}}
\bigl(B^{(2)\prime}, \bY^{(2)\prime}\bigr) \underset{\varepsilon_2}{\overset{\mu_{11}}{\longleftrightarrow}}
\bigl(B^{(3)\prime}, \bY^{(3)\prime}\bigr) \underset{\varepsilon_3}{\overset{\mu_{13}}{\longleftrightarrow}}
\bigl(B^{(4)\prime}, \bY^{(4)\prime}\bigr)\nonumber
\\
\phantom{\bigl(B^{(1)\prime}, \bY^{(1)\prime}\bigr) }{}
\underset{\varepsilon_4}{\overset{\mu_{6}}{\longleftrightarrow}}
\bigl(B^{(5)\prime}, \bY^{(5)\prime}\bigr) \overset{\sigma_{11,13}\sigma_{6,12}}{\longleftrightarrow}
\bigl(B^{(6)\prime}, \bY^{(6)\prime}\bigr),\nonumber
\\
\bigl(B^{(6)\prime}, \bY^{(6)\prime}\bigr) \underset{\varepsilon_1}{\overset{\mu_{7}}{\longleftrightarrow}}
\bigl(B^{(7)\prime}, \bY^{(7)\prime}\bigr) \underset{\varepsilon_2}{\overset{\mu_{11}}{\longleftrightarrow}}
\bigl(B^{(8)\prime}, \bY^{(8)\prime}\bigr) \underset{\varepsilon_3}{\overset{\mu_{8}}{\longleftrightarrow}}
\bigl(B^{(9)\prime}, \bY^{(9)\prime}\bigr)\nonumber
\\
\phantom{\bigl(B^{(6)\prime}, \bY^{(6)\prime}\bigr)}{}
\underset{\varepsilon_4}{\overset{\mu_{3}}{\longleftrightarrow}}
\bigl(B^{(10)\prime}, \bY^{(10)\prime}\bigr) \overset{\sigma_{8,11}\sigma_{3,7}}{\longleftrightarrow}
\bigl(B^{(11)\prime}, \bY^{(11)\prime}\bigr),\nonumber
\\
\bigl(B^{(11)\prime}, \bY^{(11)\prime}\bigr) \underset{\varepsilon_1}{\overset{\mu_{14}}{\longleftrightarrow}}
\bigl(B^{(12)\prime}, \bY^{(12)\prime}\bigr) \underset{\varepsilon_2}{\overset{\begin{color}{red}\mu_{6}\end{color}}{\longleftrightarrow}}
\bigl(B^{(13)\prime}, \bY^{(13)\prime}\bigr) \underset{\varepsilon_3}{\overset{\mu_{15}}{\longleftrightarrow}}
\bigl(B^{(14)\prime}, \bY^{(14)\prime}\bigr)\nonumber
\\
\phantom{\bigl(B^{(11)\prime}, \bY^{(11)\prime}\bigr)}{}
\underset{\varepsilon_4}{\overset{\mu_{3}}{\longleftrightarrow}}
\bigl(B^{(15)\prime}, \bY^{(15)\prime}\bigr) \overset{\sigma_{6,15}\sigma_{3,14}}{\longleftrightarrow}
\bigl(B^{(16)\prime}, \bY^{(16)\prime}\bigr)\nonumber
\\
\bigl(B^{(16)\prime}, \bY^{(16)\prime}\bigr) \underset{\varepsilon_1}{\overset{\begin{color}{red}\mu_{12}\end{color}}{\longleftrightarrow} }
\bigl(B^{(17)\prime}, \bY^{(17)\prime}\bigr) \underset{\varepsilon_2}{\overset{\mu_{13}}{\longleftrightarrow} }
\bigl(B^{(18)\prime}, \bY^{(18)\prime}\bigr) \underset{\varepsilon_3}{\overset{\mu_{15}}{\longleftrightarrow} }
\bigl(B^{(19)\prime}, \bY^{(19)\prime}\bigr)\nonumber
\\
\phantom{\bigl(B^{(16)\prime}, \bY^{(16)\prime}\bigr)}{}
\underset{\varepsilon_4}{\overset{\mu_{8}}{\longleftrightarrow} }
\bigl(B^{(20)\prime}, \bY^{(20)\prime}\bigr) \overset{\sigma_{13,15}\sigma_{8,12}}{\longleftrightarrow}
\bigl(B^{(21)\prime}, \bY^{(21)\prime}\bigr),\nonumber
\\
\bigl(B^{(21)\prime}, \bY^{(21)\prime}\bigr) \overset{\sigma_{7,14}}{\longleftrightarrow}
\bigl(B^{(22)\prime}, \bY^{(22)\prime}\bigr).\label{b11R}
\end{gather}
They correspond to the cluster transformations in the right path of Figure \ref{fig:bTE} as follows:
\begin{align}
\widehat{R}_{456} ={}&
\mathrm{Ad}\bigl(\Psi_q\bigl(\bigl(Y^{(1)\prime}_{12}\bigr)^{\varepsilon_1}\bigr)^{\varepsilon_1}\bigr)\tau_{12,\varepsilon_1}
\mathrm{Ad}\bigl(\Psi_q\bigl(\bigl(Y^{(2)\prime}_{11}\bigr)^{\varepsilon_2}\bigr)^{\varepsilon_2}\bigr)\tau_{11,\varepsilon_2}\nonumber
\\
& \times
\mathrm{Ad}\bigl(\Psi_q\bigl(\bigl(Y^{(3)\prime}_{13}\bigr)^{\varepsilon_3}\bigr)^{\varepsilon_z3}\bigr)\tau_{13,\varepsilon_3}
\mathrm{Ad}\bigl(\Psi_q\bigl(\bigl(Y^{(4)\prime}_{6}\bigr)^{\varepsilon_4}\bigr)^{\varepsilon_4}\bigr)\tau_{6,\varepsilon_4}
\sigma_{11,13}\sigma_{6,12},\nonumber
\\
\widehat{R}_{236} ={}&
\mathrm{Ad}\bigl(\Psi_q\bigl(\bigl(Y^{(6)\prime}_{7}\bigr)^{\varepsilon_1}\bigr)^{\varepsilon_1}\bigr)\tau_{7,\varepsilon_1}
\mathrm{Ad}\bigl(\Psi_q\bigl(\bigl(Y^{(7)\prime}_{11}\bigr)^{\varepsilon_2}\bigr)^{\varepsilon_2}\bigr)\tau_{11,\varepsilon_2}\nonumber
\\
& \times
\mathrm{Ad}\bigl(\Psi_q\bigl(\bigl(Y^{(8)\prime}_{8}\bigr)^{\varepsilon_3}\bigr)^{\varepsilon_3}\bigr)\tau_{8,\varepsilon_3}
\mathrm{Ad}\bigl(\Psi_q\bigl(\bigl(Y^{(9)\prime}_{3}\bigr)^{\varepsilon_4}\bigr)^{\varepsilon_4}\bigr)\tau_{3,\varepsilon_4}
\sigma_{8,11}\sigma_{3,7},\nonumber
\\
\widehat{R}_{135} ={}&
\mathrm{Ad}\bigl(\Psi_q\bigl(\bigl(Y^{(11)\prime}_{14}\bigr)^{\varepsilon_1}\bigr)^{\varepsilon_1}\bigr)\tau_{14,\varepsilon_1}
\mathrm{Ad}\bigl(\Psi_q\bigl(\bigl(Y^{(12)\prime}_{6}\bigr)^{\varepsilon_2}\bigr)^{\varepsilon_2}\bigr)\tau_{6,\varepsilon_2}\nonumber
\\
& \times
\mathrm{Ad}\bigl(\Psi_q\bigl(\bigl(Y^{(13)\prime}_{15}\bigr)^{\varepsilon_3}\bigr)^{\varepsilon_3}\bigr)\tau_{15,\varepsilon_3}
\mathrm{Ad}\bigl(\Psi_q\bigl(\bigl(Y^{(14)\prime}_{3}\bigr)^{\varepsilon_4}\bigr)^{\varepsilon_4}\bigr)\tau_{3,\varepsilon_4}
\sigma_{6,15}\sigma_{3,14},\nonumber
\\
\widehat{R}_{124} ={}&
\mathrm{Ad}\bigl(\Psi_q\bigl(\bigl(Y^{(16)\prime}_{12}\bigr)^{\varepsilon_1}\bigr)^{\varepsilon_1}\bigr)\tau_{12,\varepsilon_1}
\mathrm{Ad}\bigl(\Psi_q\bigl(\bigl(Y^{(17)\prime}_{13}\bigr)^{\varepsilon_2}\bigr)^{\varepsilon_2}\bigr)\tau_{13,\varepsilon_2}\nonumber
\\
& \times
\mathrm{Ad}\bigl(\Psi_q\bigl(\bigl(Y^{(18)\prime}_{15}\bigr)^{\varepsilon_3}\bigr)^{\varepsilon_3}\bigr)\tau_{15,\varepsilon_3}
\mathrm{Ad}\bigl(\Psi_q\bigl(\bigl(Y^{(19)\prime}_{8}\bigr)^{\varepsilon_4}\bigr)^{\varepsilon_4}\bigr)\tau_{8,\varepsilon_4}
\sigma_{13,15}\sigma_{8,12}.\label{r456b}
\end{align}
Although the formulas \eqref{r124a} and \eqref{r456b} may appear distinct,
they all signify the same transformation described in Proposition \ref{pr:RY}
for the corresponding subsets of $Y$-variables.
This fact justifies denoting them by the common symbol $\widehat{R}$.

\begin{Remark}\label{re:red}
 Consider the tropical $y$-seeds generated by the same mutation
 sequences from the initial one
 \smash{$\bigl(B^{(1)}, y^{(1)}\bigr) = \bigl(B^{(1)\prime}, y^{(1)\prime}\bigr)$}. Suppose
 \smash{$y^{(1)}_i = y^{(1)\prime}_i$} is positive for all $i=1, \dots,
 17$. Then the four mutations highlighted in red in \eqref{b21L}
 and \eqref{b11R} are associated to a negative tropical sign of the
 $y$-variable at the mutation point (the $y$-seed in the left), while
 the remaining ones are positive.
\end{Remark}

Let us introduce the monomial parts of the cluster transformations \eqref{r124a} and \eqref{r456b}
\begin{gather*}
\tau_{124 | \varepsilon_1,\varepsilon_2,\varepsilon_3,\varepsilon_4} :=
\tau_{14,\varepsilon_1}\tau_{13,\varepsilon_2}\tau_{15,\varepsilon_3}\tau_{8,\varepsilon_4}\sigma_{13,15}\sigma_{8,14}\colon\
 \mathcal{Y}\bigl(B^{(6)}\bigr) \rightarrow \mathcal{Y}\bigl(B^{(1)}\bigr),
\\
\tau_{135 | \varepsilon_1,\varepsilon_2,\varepsilon_3,\varepsilon_4} :=
\tau_{7,\varepsilon_1}\tau_{6,\varepsilon_2}\tau_{15,\varepsilon_3}\tau_{3,\varepsilon_4}\sigma_{6,15}\sigma_{3,7}\colon\
 \mathcal{Y}\bigl(B^{(11)}\bigr) \rightarrow \mathcal{Y}\bigl(B^{(6)}\bigr),
\\
\tau_{236 | \varepsilon_1,\varepsilon_2,\varepsilon_3,\varepsilon_4} :=
\tau_{12,\varepsilon_1}\tau_{11,\varepsilon_2}\tau_{8,\varepsilon_3}\tau_{3,\varepsilon_4}\sigma_{8,11}\sigma_{3,12}\colon\
 \mathcal{Y}\bigl(B^{(16)}\bigr) \rightarrow \mathcal{Y}\bigl(B^{(11)}\bigr),
\\
\tau_{456 | \varepsilon_1,\varepsilon_2,\varepsilon_3,\varepsilon_4} :=
\tau_{14,\varepsilon_1}\tau_{11,\varepsilon_2}\tau_{13,\varepsilon_3}\tau_{6,\varepsilon_4}\sigma_{11,13}\sigma_{6,14}\colon\
 \mathcal{Y}\bigl(B^{(21)}\bigr) \rightarrow \mathcal{Y}\bigl(B^{(16)}\bigr),
\\
\tau'_{456 | \varepsilon_1,\varepsilon_2,\varepsilon_3,\varepsilon_4} :=
\tau_{12,\varepsilon_1}\tau_{11,\varepsilon_2}\tau_{13,\varepsilon_3}\tau_{6,\varepsilon_4}\sigma_{11,13}\sigma_{6,12}\colon\
 \mathcal{Y}\bigl(B^{(6)\prime}\bigr) \rightarrow \mathcal{Y}\bigl(B^{(1)\prime}\bigr),
\\
\tau'_{236 | \varepsilon_1,\varepsilon_2,\varepsilon_3,\varepsilon_4} :=
\tau_{7,\varepsilon_1}\tau_{11,\varepsilon_2}\tau_{8,\varepsilon_3}\tau_{3,\varepsilon_4}\sigma_{8,11}\sigma_{3,7}\colon\
 \mathcal{Y}\bigl(B^{(11)\prime}\bigr) \rightarrow \mathcal{Y}\bigl(B^{(6)\prime}\bigr),
\\
\tau'_{135 | \varepsilon_1,\varepsilon_2,\varepsilon_3,\varepsilon_4} :=
\tau_{14,\varepsilon_1}\tau_{6,\varepsilon_2}\tau_{15,\varepsilon_3}\tau_{3,\varepsilon_4}\sigma_{6,15}\sigma_{3,14}\colon\
 \mathcal{Y}\bigl(B^{(16)\prime}\bigr) \rightarrow \mathcal{Y}\bigl(B^{(11)\prime}\bigr),
\\
\tau'_{124 | \varepsilon_1,\varepsilon_2,\varepsilon_3,\varepsilon_4} :=
\tau_{12,\varepsilon_1}\tau_{13,\varepsilon_2}\tau_{15,\varepsilon_3}\tau_{8,\varepsilon_4}\sigma_{13,15}\sigma_{8,12}\colon\
 \mathcal{Y}\bigl(B^{(21)\prime}\bigr) \rightarrow \mathcal{Y}\bigl(B^{(16)\prime}\bigr).
\end{gather*}
The primes in $\tau'_{\!ijk | \varepsilon_1,\varepsilon_2,\varepsilon_3,\varepsilon_4}$
are added just for distinction.
The maps $\tau_{\!ijk | \varepsilon_1,\varepsilon_2,\varepsilon_3,\varepsilon_4}$ and
$\tau'_{\!ijk | \varepsilon_1,\varepsilon_2,\varepsilon_3,\varepsilon_4}$
consistently adhere to $\tau_{\varepsilon_1,\varepsilon_2,\varepsilon_3,\varepsilon_4}$ in
\eqref{taue} with respect to the subset of $Y$-variables.

Now we are ready to explain monomial solutions to the twisted tetrahedron equation.
Proposition \ref{pr:sy}, Figure \ref{fig:bTE} and Remark \ref{re:sgni} indicate
the equality of the tropical $y$-variables $y^{(22)} = y^{(22)\prime}$ {\em provided} that
all the signs associated with the monomial part of the mutation are chosen to be the tropical signs.
Considering Remark \ref{re:red} alongside, we find that
\begin{gather}
\tau_{124| + + + +} \tau_{135| + + + +}\tau_{236| + + - +}\tau_{456| - + + +}\sigma_{7,12}\nonumber\\
\qquad
=
\tau'_{456| + + + +}\tau'_{236| + + + +}\tau'_{135| + - + +}\tau'_{124| - + + +}\sigma_{7,14}\label{ihte}
\end{gather}
is valid instead of the naive choice of $\tau_{ijk| ++++}$ and $\tau'_{ijk| ++++}$ everywhere.
This is an inhomogeneous version of the twisted tetrahedron equation,
as the maps involved are not always uniform in their sign indices.
The coincident image of \smash{$Y^{(22)}_1, \dots, Y^{(22)}_{17}$} by the
two sides are sign coherent monomials in the initial $Y$-variables
$\bY^{(1)} = (Y_1,\dotsc, Y_{17})$. Their explicit form is available
in~\eqref{y22}.

A natural question is whether there are monomial solutions to the twisted tetrahedron equation
with the signs homogeneously chosen as $(\varepsilon_1,\varepsilon_2,\varepsilon_3,\varepsilon_4)$
\begin{gather}
\tau_{124| \varepsilon_1,\varepsilon_2,\varepsilon_3,\varepsilon_4}
\tau_{135| \varepsilon_1,\varepsilon_2,\varepsilon_3,\varepsilon_4}
\tau_{236| \varepsilon_1,\varepsilon_2,\varepsilon_3,\varepsilon_4}
\tau_{456| \varepsilon_1,\varepsilon_2,\varepsilon_3,\varepsilon_4} \sigma_{7,12}\nonumber
\\
\qquad=
\tau'_{456| \varepsilon_1,\varepsilon_2,\varepsilon_3,\varepsilon_4}
\tau'_{236| \varepsilon_1,\varepsilon_2,\varepsilon_3,\varepsilon_4}
\tau'_{135| \varepsilon_1,\varepsilon_2,\varepsilon_3,\varepsilon_4}
\tau'_{124| \varepsilon_1,\varepsilon_2,\varepsilon_3,\varepsilon_4}\sigma_{7,14}.\label{hteq}
\end{gather}
The answer is given by a direct calculation as follows.
\begin{Proposition} \label{pr:mono}
The monomial part
satisfies the tetrahedron equation \eqref{hteq}
if and only if~${\varepsilon_1=-\varepsilon_4 = -}$, i.e.,
$(\varepsilon_1,\varepsilon_2,\varepsilon_3,\varepsilon_4) \in
\{(-,+,+,+), (-,+,-,+), (-,-,+,+), (-,-,-,+)\}$.
\end{Proposition}
Examples \ref{ex:1} and \ref{ex:2} describe the monomial parts
$\tau_{--++}$ and $\tau_{-+-+}$ explicitly.
Analogous information is supplied for the remaining two cases in Appendix \ref{ap:sup}.

\subsection{Dilogarithm identities}\label{ss:did}

Now we turn to the dilogarithm identities that will be utilized later.
Substitute \eqref{r124a} and \eqref{r456b} into the
left-hand side and the right-hand side of \eqref{TE1} respectively.
The result takes the form
\begin{gather}
\mathrm{Ad}\bigl(
\Psi_q(U_1)^{\varepsilon_1} \dotsm \Psi_q(U_4)^{\varepsilon_4}\bigr)
\tau_{124| \varepsilon_1, \varepsilon_2, \varepsilon_3, \varepsilon_4}
\mathrm{Ad}\bigl(
\Psi_q(U_5)^{\varepsilon_1} \dotsm \Psi_q(U_8)^{\varepsilon_4}\bigr)
\tau_{135| \varepsilon_1, \varepsilon_2, \varepsilon_3, \varepsilon_4}\nonumber
\\
\phantom{\qquad =}{}\times
\mathrm{Ad}\bigl(
\Psi_q(U_9)^{\varepsilon_1} \dotsm \Psi_q(U_{12})^{\varepsilon_4}\bigr)
\tau_{236 | \varepsilon_1, \varepsilon_2, \varepsilon_3, \varepsilon_4}\nonumber
\\
\phantom{\qquad =}{}\times
\mathrm{Ad}\bigl(
\Psi_q(U_{13})^{\varepsilon_1} \dotsm \Psi_q(U_{16})^{\varepsilon_4}\bigr)
\tau_{456| \varepsilon_1, \varepsilon_2, \varepsilon_3, \varepsilon_4}\sigma_{7,12}\nonumber
\\
\qquad {}=
\mathrm{Ad}\bigl(
\Psi_q\bigl(U'_1\bigr)^{\varepsilon_1} \dotsm \Psi_q\bigl(U'_4\bigr)^{\varepsilon_4}\bigr)
\tau'_{456| \varepsilon_1, \varepsilon_2, \varepsilon_3, \varepsilon_4}
\mathrm{Ad}\bigl(
\Psi_q\bigl(U'_5\bigr)^{\varepsilon_1} \dotsm \Psi_q\bigl(U'_8\bigr)^{\varepsilon_4}\bigr)
\tau'_{236| \varepsilon_1, \varepsilon_2, \varepsilon_3, \varepsilon_4}\nonumber
\\
\phantom{\qquad =}{} \times
\mathrm{Ad}\bigl(
\Psi_q\bigl(U'_9\bigr)^{\varepsilon_1} \dotsm \Psi_q\bigl(U'_{12}\bigr)^{\varepsilon_4}\bigr)
\tau'_{135| \varepsilon_1, \varepsilon_2, \varepsilon_3, \varepsilon_4}\nonumber
\\
\phantom{\qquad =}{} \times
\mathrm{Ad}\bigl(
\Psi_q\bigl(U'_{13}\bigr)^{\varepsilon_1} \dotsm \Psi_q\bigl(U'_{16}\bigr)^{\varepsilon_4}\bigr)
\tau'_{124| \varepsilon_1, \varepsilon_2, \varepsilon_3, \varepsilon_4}\sigma_{7,14},\label{psiu}
\end{gather}
where $U_t$ and $U'_t$ ($t=1, \dots, 16$) denote the $Y$-variables
depending on
$(\varepsilon_1, \varepsilon_2, \varepsilon_3, \varepsilon_4)$.
Pushing the monomial parts to the right brings \eqref{psiu} into the form
\begin{gather}
\mathrm{Ad}\bigl(\Psi_q(Z_1)^{\varepsilon_1}\dotsm \Psi_q(Z_{16})^{\varepsilon_{4}}\bigr)
\tau_{124|\varepsilon_1,\varepsilon_2,\varepsilon_3,\varepsilon_4}
\tau_{135|\varepsilon_1,\varepsilon_2,\varepsilon_3,\varepsilon_4}
\tau_{236|\varepsilon_1,\varepsilon_2,\varepsilon_3,\varepsilon_4}
\tau_{456|\varepsilon_1,\varepsilon_2,\varepsilon_3,\varepsilon_4}\sigma_{7,12}\nonumber
\\
\qquad=
\mathrm{Ad}\bigl(\Psi_q\bigl(Z'_1\bigr)^{\varepsilon_1}\dotsm \Psi_q\bigl(Z'_{16}\bigr)^{\varepsilon_{4}}\bigr)
\tau'_{456|\varepsilon_1,\varepsilon_2,\varepsilon_3,\varepsilon_4}
\tau'_{236|\varepsilon_1,\varepsilon_2,\varepsilon_3,\varepsilon_4}\nonumber
\\
\phantom{\qquad=}{}\times
\tau'_{135|\varepsilon_1,\varepsilon_2,\varepsilon_3,\varepsilon_4}
\tau'_{124|\varepsilon_1,\varepsilon_2,\varepsilon_3,\varepsilon_4}\sigma_{7,14},\label{zteq}
\end{gather}
where $Z_i$ and $Z'_i$ are monomials of $Y_1, \dots, Y_{16}$
determined by
\begin{align}\label{zu}
Z_i = \begin{cases} U_i, & i = 1,\dotsc, 4,
\\
 \tau_{124|\varepsilon_1,\varepsilon_2,\varepsilon_3,\varepsilon_4}
 (U_i), & i=5,\dotsc, 8,
\\
\tau_{124|\varepsilon_1,\varepsilon_2,\varepsilon_3,\varepsilon_4}
\tau_{135|\varepsilon_1,\varepsilon_2,\varepsilon_3,\varepsilon_4}
(U_i), & i=9, \ldots, 12,
\\
\tau_{124|\varepsilon_1,\varepsilon_2,\varepsilon_3,\varepsilon_4}
\tau_{135|\varepsilon_1,\varepsilon_2,\varepsilon_3,\varepsilon_4}
\tau_{236|\varepsilon_1,\varepsilon_2,\varepsilon_3,\varepsilon_4}
(U_i), & i=13, \ldots, 16.
\end{cases}
\end{align}
The elements $Z'_i$ are similarly determined from the
right-hand side of \eqref{psiu}.
From \eqref{zteq} and Proposition~\ref{pr:mono}, we deduce
\begin{align}\label{adeq}
\mathrm{Ad}\bigl(\Psi_q(Z_1)^{\varepsilon_1}\cdots \Psi_q(Z_{16})^{\varepsilon_{4}}\bigr)
=\mathrm{Ad}\bigl(\Psi_q\bigl(Z'_1\bigr)^{\varepsilon_1}\cdots \Psi_q\bigl(Z'_{16}\bigr)^{\varepsilon_{4}}\bigr)
\end{align}
for $(\varepsilon_1,\varepsilon_2,\varepsilon_3,\varepsilon_4) \in
\{(-,+,+,+), (-,+,-,+), (-,-,+,+), (-,-,-,+)\}$.
Actually a stronger equality holds.

\begin{Proposition}\label{pr:dilog}
For $(\varepsilon_1,\varepsilon_2,\varepsilon_3,\varepsilon_4) \in \{(-,+,-,+), (-,-,+,+)\}$,
the products of quantum dilogarithms within $\mathrm{Ad}$ in \eqref{adeq}
are well defined formal Laurent series in
the nine $Y$-variables
$Y_3$, $Y_6$, $Y_7$, $Y_8$, $Y_{11}$, $Y_{12}$, $Y_{13}$, $Y_{14}$ and $Y_{15}$.
Moreover, they are equal, i.e.,
\begin{align}\label{pseq}
\Psi_q(Z_1)^{\varepsilon_1}\dotsm \Psi_q(Z_{16})^{\varepsilon_{4}}
=
\Psi_q\bigl(Z'_1\bigr)^{\varepsilon_1}\dotsm \Psi_q\bigl(Z'_{16}\bigr)^{\varepsilon_{4}}.
\end{align}
\end{Proposition}

\begin{proof}
 We show the claim for
 $(\varepsilon_1,\varepsilon_2,\varepsilon_3,\varepsilon_4)
 =(-,-,+,+)$. The case $(-,+,-,+)$ is similar. The data $Z_1,
 \dots, Z_{16}$ for
 $(\varepsilon_1,\varepsilon_2,\varepsilon_3,\varepsilon_4)=(-,-,+,+)$
 is given by
\begin{align}\label{zdL}
\begin{pmatrix}
Y_{14}^{-1} &
qY^{-1}_{13}Y^{-1}_{14} &
q^{-1}Y_{14}Y_{15} &
q^{-2}Y_{8}Y_{14}Y_{15}
\\
q^{2}Y^{-1}_7Y^{-1}_{13}Y^{-1}_{14} &
q^3Y^{-1}_6Y^{-1}_7Y^{-1}_{13}Y^{-1}_{14} &
q^{-3}Y_7Y_8Y_{14}Y_{15} &
q^{-4}Y_3Y_7Y_8Y_{14}Y_{15}
\\
Y^{-1}_{12} &
qY^{-1}_{11}Y^{-1}_{12} &
q^{-1}Y_{12}Y_{13}&
q^{-2}Y_6Y_{12} Y_{13}
\\
Y^{-1}_{7} &
qY^{-1}_6Y^{-1}_7 &
q^{-1}Y_7Y_8 &
q^{-2}Y_3Y_7Y_8
\end{pmatrix},
\end{align}
where the element at $i$th row and the $j$th column from the top left signifies
$Z_{4i+j-4}$.
Similarly,
the data $Z'_1, \dotsc, Z'_{16}$
is given as follows:
\begin{align}\label{zdR}
\begin{pmatrix}
Y^{-1}_{12} &
qY^{-1}_{11}Y^{-1}_{12} &
q^{-1}Y_{12} Y_{13}&
q^{-2}Y_6Y_{12} Y_{13}
\\
Y^{-1}_7 &
qY^{-1}_6Y^{-1}_7 &
q^{-1}Y_7Y_8 &
q^{-2}Y_3Y_7Y_8
\\
Y^{-1}_{12}Y^{-1}_{13}Y^{-1}_{14} &
qY^{-1}_{11}Y^{-1}_{12}Y^{-1}_{13}Y^{-1}_{14} &
q^{-1}Y_{12} Y_{13} Y_{14} Y_{15} &
q^{-2}Y_6Y_{12} Y_{13} Y_{14} Y_{15}
\\
Y_6Y^{-1}_{11}Y^{-1}_{14} &
qY^{-1}_{13}Y^{-1}_{14} &
q^{-1}Y_{14}Y_{15} &
q^{-2}Y_3Y^{-1}_6Y_8Y_{14}Y_{15}
\end{pmatrix}.
\end{align}
Note that $Z'_{13}$ and $Z'_{16}$ in \eqref{zdR} are not sign
coherent. In order to show the well-definedness of the left-hand side of
\eqref{pseq}, expand the 16 $\Psi_q$'s via \eqref{expa} with the
summation variables $n_1, \dots, n_{16} \in \Z_{\ge 0}$. By using
the $q$-commutativity of $Y$-variables, one can arrange each term of
the expansion uniquely~as
\begin{gather}
\Psi_q(Z_1)^{-1}\Psi_q(Z_2)^{-1}\Psi_q(Z_3)\Psi_q(Z_4)
\dotsm \Psi_q(Z_{13})^{-1}\Psi_q(Z_{14})^{-1}\Psi_q(Z_{15})\Psi_q(Z_{16})\nonumber
 \\
 \qquad= \sum_{\mathbf{n} \in (\Z_{\ge 0})^{16}}C(\mathbf{n})
 Y_3^{p_1} Y_6^{p_2} Y_7^{p_3} Y_8^{p_4} Y_{11}^{p_5} Y_{12}^{p_6} Y_{13}^{p_7} Y_{14}^{p_8}Y_{15}^{p_9},\label{c16}
 \end{gather}
 where $C(\mathbf{n})$ is a rational function of $q$ depending on
 $\mathbf{n}=(n_1, \dotsc, n_{16})$. The powers $p_i$'s are given~by
\begin{gather}
p_1 = n_8+n_{16},\qquad p_2 = -n_6+n_{12}-n_{14},\qquad p_3 = -n_{5} - n_{6} - n_{7} + n_{12}-n_{14}-n_{16},
\nonumber\\
p_4 = n_4-n_7+n_{12}+n_{13}-n_{16}, \qquad
p_5 = -n_{10},
 \qquad p_6 = -n_9-n_{10}-n_{11}-n_{13}-n_{15},
\nonumber\\
 p_7 = -n_2-n_5-n_{6}-n_{11}-n_{13}-n_{15},
\qquad
 p_8 = -n_1 -n_2-n_3-n_5-n_6 - n_{11}-n_{13},\nonumber\\
 p_9 = -n_3-n_{11}-n_{13}.
\label{pne}
\end{gather}
The series \eqref{c16} is well defined if the coefficient of the monomial
\[
Y_3^{p_1} Y_6^{p_2} Y_7^{p_3} Y_8^{p_4} Y_{11}^{p_5} Y_{12}^{p_6} Y_{13}^{p_7} Y_{14}^{p_8} Y_{15}^{p_9}
\]
for any given $(p_1,\dotsc, p_9) \in \Z^9$ is finite.
This is shown by checking that there are none or finitely many
$\mathbf{n} \in (\Z_{\ge 0})^{16}$
satisfying the nine equations \eqref{pne}.
This is straightforward.
The well-definedness of the right-hand side is verified in the same manner.

Next we prove \eqref{pseq}.
Write $\Phi$ for $(\text{LHS of \eqref{pseq}})(\text{RHS of
\eqref{pseq}})^{-1}$.
From an argument similar to the proof of \cite[Theorem~3.5]{KN11}, we prove
that $\Phi= c$ where $c$ only depends on $q$ as follows.
We can extend the degenerate exchange matrix $B^{(1)}$ to a non-degenerate one
$\tilde{B}$ which has a twice size as $B^{(1)}$ (see \cite[Example 2.5]{KN11}).
Then, due to the extension theorem \cite[Theorem~4.3]{N11} the
periodicity of the seed $\bigl(B^{(1)},Y^{(1)}\bigr)$ is also that of the seed
$\bigl(\tilde{B},\tilde{Y}\bigr)$. Hence~$\Phi$ commutes with any element of the quantum torus algebra
\smash{$\mathcal{T}\bigl(\tilde{B}\bigr)$}. This means that $\Phi=c$, since~$\tilde{B}$
is nondegenerate.
To determine $c$, we compare the constant terms contained in the
left-hand side and the right-hand side of \eqref{pseq}. For the left-hand side, one looks for
$\mathbf{n} \in (\Z_{\ge 0})^{16}$ such that $p_1 = \dotsb = p_9 = 0$.
It is easy to see that $\mathbf{n}=(0,\dotsc, 0)$ is the only solution
indicating that the constant term of the left-hand side is 1. Similarly, the
constant term of the right-hand side is found to be 1. Therefore, $c=1$.
\end{proof}

\begin{Remark}
 For the two cases
 $(\varepsilon_1,\varepsilon_2,\varepsilon_3,\varepsilon_4)=(-,+,+,+),
 (-,-,-,+)$ in Proposition \ref{pr:mono}, there are infinitely many
 choices of $\mathbf{n} \in (\Z_{\ge 0})^{16}$ satisfying
 \eqref{pne}. Therefore, the simple argument in the above proof is not
 valid.
\end{Remark}

\section[Realization in terms of q-Weyl algebras]{Realization in terms of $\boldsymbol{q}$-Weyl algebras}\label{s:qw}

\subsection[Y-variables and q-Weyl algebras]{$\boldsymbol{Y}$-variables and $\boldsymbol{q}$-Weyl algebras}

Hereafter we also use $\hbar $ which is related to $q$ by
$q=\e^\hbar$. By a $q$-Weyl algebra we mean an associative algebra
generated by $U^{\pm 1}$ and $W^{\pm 1}$ obeying the relations
$U U^{-1} = U^{-1} U= W W^{-1} = W^{-1}W = 1$ and $UW = q WU$.
To each crossing $i$ of the wiring diagram we associate parameters~${\mathscr{P}_i = (a_i, b_i, c_i, d_i, e_i)}$ and canonical variables
$\uu_i$, $\ww_i$ satisfying
\begin{align}
&[\uu_i, \ww_j]= \hbar \delta_{ij}, \qquad [\uu_i, \uu_j]=[\ww_i, \ww_j]=0.
\label{uwh}
\\
&a_i+b_i+c_i+d_i+e_i=0.
\label{ae0}
\end{align}
The exponentials of the canonical variables obey the relations in a
direct product of $q$-Weyl algebras, e.g.,
$\e^{\uu_i} \e^{\ww_j} = q^{\delta_{ij}}\e^{\ww_j}\e^{\uu_i}$. Given a wiring
diagram and the associated symmetric butterfly quiver, we
``parametrize'' the $Y$-variables by the graphical rule explained in
Figure~\ref{fig:para}.

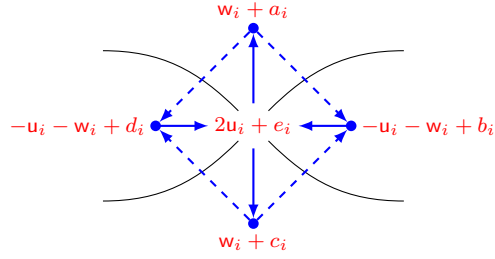
\begin{figure}[t]\centering
\begin{tikzpicture}
\begin{scope}[>=latex,xshift=0pt]
\draw [-] (0,2) to [out = 0, in = 135] (1.8,1.2);
\draw [-] (0,0) to [out = 0, in = -135] (1.8,0.8);
\draw [-] (2.2,0.8) to [out = -45, in = 180] (4,0);
\draw [-] (2.2,1.2) to [out = 45, in = 180] (4,2);
{\color{blue}
\fill (0.7,1) circle(2pt) coordinate(A) node[left]{\color{red} \scriptsize{$-\uu_i-\ww_i+d_i$}};
\fill (3.3,1) circle(2pt) coordinate(B) node[right]{\color{red} \scriptsize{$-\uu_i-\ww_i+b_i$}};
\fill (2,2.3) circle(2pt) coordinate(C) node[above]{\color{red} \scriptsize{$\ww_i+a_i$}};
\fill (2,-0.3) circle(2pt) coordinate(D) node[below]{\color{red} \scriptsize{$\ww_i+c_i$}};
\draw (2,1) node {\color{red} \scriptsize{$2\uu_i+e_i$}};
\draw[->,shorten >=2pt] (2,1.3) -- (C) [thick];
\draw[->,shorten >=2pt] (2,0.7) -- (D) [thick];
\draw[->] (A) -- (1.4,1) [thick];
\draw[->] (B) -- (2.6,1) [thick];
\qdarrow{C}{A}
\qdarrow{C}{B}
\qdarrow{D}{A}
\qdarrow{D}{B}
}
\end{scope}
\end{tikzpicture}
\caption{Graphical rule to parametrize the $Y$-variables in terms of
 $q$-Weyl algebra generators in the vicinity of the crossing $i$
 (center) of the wiring diagram (black). A $Y$-variable situated at
 a vertex (blue circle) of the symmetric butterfly quiver acquires
 factors $\e^{a_i+ \ww_i}$, $\e^{b_i-\uu_i-\ww_i}$, $\e^{c_i+\ww_i}$,
 $\e^{d_i-\uu_i-\ww_i}$ from the neighboring crossing $i$ of the wiring
 diagram if the vertex is located at the north, east, south, west of
 $i$, respectively. A $Y$-variable on the vertex $i$ is
 $\e^{e_i+2\uu_i}$. The ordering of these factors from different $i$'s
 (if any) is inconsequential due to the commutativity of the
 associated canonical variables. }\label{fig:para}
\end{figure}

The claim is that the relation
$Y_r Y_s = q^{2b_{rs}}Y_sY_r$ \eqref{yyq} is satisfied under this parametrization.
To state it formally, let $\mathcal{W}_n$ be the direct product of the
$q$-Weyl algebras generated by~$\e^{\pm \uu_i}$,~$\e^{\pm \ww_i}$ for
$i=1, \dots, n$. Let further $\mathcal{A}_n$ be the non-commuting
fractional field of $\mathcal{W}_n$. Then for $B$ corresponding to
the left diagram in Figure \ref{fig:R123}, we have a morphism
$\phi_\mathrm{SB}\colon \mathcal{Y}(B) \rightarrow \mathcal{A}_3$
defined by
\begin{align}
\phi_\mathrm{SB}\colon\
\begin{cases}
Y_1 \mapsto \e^{c_1+c_3+\ww_1+\ww_3},
& Y_6 \mapsto \e^{b_1-\uu_1-\ww_1},
\\
Y_2 \mapsto \e^{d_3-\uu_3-\ww_3},
& Y_7 \mapsto \e^{d_2+a_3-\uu_2-\ww_2+\ww_3},
\\
Y_3 \mapsto \e^{e_3+2\uu_3},
& Y_8 \mapsto \e^{e_2+2\uu_2},
\\
Y_4 \mapsto \e^{d_1+c_2+b_3-\uu_1-\ww_1+\ww_2-\uu_3-\ww_3},
& Y_9 \mapsto \e^{a_1+b_2+\ww_1-\uu_2-\ww_2},
\\
Y_5 \mapsto \e^{e_1+2\uu_1},
& Y_{10} \mapsto \e^{a_2+\ww_2},
\end{cases}
\label{Yw}
\end{align}
where SB signifies ``symmetric butterfly''. Similarly, for the right
diagram of Figure \ref{fig:R123}, we have a morphism
$\phi'_\mathrm{SB}\colon \mathcal{Y}\bigl(B'\bigr) \rightarrow \mathcal{A}_3$ as
\begin{align}
\phi'_\mathrm{SB}\colon\
\begin{cases}
Y'_1 \mapsto \e^{c_2+\ww_2},
& Y'_6 \mapsto \e^{b_2+c_3-\uu_2-\ww_2+\ww_3},
\\
Y'_2 \mapsto \e^{c_1+d_2+\ww_1-\uu_2-\ww_2},
& Y'_7 \mapsto \e^{d_1-\uu_1-\ww_1},
\\
Y'_3 \mapsto \e^{e_3+2\uu_3},
& Y'_8 \mapsto \e^{e_2+2\uu_2},
\\
Y'_4 \mapsto \e^{b_1+a_2+d_3-\uu_1-\ww_1+\ww_2-\uu_3-\ww_3},
& Y'_9 \mapsto \e^{b_3-\uu_3-\ww_3},
\\
Y'_5 \mapsto \e^{e_1+2\uu_1},
& Y'_{10} \mapsto \e^{a_1+a_3+\ww_1+\ww_3}.
\end{cases}
\label{Ypw}
\end{align}

\begin{Remark}\label{re:c3}
In the parametrization \eqref{Yw}, the centers in Remark \ref{re:c1} take the following values:
\begin{gather*}
\phi_\mathrm{SB}\bigl(Y^{-1}_1Y_7Y_8Y_9Y^2_{10}\bigr) =\e^{a_1-c_1+a_2-c_2+a_3-c_3},
\\
\phi_\mathrm{SB}\bigl(Y_2Y^{-1}_4Y_6Y_{10}\bigr) =\e^{b_1-d_1+a_2-c_2+d_3-b_3},
\\
\phi_\mathrm{SB}\bigl(Y_3Y^2_4Y^{-2}_6Y^2_7Y_8\bigr) = \e^{2d_1-2b_1+2c_2+2d_2+e_2+2a_3+2b_3+e_3},
\\
\phi_\mathrm{SB}\bigl(Y_5Y^2_6Y_8Y^2_9Y^2_{10}\bigr) =q^{-4}\e^{2a_1+2b_1+e_1+2a_2+2b_2+e_2}.
\end{gather*}
In both parametrization \eqref{Yw} and \eqref{Ypw}, the invariants in
Remark \ref{re:c2} take the following forms:
\begin{align}
&\phi_\mathrm{SB}(Y_3Y_8) = \phi'_\mathrm{SB}\bigl(Y'_3Y'_8\bigr) = \e^{e_2+e_3+\uu_2+\uu_3},
\label{Y3Y8-uw}
\\
&\phi_\mathrm{SB}(Y_5Y_8) = \phi'_\mathrm{SB}\bigl(Y'_5Y'_8\bigr) = \e^{e_1+e_2+\uu_1+\uu_2},
\label{Y5Y8-uw}
\\
&\phi_\mathrm{SB}\bigl(Y_1Y_2Y_4Y_6Y_7Y^{-1}_8Y_9Y_{10}\bigr)
=\phi'_\mathrm{SB}\bigl( Y'_1Y'_2Y'_4Y'_6Y'_7Y'^{-1}_8Y'_9Y'_{10}\bigr)
= \e^{-e_1-2e_2-e_3-2\uu_1-4\uu_2-2\uu_3}.\nonumber
\end{align}
The combinations $\uu_1+\uu_2$ and $\uu_2+\uu_3$ will reemerge as
conserved quantities within the delta functions in the matrix elements of the
$R$-matrix in coordinate representations.
See \eqref{Rf1} and~\eqref{Rm0}.
\end{Remark}

\subsection[Extracting R\_123 from R\_123]{Extracting $\boldsymbol{R_{123}}$ from $\boldsymbol{\widehat{R}_{123}}$}

Let us illustrate the action of the monomial part
$\tau_{\varepsilon_1, \varepsilon_2, \varepsilon_3, \varepsilon_4}$
 \eqref{taue} of $\widehat{R}_{123}$ on the canonical variables for the case
 $(\varepsilon_1, \varepsilon_2, \varepsilon_3, \varepsilon_4) = (-,-,+,+)$.
 From Example \ref{ex:1} and \eqref{Yw}--\eqref{Ypw}, we find that
$\tau_{--++}$ is translated into a transformation $\tau^{\uu\ww}_{--++}$
of the canonical variables\footnote{$\tau^{\uu\ww}_{--++}$ is naturally regarded also as
a transformation in $\mathcal{W}_3$ via exponentials.} as
\begin{align}
\tau^{\uu\ww}_{--++}\colon\
\begin{cases}
\uu_1 \mapsto \uu_1+\uu_2-\uu_3+\lambda_0,
&
\ww_1 \mapsto \ww_1+ \lambda_1,
\\
\uu_2 \mapsto \uu_3-\lambda_0,
&
\ww_2 \mapsto \ww_1+\ww_3+\lambda_3,
\\
\uu_3 \mapsto \uu_2+\lambda_0,
&
\ww_3 \mapsto -\ww_1+\ww_2+\lambda_2,
\end{cases}
\label{uwL1}
\end{align}
where $\lambda_r=\lambda_r(\mathscr{P}_1,\mathscr{P}_2,\mathscr{P}_3)$
for $r=0, 1, 2, 3$ is defined, under the condition \eqref{ae0},
by
\begin{gather}
\lambda_0 = \frac{e_2-e_3}{2},\qquad
\lambda_1 = a_2-a_3+b_2-b_3 + \lambda_0,\qquad
\lambda_2 = -a_1-b_2+b_3-\lambda_0,\nonumber\\
\lambda_3 = c_1-c_2+c_3.\label{lad}
\end{gather}

In order to realize \eqref{uwL1} as an adjoint action, we introduce
the group $N_n$ generated by
\begin{gather*}
\e^{\pm \tfrac{1}{\hbar}\uu_i\ww_j}\quad (i \neq j),\qquad
\e^{ \tfrac{a}{\hbar}\uu_i}, \qquad \e^{ \tfrac{a}{\hbar}\ww_i}  \quad (a \in \C) ,\qquad b \in \C^\times
\end{gather*}
with $i, j\in \{1,\dotsc, n\}$. The multiplication is defined by
the (generalized) Baker--Campbell--Hausdorff (BCH) formula and
\eqref{uwh}, which is well defined due to the grading by $\hbar^{-1}$.
Let $\mathfrak{S}_n$ be the symmetric group generated by the
transpositions $\rho_{ij}$ ($i, j \in \{1,\dotsc, n\}$). It acts on
$N_n$ via the adjoint action, inducing permutations of the indices of
the canonical variables. Thus one can form the semi-direct product
$N_n \rtimes \mathfrak{S}_n$, and let it act on $\mathcal{W}_n$ by the
adjoint action.

Now the monomial part $\tau^{\uu\ww}_{--++}$
is described as the adjoint action as follows:
\begin{align}
&\tau^{\uu\ww}_{--++} = \mathrm{Ad}(P_{--++}),
\label{Pbch0}\\
&P_{--++}=
\e^{\tfrac{1}{\hbar}(\uu_3-\uu_2)\ww_1}
\e^{\tfrac{\lambda_0}{\hbar}(-\ww_1-\ww_2+\ww_3)}
\e^{\tfrac{1}{\hbar}(\lambda_1\uu_1+\lambda_2\uu_2+\lambda_3\uu_3)}\rho_{23}
\in N_3 \rtimes \mathfrak{S}_3,
\label{Pbch1}
\end{align}
where $\rho_{23}$ acts trivially on the parameters $\lambda_r$.
Extending the indices and suppressing the sign choice of $--++$ in \eqref{Pbch0} and \eqref{Pbch1}, we introduce
\begin{align}
\tau^{\uu\ww}_{ijk} &= \mathrm{Ad}(P_{ijk}),
\label{tijk}\\
P_{ijk} &=
\e^{\tfrac{1}{\hbar}(\uu_k-\uu_j)\ww_i}
\e^{\tfrac{\lambda_0}{\hbar}(-\ww_i-\ww_j+\ww_k)}
\e^{\tfrac{1}{\hbar}(\lambda_1\uu_i+\lambda_2\uu_j+\lambda_3\uu_k)}\rho_{jk}
\in N_6 \rtimes \mathfrak{S}_6,
\label{pijk}
\end{align}
where $\lambda_r=\lambda_r(\mathscr{P}_i, \mathscr{P}_j, \mathscr{P}_k)$ is given by
\eqref{lad} by replacing the parameters as
$(\mathscr{P}_1 ,\mathscr{P}_2 ,\mathscr{P}_3) \rightarrow
(\mathscr{P}_i ,\mathscr{P}_j ,\mathscr{P}_k)$.
By a straightforward calculation using the BCH formula, one can prove the following.
\begin{Lemma}\label{le:tep}
$P_{ijk}$ satisfies the tetrahedron equation in $N_6 \rtimes \mathfrak{S}_6$ by itself
\begin{align}\label{pte}
P_{124}P_{135}P_{236}P_{456} =
P_{456}P_{236}P_{135}P_{124}.
\end{align}
\end{Lemma}
The fact that $P_{ijk}$ acts on the canonical variables rather than $Y$-variables
has led to the tetrahedron equation {\em without} a twist by permutations.

Define $R_{123}= R(\mathscr{P}_1,\mathscr{P}_2,\mathscr{P}_3)_{123}$ by
\begin{align}
R_{123}
={}&
\Psi_q\bigl(\e^{-d_1-c_2-b_3+\uu_1+\uu_3+\ww_1-\ww_2+\ww_3}\bigr)^{-1}
\Psi_q\bigl(\e^{-d_1-c_2-b_3-e_3+\uu_1-\uu_3+\ww_1-\ww_2+\ww_3}\bigr)^{-1}
\nonumber \\
& \times
\Psi_q\bigl(\e^{d_1+e_1+c_2+b_3+\uu_1-\uu_3-\ww_1+\ww_2-\ww_3}\bigr)\nonumber \\
& \times
\Psi_q\bigl(\e^{d_1+e_1+c_2+e_2+b_3+\uu_1+2\uu_2-\uu_3-\ww_1+\ww_2-\ww_3}\bigr)P_{123}
\label{RL0}
\\
={}&
\Psi_q\bigl(\e^{-d_1-c_2-b_3+\uu_1+\uu_3+\ww_1-\ww_2+\ww_3}\bigr)^{-1}
\Psi_q\bigl(\e^{-d_1-c_2-b_3-e_3+\uu_1-\uu_3+\ww_1-\ww_2+\ww_3}\bigr)^{-1}
\nonumber\\
& \times P_{123}
\Psi_q\bigl(\e^{-b_1-a_2-d_3-e_3+\uu_1-\uu_3+\ww_1-\ww_2+\ww_3}\bigr)\nonumber \\
& \times
\Psi_q\bigl(\e^{-b_1-a_2-d_3+\uu_1+\uu_3+\ww_1-\ww_2+\ww_3}\bigr).
\label{RL1}
\end{align}
Let $\widehat{R}^{\uu\ww}_{123}$ be the cluster transformation $\widehat{R}_{123}$ \eqref{dode}
viewed as the one for the canonical variables~${\{\uu_i, \ww_i\}_{i=1,2,3}}$.
Then from \eqref{Yw}, \eqref{Ypw} and \eqref{Pbch0}, we have
\begin{align}\label{knm}
\widehat{R}^{\uu\ww}_{123} &= \mathrm{Ad}\bigl(R(\mathscr{P}_1,\mathscr{P}_2,\mathscr{P}_3)_{123}\bigr).
\end{align}
Note that the right-hand side is invariant under $R \rightarrow cR$ for any scalar $c$.
A proper normalization of $R$ based on a symmetry consideration will be proposed
in Section \ref{ss:symR}.

Formally the results \eqref{uwL1} and \eqref{knm} may be stated as the commutativity of the diagrams
\begin{align*}
\begin{CD}
\mathcal{Y}\bigl(B'\bigr) @> {\tau_{--++}}>> \mathcal{Y}(B) \\
@V{\phi'_\mathrm{SB}}VV @VV{\phi_\mathrm{SB}}V \\
\mathcal{A}_3 @>{\tau^{\uu\ww}_{--++}}>> \mathcal{A}_3,
\end{CD}
\qquad\qquad
\begin{CD}
\mathcal{Y}\bigl(B'\bigr) @> {\widehat{R}_{123}}>> \mathcal{Y}(B) \\
@V{\phi'_\mathrm{SB}}VV @VV{\phi_\mathrm{SB}}V \\
\mathcal{A}_3 @>{\widehat{R}^{\uu\ww}_{123}}>> \mathcal{A}_3.
\end{CD}
\end{align*}

Extending \eqref{RL0} and \eqref{RL1}, we introduce
$R_{ijk}=R(\mathscr{P}_i ,\mathscr{P}_j ,\mathscr{P}_k)_{ijk}$
by
\begin{align}
R_{ijk}
={}&
\Psi_q\bigl(\e^{-d_i-c_j-b_k+\uu_i+\uu_k+\ww_i-\ww_j+\ww_k}\bigr)^{-1}
\Psi_q\bigl(\e^{-d_i-c_j-b_k-e_k+\uu_i-\uu_k+\ww_i-\ww_j+\ww_k}\bigr)^{-1}
\nonumber \\
& \times
\Psi_q\bigl(\e^{d_i+e_i+c_j+b_k+\uu_i-\uu_k-\ww_i+\ww_j-\ww_k}\bigr)\nonumber \\
& \times
\Psi_q\bigl(\e^{d_i+e_i+c_j+e_j+b_k+\uu_i+2\uu_j-\uu_k-\ww_i+\ww_j-\ww_k}\bigr)P_{ijk}
\label{rijk0}\\
={}&
\Psi_q\bigl(\e^{-d_i-c_j-b_k+\uu_i+\uu_k+\ww_i-\ww_j+\ww_k}\bigr)^{-1}
\Psi_q\bigl(\e^{-d_i-c_j-b_k-e_k+\uu_i-\uu_k+\ww_i-\ww_j+\ww_k}\bigr)^{-1}
\nonumber\\
& \times P_{ijk}
\Psi_q\bigl(\e^{-b_i-a_j-d_k-e_k+\uu_i-\uu_k+\ww_i-\ww_j+\ww_k}\bigr)
\Psi_q\bigl(\e^{-b_i-a_j-d_k+\uu_i+\uu_k+\ww_i-\ww_j+\ww_k}\bigr),
\label{rijk}
\end{align}
where $P_{ijk}$ is given by \eqref{pijk}.
Now we state the main result of the paper.

\begin{Theorem}\label{th:main}
$R(\mathscr{P}_i, \mathscr{P}_j, \mathscr{P}_k)_{ijk} $ satisfies the tetrahedron equation
\begin{gather}
 R(\mathscr{P}_1, \mathscr{P}_2, \mathscr{P}_4)_{124}
R(\mathscr{P}_1, \mathscr{P}_3, \mathscr{P}_5)_{135}
R(\mathscr{P}_2, \mathscr{P}_3, \mathscr{P}_6)_{236}
R(\mathscr{P}_4, \mathscr{P}_5, \mathscr{P}_6)_{456}\nonumber
\\
\qquad= R(\mathscr{P}_4, \mathscr{P}_5, \mathscr{P}_6)_{456}
R(\mathscr{P}_2, \mathscr{P}_3, \mathscr{P}_6)_{236}
R(\mathscr{P}_1, \mathscr{P}_3, \mathscr{P}_5)_{135}
R(\mathscr{P}_1, \mathscr{P}_2, \mathscr{P}_4)_{124}.\label{tela}
\end{gather}
\end{Theorem}
\begin{proof}
 Consider the dilogarithm identity \eqref{pseq} with
 $(\varepsilon_1,\varepsilon_2,\varepsilon_3,\varepsilon_4) =(-, -,
 +, +)$ in terms of canonical variables\footnote{The signs
 $\varepsilon_1$, $\varepsilon_2$, $\varepsilon_3$, $\varepsilon_4$
 are actually $-$, $-$, $+$, $+$, but for clarity, they are left as
 they are.}
\begin{align}\label{zt16}
\Psi_q\bigl(\tilde{Z}_1\bigr)^{\varepsilon_1}\dotsm \Psi_q\bigl({\tilde Z}_{16}\bigr)^{\varepsilon_4}
=
\Psi_q\bigl(\tilde{Z'}_1\bigr)^{\varepsilon_1}\dotsm \Psi_q\bigl(\tilde{Z'}_{16}\bigr)^{\varepsilon_4}.
\end{align}
Here $\tilde{Z}_i = \tilde{\phi}_\mathrm{SB}(Z_i)$ and
$\tilde{Z'}_i =\tilde{\phi}'_\mathrm{SB}(Z'_i)$ with $Z_i$ and $Z'_i$
given in \eqref{zdL} and \eqref{zdR}. The morphisms~$\tilde{\phi}_\mathrm{SB}$ and $\tilde{\phi}'_\mathrm{SB}$ are natural
generalizations of \eqref{Yw} and \eqref{Ypw}. They send the
$Y$-variables to $\mathcal{W}_6$ according to the rule in Figure
\ref{fig:para} applied to the bottom right diagram in Figure
\ref{fig:bTE}.\footnote{The map $\phi_\mathrm{SB}$ does not spoil the
 well-definedness of the expansion like \eqref{c16} with respect to
 $\e^{\uu_i}$ and $\e^{\ww_i}$ since it preserves the rank of the quantum
 torus generated by $Y_i$ ($i=3, 6 , 7 , 8 , 11 , 12 , 13 ,
 14 , 15$).}
Multiplication of \eqref{pte} to \eqref{zt16} from the right leads to
\begin{gather}
\Psi_q\bigl(\tilde{Z}_1\bigr)^{\varepsilon_1}\dotsm \Psi_q\bigl({\tilde Z}_{16}\bigr)^{\varepsilon_4}
P_{124}P_{135}P_{236}P_{456}\nonumber\\
\qquad=
\Psi_q\bigl(\tilde{Z'}_1\bigr)^{\varepsilon_1}\dotsm \Psi_q\bigl(\tilde{Z'}_{16}\bigr)^{\varepsilon_4}
P_{456}P_{236}P_{135}P_{124}.\label{zp16}
\end{gather}
Let us consider the left-hand side.
It is obviously equal to
\begin{gather}
\Psi_q\bigl(\tilde{Z}_1\bigr)^{\varepsilon_1}\dotsm \Psi_q\bigl({\tilde Z}_{4}\bigr)^{\varepsilon_4}
P_{124}
P^{-1}_{124}
\Psi_q\bigl(\tilde{Z}_5\bigr)^{\varepsilon_1}\dotsm \Psi_q\bigl({\tilde Z}_{8}\bigr)^{\varepsilon_4}
P_{124} \cdot P_{135}\nonumber
\\
\qquad \times
P^{-1}_{135} P^{-1}_{124}
\Psi_q\bigl(\tilde{Z}_9\bigr)^{\varepsilon_1}\dotsm \Psi_q\bigl({\tilde Z}_{12}\bigr)^{\varepsilon_4}
P_{124} P_{135} \cdot P_{236}\nonumber
\\
\qquad \times
P^{-1}_{236} P^{-1}_{135} P^{-1}_{124}
\Psi_q\bigl(\tilde{Z}_{13}\bigr)^{\varepsilon_1}\dotsm \Psi_q\bigl({\tilde Z}_{16}\bigr)^{\varepsilon_4}
P_{124} P_{135} P_{236} \cdot P_{456}.\label{tochu1}
\end{gather}
On the other hand from \eqref{tijk} and the image of $\eqref{zu}$ by $\tilde{\phi}_\mathrm{SB}$, we know
\begin{align*}
{\tilde U}_i = \begin{cases}
\tilde{Z}_i, & i=1,\dotsc, 4,
\\
P_{124}^{-1} \tilde{Z}_i P_{124}, & i=5,\dotsc, 8,
\\
P_{135}^{-1}P_{124}^{-1} \tilde{Z}_i P_{124}P_{135}, & i=9,\dotsc, 12,
\\
P_{236}^{-1}P_{135}^{-1}P_{124}^{-1} \tilde{Z}_i
P_{124}P_{135} P_{236}, & i=13,\dotsc, 16,
\end{cases}
\end{align*}
where $\tilde{U}_i = \tilde{\phi}_\mathrm{SB}(U_i)$.
Thus \eqref{tochu1} is cast into the form
\begin{gather*}
\Psi_q\bigl(\tilde{U}_1\bigr)^{\varepsilon_1}\cdots \Psi_q\bigl({\tilde U}_{4}\bigr)^{\varepsilon_4}
P_{124}
\Psi_q\bigl(\tilde{U}_5\bigr)^{\varepsilon_1}\cdots \Psi_q\bigl({\tilde U}_{8}\bigr)^{\varepsilon_4}
P_{135}
\\
\qquad\times
\Psi_q\bigl(\tilde{U}_9\bigr)^{\varepsilon_1}\cdots \Psi_q\bigl({\tilde U}_{12}\bigr)^{\varepsilon_4}
P_{236}
\Psi_q\bigl(\tilde{U}_{13}\bigr)^{\varepsilon_1}\cdots \Psi_q\bigl({\tilde U}_{16}\bigr)^{\varepsilon_4}
P_{456}.
\end{gather*}
This is identified with the left-hand side of \eqref{tela} for \eqref{rijk0}.
The right-hand side of \eqref{tela} is similarly derived from that in \eqref{zp16}.
\end{proof}

The monomial part $\tau_{\ve_1,\ve_2,\ve_3,\ve_4}$ which admits the
adjoint action description as \eqref{Pbch0} can be searched in the
same manner as explained around \cite[equation~(4.27)]{IKT2}. We find
that $(\ve_1,\ve_2,\ve_3,\allowbreak\ve_4)=(-,-,+,+)$ and $(-,+,-,+)$ are the
only such cases. The formulas corresponding to $(-,+,-,+)$ are
summarized in Appendix \ref{app:ruw}. They are obtained from
$(-,-,+,+)$ case by the interchange of the parameters
$a_i \leftrightarrow a_{4-i}$, $c_i \leftrightarrow c_{4-i}$,
$e_i \leftrightarrow e_{4-i}$ and $b_i \leftrightarrow d_{4-i}$ which
is compatible with \eqref{ae0}.

\subsection[Symmetries of R-matrices]{Symmetries of $\boldsymbol{R}$-matrices}\label{ss:symR}

Let $(B,Y)$ and $\bigl(B',Y'\bigr)$ be the quantum $y$-seeds corresponding to
the quivers in Figure \ref{fig:3.2}\,(a) and (b), respectively, and
\smash{$\bigl(\widetilde{B}, \widetilde{Y}\bigr) = \mu_8 \mu_5 \mu_3(B,Y)$} and
\smash{$\bigl(\widetilde{B}', \widetilde{Y}'\bigr) = \mu_8 \mu_5 \mu_3\bigl(B',Y'\bigr)$}. Note
that $\widetilde{B} = -B$ and $\widetilde{B}' = -B'$. Define
isomorphisms
\begin{gather*}
 \alpha
 := \sigma_{2,6} \sigma_{3,5} \sigma_{7,9}
\colon\ \mathcal{Y}(B) \to \mathcal{Y}(B),
 \qquad
 \alpha'
 := \sigma_{2,6} \sigma_{3,5} \sigma_{7,9}
\colon\ \mathcal{Y}\bigl(B'\bigr) \to \mathcal{Y}\bigl(B'\bigr),
 \\
 \beta
 := \bigl(Y \to Y'\bigr) \sigma_{1,10} \sigma_{2,9} \sigma_{6,7}
\colon\ \mathcal{Y}(B) \to \mathcal{Y}\bigl(B'\bigr),
 \\
 \beta'
 := \bigl(Y' \to Y\bigr) \sigma_{1,10} \sigma_{2,9} \sigma_{6,7}
\colon\ \mathcal{Y}\bigl(B'\bigr) \to \mathcal{Y}(B),
 \\
 \gamma
 := \bigl(q \to q^{-1}\bigr) \bigl(\widetilde{Y} \to Y^{-1}\bigr) \mu_3^* \mu_5^* \mu_8^*
\colon\ \mathcal{Y}(B) \to \mathcal{Y}(B),
 \\
 \gamma'
 := \bigl(q \to q^{-1}\bigr) \bigl(\widetilde{Y}' \to Y^{\prime-1}\bigr) \mu_3^* \mu_5^* \mu_8^*
 \colon\ \mathcal{Y}\bigl(B'\bigr) \to \mathcal{Y}\bigl(B'\bigr),
\end{gather*}
where $(x \to y)$ denotes the operation of replacing $x$ with $y$.

\begin{Proposition}
 The cluster transformation
 $\widehat{R}_{123}\colon \mathcal{Y}\bigl(B'\bigr) \to \mathcal{Y}(B)$
 satisfies
 \begin{align}
 \label{sym-alpha}
 \alpha\widehat{R}_{123} &= \widehat{R}_{123} \alpha',
 \\
 \label{sym-beta}
 \beta\widehat{R}_{123} &= \widehat{R}_{123}^{-1} \beta',
 \\
 \label{sym-gamma}
 \gamma\widehat{R}_{123} &= \widehat{R}_{123} \gamma'.
 \end{align}
 \end{Proposition}

\begin{proof}
 The first equation is a consequence of the reflection symmetries of
 the quivers in Figure~\ref{fig:3.2} and the mutation sequence
 for $\widehat{R}$ about the vertical axis going through vertices $1$
 and $10$. The mutation sequence is symmetric because vertices $3$
 and $5$ are disconnected in the relevant quiver.

 The second equation can be understood by turning Figure
 \ref{fig:mus} upside down. Since
\[
\sigma_{3,5} \sigma_{4,8} \mu_8 \mu_5 \mu_3 \mu_4 = \mu_4 \mu_3
 \mu_5 \mu_8 \sigma_{3,5} \sigma_{4,8}, \]
 the mutation sequence going
 from $B^{(1)}$ to $B^{(6)}$ in the upside down figure is the same as
 the reverse sequence going from $B^{(6)}$ to $B^{(1)}$ in the
 original figure, with labels $2$ and~$9$, labels~$6$ and~$7$, and labels $1$
 and $10$ swapped.

 To show the third equation, we use the fact that
 \[
 \sigma_{3,5} \sigma_{4,8} \mu_8 \mu_5 \mu_3 \mu_4
 \mu_8 \mu_5 \mu_3(B,y)
 =
 \mu_8 \mu_5 \mu_3
 \sigma_{3,5} \sigma_{4,8} \mu_8 \mu_5 \mu_3 \mu_4(B,y),
 \]
 as one can check by direct calculation. By Theorem \ref{thm:piv},
 this implies the following equality between maps from
 $\mathcal{Y}\bigl(B'\bigr)$ to $\mathcal{Y}\bigl(\widetilde{B}\bigr)$
 \[
 \mu_3^* \mu_5^* \mu_8^*
 (\mu_4^* \mu_3^* \mu_5^* \mu_8^* \sigma_{3,5} \sigma_{4,8})
 =
 (\mu_4^* \mu_3^* \mu_5^* \mu_8^* \sigma_{3,5} \sigma_{4,8})
 \mu_3^* \mu_5^* \mu_8^*.
 \]
 Multiplying both sides with
 $\bigl(q \to q^{-1}\bigr) \bigl(\widetilde{Y} \to Y^{-1}\bigr)$, we get
 \[
 \gamma\widehat{R}
 =
 \bigl(q \to q^{-1}\bigr) \bigl(\widetilde{Y} \to Y^{-1}\bigr)
 (\mu_4^* \mu_3^* \mu_5^* \mu_8^* \sigma_{3,5} \sigma_{4,8})
 \mu_3^* \mu_5^* \mu_8^*.
 \]
 This is the desired equation because $\widetilde{Y}'_i$ transforms
 under mutations in the same way as $Y_i^{\prime-1}$, except that $q$
 appearing in the formula is replaced by $q^{-1}$.
\end{proof}

Let $\alpha^{\uu\ww}$, $\beta^{\uu\ww}$, $\gamma^{\uu\ww}$ be $\alpha$, $\beta$,
$\gamma$ expressed in terms of the canonical variables
$\{\uu_i, \ww_i\}_{i=1,2,3}$. In other words, these are operators such
that $\alpha^{\uu\ww} \circ \phi_\mathrm{SB} = \phi_\mathrm{SB} \circ \alpha$
and
$\alpha^{\uu\ww} \circ \phi'_\mathrm{SB} = \phi'_\mathrm{SB} \circ \alpha$, etc.
Explicitly, they act on the parameters and the canonical variables as follows:
\begin{gather}
 \alpha^{\uu\ww}\colon\
 a_1 \leftrightarrow a_3, \qquad
 b_1 \leftrightarrow d_3, \qquad
 c_1 \leftrightarrow c_3, \qquad
 d_1 \leftrightarrow b_3, \qquad
 e_1 \leftrightarrow e_3, \qquad
 b_2 \leftrightarrow d_2, \qquad
\nonumber \\
 \phantom{ \alpha^{\uu\ww}\colon\ }{} \uu_1 \leftrightarrow \uu_3, \qquad
 \ww_1 \leftrightarrow \ww_3,
\label{uwal} \\
 \beta^{\uu\ww}\colon\
 a_i \leftrightarrow c_i, \qquad
 b_i \leftrightarrow d_i,\nonumber
 \\
 \gamma^{\uu\ww}\colon\
 q \mapsto q^{-1}, \qquad
 a_i \leftrightarrow c_i, \qquad
 b_i \leftrightarrow d_i, \qquad
 \e^{-\ww_i} \mapsto \e^{\ww_i + a_i + c_i} \bigl(1 + q \e^{2\uu_i + e_i}\bigr).\nonumber
\end{gather}
Acting on \eqref{sym-alpha}, \eqref{sym-beta} and \eqref{sym-gamma}
with $\phi_\mathrm{SB}$, we obtain the following relations that hold
in~$\phi_\mathrm{SB}\bigl(\mathcal{Y}\bigl(B'\bigr)\bigr)$:
\begin{gather}
 \alpha^{\uu\ww} \widehat{R}^{\uu\ww}_{123} = \widehat{R}^{\uu\ww}_{123} \alpha^{\uu\ww},
 \qquad
 \beta^{\uu\ww} \widehat{R}^{\uu\ww}_{123} = \widehat{R}^{\uu\ww -1}_{123} \beta^{\uu\ww}, \label{alphauw-Rhat}
 \\
 \gamma^{\uu\ww} \widehat{R}^{\uu\ww}_{123} = \widehat{R}^{\uu\ww}_{123} \gamma^{\uu\ww}.
 \label{gammauw-Rhat}
\end{gather}

The symmetry \eqref{alphauw-Rhat} can also be deduced from the formula
\eqref{RL0} for $R_{123}$ and its counterpart~\eqref{RL3} for the sign
choice $(-,+,-,+)$, which are mapped to each other by $\alpha^{\uu\ww}$.
In fact, not only the adjoint action of $R_{123}$ but $R_{123}$ itself
enjoys the symmetries $\alpha^{\uu\ww}$ and $\beta^{\uu\ww}$.
\begin{Proposition}
 Let $f$ be a function of the parameters
 $(a_i, b_i, c_i, d_i)_{i=1,2,3}$ such that
 \begin{align} 
& \alpha^{\uu\ww}(f) f^{-1}
= \exp\biggl(-\frac{1}{4\hbar}
 (e_1 - e_3)(a_3 - 2c_2 + c_3 - d_1 + d_3 + a_1 + c_1 - b_3 + b_1)\biggr),
\nonumber
 \\
 &\beta^{\uu\ww} (f) f
=
 \exp\biggl(
 \frac{1}{2\hbar}
 (e_2 - e_3)(a_1 - a_2 + a_3 + c_1 - c_2 + c_3)
 \biggr).\label{alphauw-f}
 \end{align}
 As an operator in either the $u$-diagonal representation or the
 $w$-diagonal representation introduced in Section {\rm\ref{s:me}},
 $R_{123}$ satisfies
 \begin{align}
 \alpha^{\uu\ww} (f R_{123}) &= f R_{123},
 \label{alphauw-R}
 \\
 \beta^{\uu\ww} (f R_{123}) &= (f R_{123})^{-1}.
 \label{betauw-R}
 \end{align}
\end{Proposition}

\begin{proof}
 The symmetry \eqref{alphauw-R} under $\alpha^{\uu\ww}$ can be seen from
 the formulas for the matrix elements of $R_{123}$ in the
 $u$-diagonal representation and the $w$-diagonal representation,
 obtained in Theorems~\ref{th:re} and~\ref{th:ew},
 respectively. In both cases, the only part of the matrix elements
 that is not manifestly invariant under $\alpha^{\uu\ww}$ is the factor
 $\e^{-C_5^2}$. (See Remark \ref{re:al}.) The symmetry~\eqref{betauw-R} under $\beta^{\uu\ww}$ actually holds at the level of
 an element of $N_3 \rtimes \mathfrak{S}_3$, as one can check by
 calculating~${\beta^{\uu\ww}(R_{123}) R_{123}}$, say using the expression
 \eqref{RL1} for $R_{123}$.
\end{proof}

An example of a function that satisfies the above two conditions is
\[
 f
 =\exp\left(\frac{1}{4\hbar}
 (e_2 - e_3)(a_1 + a_3 + c_1 - 2c_2 + c_3 + b_1 - b_3 - d_1 + d_3)\right).
\]

\section[Matrix elements of R]{Matrix elements of $\boldsymbol{R}$}\label{s:me}

In this section, we derive explicit formulas for the elements of the
$R$-matrix given in \eqref{RL1} and~\eqref{pijk} in some infinite
dimensional representations of the $q$-Weyl algebra. When the overall
normalization is not our primary concern, we will use notation such as
\smash{$A^{n_1, n_2, n_3}_{n'_1, n'_2, n'_3} \equiv B^{n_1, n_2, n_3}_{n'_1,
 n'_2, n'_3}$} to imply
\smash{$A^{n_1, n_2, n_3}_{n'_1, n'_2, n'_3} = c B^{n_1, n_2, n_3}_{n'_1,
 n'_2, n'_3}$} for some $c$ that does not depend on the indices $n_1$,
$n_2$, $n_3$, $n'_1$, $n'_2$, $n'_3$. When discussing the symmetry of
the elements, it is important to consider whether $c$ depends on the
parameters $C_1, \dots, C_8$ in \eqref{C8} or not. In such a
circumstance, we will address the dependence case by case. For
simplicity, $\bigl(z;q^2\bigr)_m$ will be denoted as $(z)_m$.

\subsection{Parameters}\label{ss:para}
Recall that we have the parameters $\mathscr{P}_i=(a_i,b_i,c_i,d_i,e_i)$
satisfying \eqref{ae0} attached to each vertex~$i$ of the quiver.
In what follows, we will also use the following:
\begin{gather}
C_1 =\frac{1}{2}(b_1-b_2+c_1-c_3+d_2-d_3),
\qquad
C_2 = -\frac{1}{2}(c_1-c_2+c_3+b_1+a_2+d_3),\nonumber
\\
C_3 =\frac{1}{2}(c_1-c_2+c_3),
\qquad
C_4 = \frac{1}{2}(a_2+b_2+c_2+d_2),
\qquad
C_5 = \frac{1}{2}(a_3-c_2+c_3-d_1+d_3),\nonumber
\\
C_6 = \frac{1}{2}(a_1+b_1-b_3+c_1-c_2),
\qquad
C_7 = \frac{1}{2}(-d_1-c_2-b_3),\nonumber
\\
C_8 = \frac{1}{2}(a_1+a_3+b_1+c_1-c_2+c_3+d_3).\label{C8}
\end{gather}
They satisfy the relation
\begin{equation}\label{cz}
C_5+C_6-C_7-C_8=0.
\end{equation}
The parameters $e_i$ in \eqref{ae0} and $\lambda_i$ in \eqref{lad} are expressed as
\begin{gather}
\lambda_0 = -C_4+C_5-C_7,
\qquad
\lambda_1 = -C_1-C_2-C_5+C_7,
\qquad
e_1 = 2(C_7-C_6),
\qquad
g_1 = \frac{e_1}{2\hbar},\nonumber
\\
\lambda_2 = C_1-C_2-C_6-C_8,
\qquad
e_2 = -2C_4,
\qquad
g_2 = \frac{e_2}{2\hbar},\nonumber
\\
\lambda_3 = 2C_3,
\quad
e_3 = 2(C_7-C_5),
\qquad
g_3 = \frac{e_3}{2\hbar},\label{lac}
\end{gather}
where we have also introduced $g_1$, $g_2$ and $g_3$.
Now the $R$-matrix $R=R_{123}$ \eqref{RL1} is expressed as
\begin{align}
R ={}&\Psi_q\bigl(\e^{2C_7+\uu_1+\uu_3+\ww_1-\ww_2+\ww_3}\bigr)^{-1}
\Psi_q\bigl(\e^{2C_5+\uu_1-\uu_3+\ww_1-\ww_2+\ww_3}\bigr)^{-1}\nonumber
\\
& \times P
\Psi_q\bigl(\e^{\alpha_6+\uu_1-\uu_3+\ww_1-\ww_2+\ww_3}\bigr)
\Psi_q\bigl(\e^{\alpha_8+\uu_1+\uu_3+\ww_1-\ww_2+\ww_3}\bigr),\label{RL12}
\end{align}
where $\alpha_6= -b_1-a_2-d_3-e_3$ and
$\alpha_8 = -b_1-a_2-d_3$.
They are also expressed as
\begin{equation}\label{C68}
\begin{pmatrix}
\alpha_6
\\
\alpha_8
\end{pmatrix}
=
\begin{pmatrix}
-\lambda_1-\lambda_2+\lambda_3-2C_6
\\
-\lambda_1-\lambda_2+\lambda_3-2C_8
\end{pmatrix}
=
\begin{pmatrix}
2 C_2 + 2 C_3 + 2 C_5 - 2 C_7
\\
2 C_2 + 2 C_3
\end{pmatrix}.
\end{equation}
The operator $P=P_{123}$ \eqref{pijk} reads
\begin{align}\label{pp123}
P=
\e^{\tfrac{1}{\hbar}(\uu_3-\uu_2)\ww_1}
\e^{g_{23}(-\ww_1-\ww_2+\ww_3)}
\e^{\tfrac{1}{\hbar}(\lambda_1\uu_1+\lambda_2\uu_2+\lambda_3\uu_3)}\rho_{23},
\end{align}
where we have set $g_{23}=g_2-g_3$, which is equal to
$\frac{\lambda_0}{\hbar}= \frac{-C_4+C_5-C_7}{\hbar}$.
The formula \eqref{RL12} is the same with \eqref{Rint}.
We note that the transformation $\alpha^{\uu\ww}$ \eqref{uwal} is expressed as
\begin{align}\label{alC}
\alpha^{\uu\ww}\colon\ C_1 \rightarrow -C_1,\qquad C_5 \leftrightarrow C_6,
\qquad \text{other $C_i$'s are invariant},
\end{align}
as far as the parameters are concerned.

\begin{Remark}\label{re:X}
By shifting the canonical variables $\uu_i$ and $\ww_i$, one can set $c_i = d_i = 0$ without loss of generality.
See Figure \ref{fig:para}. In this parametrization, the constraint $C_3 = C_1 + C_2 + C_4 = 0$ holds in addition to \eqref{cz}.
Consequently, our solution \eqref{Rint} involves five parameters, in addition to the parameter~$q$.
\end{Remark}

\subsection[Elements of R in u-diagonal representation]{Elements of $\boldsymbol{R}$ in $\boldsymbol{\uu}$-diagonal representation}\label{ss:ud}

Let $F= \bigoplus_{n_1,n_2,n_3 \in \Z}\C |n_1,n_2,n_3\rangle$
and $F^\ast = \bigoplus_{n_1,n_2,n_3 \in \Z}\C \langle n_1,n_2,n_3|$ be the infinite dimensional spaces.
Define the representations of the direct product of the $q$-Weyl algebra
(see the explanation around \eqref{uwh}) on them by
\begin{gather}
\e^{\uu_k}|\mathbf{n}\rangle = i q^{n_k+\tfrac{1}{2}}|\mathbf{n}\rangle,
\qquad
\e^{\ww_k}|\mathbf{n}\rangle = |\mathbf{n} + \mathbf{e}_k\rangle,
\qquad
\langle \mathbf{n}|\e^{\uu_k}=\langle \mathbf{n}| i q^{n_k+\tfrac{1}{2}},\nonumber
\\
\langle \mathbf{n}|\e^{\ww_k} = \langle \mathbf{n} - \mathbf{e}_k|\label{qrep}
\end{gather}
for $k=1,2,3$. Here $|n_1,n_2,n_3\rangle$ (resp.\
$\langle n_1,n_2, n_3|$) is simply denoted by $|\mathbf{n}\rangle$
(resp.\ $\langle \mathbf{n}|$) with $\mathbf{n} \in \Z^3$, and
$\mathbf{e}_1=(1,0,0)$, $\mathbf{e}_2=(0,1,0)$, $\mathbf{e}_3=(0,0,1)$. The
dual pairing of $F$ and $F^\ast$ is defined by~${\langle \mathbf{n}|\mathbf{n}'\rangle = \delta_{\mathbf{n}, \mathbf{n}'}}$,
which satisfies
$(\langle \mathbf{n}|X) |\mathbf{n}'\rangle = \langle \mathbf{n}|(X|{\bf
 n}'\rangle)$.

In the rest of this subsection, we assume that $g_i$'s defined in \eqref{lac} satisfy the condition
\begin{align}\label{gas}
g_i \in \Z\quad (i=1,2,3).
\end{align}

\begin{Theorem}\label{th:re}
Under the assumption \eqref{gas},
the element \smash{$R^{n_1, n_2, n_3}_{n'_1, n'_2, n'_3} :=
\langle n_1, n_2, n_3| R \bigl| n'_1, n'_2, n'_3\bigr\rangle$} of the $R$-matrix \eqref{RL12} is given by
\begin{align}
R^{n_1, n_2, n_3}_{n'_1, n'_2, n'_3}
={}& \kappa \delta^{n_1+n_2}_{n'_1+n'_2} \delta^{n_2+n_3}_{n'_2+n'_3}
\e^{\lambda_1 n'_1+\lambda_2n'_3+\lambda_3 n'_2}
\bigl(\e^{2C_5}q^{n_1+g_3}\bigr)^{n_3+g_3}q^{n'_2+g_2}\nonumber
\\
& \times
\oint\frac{{\rm d}z}{2\pi {\rm i} z^{n'_2+g_2+1}}
\frac{\bigl(-z\e^{-2C_8}q^{2+n'_1+n'_3}\bigr)_\infty\bigl(-z\e^{-2C_7}q^{-n_1-n_3}\bigr)_\infty}
{\bigl(-z\e^{-2C_6}q^{n'_1-n'_3}\bigr)_\infty\bigl(-z\e^{-2C_5}q^{n'_3-n'_1}\bigr)_\infty},\label{Rf1}
\end{align}
where
\smash{$\kappa = \e^{\tfrac{1}{2}(\tfrac{ {\rm i} \pi}{\hbar}+1)(\lambda_1+\lambda_2+\lambda_3)}$}.
The integral encircles $z=0$ anti-clockwise picking the residue at the origin.
\end{Theorem}

\begin{proof}
From the expansions \eqref{Psiq} and
the commutation relations $\e^{\uu_i} \e^{\ww_j}=q^{\delta_{ij}}\e^{\ww_j}\e^{\uu_i}$,
$\e^{\uu_i}\e^{\uu_j}=\e^{\uu_j}\e^{\uu_i}$ and $\e^{\ww_i}\e^{\ww_j}=\e^{\ww_j}\e^{\ww_i}$, one has
\begin{gather*}
\Psi_q\bigl(\e^{2C_7+\uu_1+\uu_3+\ww_1-\ww_2+\ww_3}\bigr)^{-1}=\sum_{k\ge 0}\frac{1}{\bigl(q^2\bigr)_k}
\bigl(\e^{2C_7+\uu_1+\uu_3}\bigr)^k\bigl(\e^{\ww_1-\ww_2+\ww_3}\bigr)^k,
\\
\Psi_q\bigl(\e^{2C_5+\uu_1-\uu_3+\ww_1-\ww_2+\ww_3}\bigr)^{-1} = \sum_{l \ge 0}\frac{q^{l^2}}{\bigl(q^2\bigr)_l}
\bigl(\e^{2C_5+\uu_1-\uu_3}\bigr)^l\bigl(\e^{\ww_1-\ww_2+\ww_3}\bigr)^l,
\\
\Psi_q(\e^{\alpha_6+\uu_1-\uu_3+\ww_1-\ww_2+\ww_3}\bigr) = \sum_{r\ge 0}\frac{q^r}{\bigl(q^2\bigr)_r}
\bigl(\e^{\ww_1-\ww_2+\ww_3}\bigr)^r\bigl(-\e^{\alpha_6+\uu_1-\uu_3}\bigr)^r,
\\
\Psi_q\bigl(\e^{\alpha_8+\uu_1+\uu_3+\ww_1-\ww_2+\ww_3}\bigr)
 = \sum_{s\ge 0}\frac{q^{s^2+s}}{\bigl(q^2\bigr)_s}
\bigl(\e^{\ww_1-\ww_2+\ww_3}\bigr)^s\bigl(-\e^{\alpha_8+\uu_1+\uu_3}\bigr)^s.
\end{gather*}
By utilizing them and \eqref{qrep}, we get
\begin{gather}
\langle n_1, n_2, n_3| \Psi_q\bigl(\e^{2C_7+\uu_1+\uu_3+\ww_1-\ww_2+\ww_3}\bigr)^{-1}\Psi_q\bigl(\e^{2C_5+\uu_1-\uu_3+\ww_1-\ww_2+\ww_3}\bigr)^{-1}
\nonumber\\
\qquad= \sum_{k,l\ge 0}\frac{q^{l^2}}{\bigl(q^2\bigr)_k\bigl(q^2\bigr)_l}
\bigl(-\e^{2C_7}q^{1+n_1+n_3}\bigr)^k\nonumber\\
\phantom{\qquad= }{}\times\bigl(\e^{2C_5}q^{n_1-n_3}\bigr)^l
\langle n_1-k-l, n_2+k+l, n_3-k-l|,
\label{npp}\\
\Psi_q\bigl(\e^{\alpha_6+\uu_1-\uu_3+\ww_1-\ww_2+\ww_3}\bigr) \Psi_q\bigl(\e^{\alpha_8+\uu_1+\uu_3+\ww_1-\ww_2+\ww_3}\bigr) \bigl| n'_1, n'_2, n'_3\bigr\rangle
\nonumber\\
\qquad=\sum_{r,s\ge 0}\frac{q^{s^2}}{\bigl(q^2\bigr)_r\bigl(q^2\bigr)_s}
\bigl(\e^{\alpha_8}q^{2+n'_1+n'_3}\bigr)^s\nonumber\\
\phantom{\qquad=}{}\times\bigl(-\e^{\alpha_6}q^{1+n'_1-n'_3}\bigr)^r
\bigl|n'_1+r+s, n'_2-r-s, n'_3+r+s\bigr\rangle.
\label{ppn}
\end{gather}
Elements of $P$ \eqref{pp123} are calculated as
\begin{gather}
\langle n_1,n_2,n_3 | P \bigl|n'_1, n'_2, n'_3\bigr\rangle
\nonumber \\
\qquad=
 \langle n_1,n_2,n_3 |
\e^{(n_3-n_2)\ww_1}
\e^{g_{23}(-\ww_1-\ww_2+\ww_3)}
\e^{\tfrac{1}{\hbar}(\lambda_1\uu_1+\lambda_2\uu_2+\lambda_3\uu_3)}
 \bigl|n'_1, n'_3, n'_2\bigr\rangle
 \nonumber\\
\qquad=
\kappa \langle n_1+n_2-n_3, n_2,n_3|\e^{g_{23}(-\ww_1-\ww_2+\ww_3)}
 \bigl|n'_1, n'_3, n'_2\bigr\rangle \e^{\lambda_1 n'_1+\lambda_2n'_3+\lambda_3 n'_2}
 \nonumber\\
\qquad=
\kappa \langle n_1+n_2-n_3, n_2,n_3|
 \bigl|n'_1-g_{23}, n'_3-g_{23}, n'_2+g_{23}\bigr\rangle \e^{\lambda_1 n'_1+\lambda_2n'_3+\lambda_3 n'_2}
 \nonumber\\
\qquad=
 \kappa \delta^{n_1+n_2}_{n'_1+n'_2} \delta^{n_2+n_3}_{n'_2+n'_3} \delta^{n_3}_{n'_2+g_{23}}
 \e^{\lambda_1 n'_1+\lambda_2n'_3+\lambda_3 n'_2},
 \label{npn}
\end{gather}
where $\kappa$ is defined after \eqref{Rf1}.
Combining \eqref{npp}, \eqref{ppn} and \eqref{npn}, we get
\begin{gather*}
R^{n_1, n_2, n_3}_{n'_1, n'_2, n'_3}
\nonumber \\
\qquad= \kappa \delta^{n_1+n_2}_{n'_1+n'_2} \delta^{n_2+n_3}_{n'_2+n'_3}
\sum_{k,l,r,s\ge 0}\delta^{n_3-k-l}_{n'_2-r-s+g_{23}} q^{l^2+s^2}
\e^{\lambda_1(n'_1+r+s)+\lambda_2(n'_3+r+s)+\lambda_3(n'_2-r-s)}
\nonumber \\
\phantom{\qquad=}{} \times \frac{\bigl(-\e^{2C_7}q^{1+n_1+n_3}\bigr)^k\bigl(\e^{2C_5}q^{n_1-n_3}\bigr)^l
\bigl(\e^{\alpha_8}q^{2+n'_1+n'_3}\bigr)^s\bigl(-\e^{\alpha_6}q^{1+n'_1-n'_3}\bigr)^r}{\bigl(q^2\bigr)_k\bigl(q^2\bigr)_l\bigl(q^2\bigr)_s\bigl(q^2\bigr)_r}
\nonumber\\
\qquad= \kappa \delta^{n_1+n_2}_{n'_1+n'_2} \delta^{n_2+n_3}_{n'_2+n'_3}
 \e^{\lambda_1n'_1+\lambda_2n'_3+\lambda_3n'_2}
 \sum_{k,l,r,s\ge 0}\delta^{n_3-k-l}_{n'_2-r-s+g_{23}} q^{l^2+s^2}
\nonumber \\
\phantom{\qquad=}{}\times \frac{\bigl(-\e^{2C_7}q^{1+n_1+n_3}\bigr)^k\bigl(\e^{2C_5}q^{n'_1-n'_3}\bigr)^l
\bigl(\e^{-2C_8}q^{2+n'_1+n'_3}\bigr)^s\bigl(-\e^{-2C_6}q^{1+n'_1-n'_3}\bigr)^r}{\bigl(q^2\bigr)_k\bigl(q^2\bigr)_l\bigl(q^2\bigr)_s\bigl(q^2\bigr)_r},
\end{gather*}
where the last step uses \eqref{C68} and $n_1-n_3=n'_1-n'_3$ under the two Kronecker delta's.
The last line is the coefficient of $z^{n'_2-n_3+g_{23}}$ of
\begin{gather*}
\sum_{k,l,r,s\ge 0}
\frac{\bigl(-z^{-1}\e^{2C_7}q^{1+n_1+n_3}\bigr)^k}{\bigl(q^2\bigr)_k}
\frac{q^{l^2}\bigl(z^{-1}\e^{2C_5}q^{n'_1-n'_3}\bigr)^l}{\bigl(q^2\bigr)_l}
\frac{q^{s^2}\bigl(z\e^{-2C_8}q^{2+n'_1+n'_3}\bigr)^s}{\bigl(q^2\bigr)_s}\\
\qquad\times
\frac{\bigl(-z\e^{-2C_6}q^{1+n'_1-n'_3}\bigr)^r}{\bigl(q^2\bigr)_r}
=\frac{\bigl(-z^{-1}\e^{2C_5}q^{1+n'_1-n'_3}\bigr)_\infty\bigl(-z\e^{-2C_8}q^{3+n'_1+n'_3}\bigr)_\infty}
{\bigl(-z^{-1}\e^{2C_7}q^{1+n_1+n_3}\bigr)_\infty\bigl(-z\e^{-2C_6}q^{1+n'_1-n'_3}\bigr)_\infty}
=\Gamma(z)H(z),
\end{gather*}
where
\begin{gather*}
\Gamma(z)
= \frac{\bigl(-z\e^{-2C_7}q^{1-n_1-n_3}\bigr)_\infty\bigl(-z\e^{-2C_8}q^{3+n'_1+n'_3}\bigr)_\infty}
{\bigl(-z\e^{-2C_5}q^{1+n'_3-n'_1}\bigr)_\infty\bigl(-z\e^{-2C_6}q^{1+n'_1-n'_3}\bigr)_\infty},
\\
H(z)
=\frac{\bigl(-z\e^{-2C_5}q^{1+n'_3-n'_1}\bigr)_\infty\bigl(-z^{-1}\e^{2C_5}q^{1+n'_1-n'_3}\bigr)_\infty}
{\bigl(-z\e^{-2C_7}q^{1-n_1-n_3}\bigr)_\infty\bigl(-z^{-1}\e^{2C_7}q^{1+n_1+n_3}\bigr)_\infty}.
\end{gather*}
Thus far we have shown
\begin{align}\label{pre1}
R^{n_1, n_2, n_3}_{n'_1, n'_2, n'_3} &= \kappa \delta^{n_1+n_2}_{n'_1+n'_2} \delta^{n_2+n_3}_{n'_2+n'_3}
\e^{\lambda_1n'_1+\lambda_2n'_3+\lambda_3n'_2}
\oint\frac{{\rm d}z}{2\pi {\rm i} z}z^{-n'_2+n_3-g_{23}}\Gamma(z)H(z).
\end{align}
Note that \smash{$f(\xi) := (-q\xi)_\infty\bigl(-q\xi^{-1}\bigr)_\infty$} satisfies
\[
f(\xi)=q\xi^{-1}f\bigl(\xi q^{-2}\bigr)=q^{m^2}\xi^{-m}f\bigl(\xi q^{-2m}\bigr)
\]
for any $m \in \Z$.
Setting $\xi = z\e^{-2C_5}q^{n'_3-n'_1}= z\e^{-2C_5}q^{n_3-n_1}$, $m=n_3+g_3$, and using
$\e^{-2C_5} = \e^{-2C_7+e_3} = \e^{-2C_7}q^{2g_3}$, we find
\[
H(z) = f(\xi)/f\bigl(\xi q^{-2m}\bigr) = \bigl(\e^{2C_5} q^{n_1+g_3}\bigr)^{n_3+g_3}z^{-n_3-g_3}.
\]
Substituting this into \eqref{pre1} and replacing $z$ by $q^{-1}z$, we obtain \eqref{Rf1}.
\end{proof}

\begin{Remark}\label{re:al}
The factor \smash{$\e^{\lambda_1 n'_1+\lambda_2n'_3+\lambda_3 n'_2}
\bigl(\e^{2C_5}q^{n_1+g_3}\bigr)^{n_3+g_3}q^{n'_2+g_2}$} in the first line of \eqref{Rf1} is equal to
\[
\e^{-C_4 + \tfrac{1}{\hbar} (-C_5^2 + C_7^2)
+C_7 (n_1 + n_3) + (C_1 + C_5 - C_6) (n_3' - n_1') - (C_2 + C_8)(n_1' + n_3') + 2 C_3 n_2'}
q^{n_1 n_3 + n'_2}
\]
under the condition implied by \smash{$\delta^{n_1+n_2}_{n'_1+n'_2} \delta^{n_2+n_3}_{n'_2+n'_3}$}.
Therefore, the result~\eqref{Rf1} fulfills the symmetry~\eqref{alphauw-R}\footnote{$\alpha^{\uu\ww}$ should
also be accompanied by the interchange $\bigl(n_i, n'_i\bigr) \leftrightarrow \bigl(n_{4-i},n'_{4-i}\bigr)$
in view of \eqref{uwal}.}
due to \eqref{alC} and the fact that the right-hand side of \eqref{alphauw-f} is equal to
$\exp\bigl(\bigl(C_6^2-C_5^2\bigr)/\hbar\bigr)$.
\end{Remark}

\begin{Remark}\label{re:R}
 When $a_i=b_i=c_i=d_i=e_i=0$ for $i=1,2,3$, hence
 $\forall C_i = \forall\lambda_i=\forall g_i=0$, the formula
 \eqref{Rf1} coincides, including the overall normalization, with
 \smash{$R^{n'_1,n'_2,n'_3}_{n_1,n_2,n_3}$} from \cite[equation~(3.81)]{K22} for
 the elements of the $R$-matrix \cite{KV94} originally discovered
 from the representation theory of quantized coordinate ring. A
 similar integral formula was recognized earlier in the footnote of
 \cite[p.\ 5]{BS06}. In this case, $\Gamma(z)$ reduces to a rational
 function of $z$ (see the explanation after \cite[equation~(3.81)]{K22}),
 and the tetrahedron equation closes among elements with non-negative
 integer indices.
\end{Remark}

\subsection[Elements of R in w-diagonal representation]{Elements of $\boldsymbol{R}$ in $\boldsymbol{\ww}$-diagonal representation}\label{ss:wd}

Let us turn to the representation of the canonical variables in which $\ww_k$'s are diagonal,
\begin{gather}
\e^{\uu_k}|\mathbf{n}\rangle = |\mathbf{n} - \mathbf{e}_k\rangle,
\qquad
\e^{\ww_k}|\mathbf{n}\rangle = q^{n_k}|\mathbf{n}\rangle,
\qquad
\langle \mathbf{n}|\e^{\uu_k}=\langle \mathbf{n} + \mathbf{e}_k|,\nonumber
\\
\langle \mathbf{n}|\e^{\ww_k} = \langle \mathbf{n}|q^{n_k},\label{wrep}
\end{gather}
where notations are similar to \eqref{qrep}.
We employ the same pairing $\langle \mathbf{n}|\mathbf{n}'\rangle = \delta_{\mathbf{n}, \mathbf{n}'}$ and the
notation~${g_{23}=g_2-g_3 = (-C_4+C_5-C_7)\hbar^{-1}}$ introduced after \eqref{pp123}.
In this subsection, we assume
\begin{align}\label{els}
\ell_i := \frac{\lambda_i}{\hbar} \in \Z,\qquad i=1,2,3,
\end{align}
where $\lambda_i$'s are defined in \eqref{lac}.

\begin{Theorem}\label{th:ew}
Under the assumption \eqref{els}, the element
\smash{$S^{n_1,n_2, n_3}_{n'_1, n'_2, n'_3} :=
\langle n_1, n_2, n_3| R\bigl| n'_1, n'_2, n'_3\bigr\rangle$} of the $R$-matrix \eqref{RL12} is given by
\begin{gather}
S^{n_1,n_2, n_3}_{n'_1, n'_2, n'_3} =
\frac{(-1)^{\frac{\nu_1}{2}}q^{\psi+\omega}(q^{\nu_3+\nu_4})_\infty\bigl(q^2\bigr)^3_\infty}
{(q^{\nu_1})_\infty (q^{\nu_2})_\infty (q^{\nu_3})_\infty (q^{\nu_4})_\infty},
\label{Sf}\\
\nu_1 = 2C_3\hbar^{-1}+n_1+n_3-n'_2,
\qquad
\nu_2 = -2(C_2+C_8)\hbar^{-1}+ n_2-n'_1-n'_3,
\\
\nu_3 = 2(C_1-C_3+C_5)\hbar^{-1}-n_1-n_2+n_3+n'_1+n'_2-n'_3,
\\
\nu_4 = 2(-C_1-C_3+C_6)\hbar^{-1}+ n_1-n_2-n_3-n'_1+n'_2+n'_3,
\\
\psi= \tfrac{1}{4}\bigl(-(\nu_1-\nu_2)(\nu_3+\nu_4)+\nu_3\nu_4-\nu_1^2+2\nu_1\bigr)
\nonumber\\
\phantom{\psi= }{} + \tfrac{1}{2\hbar}
\bigl((C_8-C_7)(\nu_1+\nu_2)+(C_8-C_6-C_4)\nu_3 + (C_8-C_5-C_4)\nu_4\bigr),
\label{Spf}
\end{gather}
where $\omega$ is independent of $n_i$, $n'_i$, and given by
$\omega= (C_5+C_6)(C_4-C_5+C_7)/\hbar^2$.
\end{Theorem}
\begin{proof}
The derivation is similar to Theorem \ref{th:re}.
We have
\begin{gather*}
\langle n_1, n_2, n_3| \Psi_q\bigl(\e^{2C_7+\uu_1+\uu_3+\ww_1-\ww_2+\ww_3}\bigr)^{-1}\Psi_q\bigl(\e^{2C_5+\uu_1-\uu_3+\ww_1-\ww_2+\ww_3}\bigr)^{-1}
\nonumber\\
\qquad= \sum_{k,l\ge 0}\frac{q^{2k^2+l^2}}{\bigl(q^2\bigr)_k\bigl(q^2\bigr)_l}
\bigl(\e^{2C_7}q^{n_1-n_2+n_3}\bigr)^k A^l
\langle n_1+k+l, n_2, n_3+k-l|,
\\
\Psi_q\bigl(\e^{\alpha_6+\uu_1-\uu_3+\ww_1-\ww_2+\ww_3}\bigr) \Psi_q\bigl(\e^{\alpha_8+\uu_1+\uu_3+\ww_1-\ww_2+\ww_3}\bigr) \bigl| n'_1, n'_2, n'_3\bigr\rangle
\nonumber\\
\qquad=\sum_{r,s\ge 0}\frac{q^{-s^2+s+r}}{\bigl(q^2\bigr)_r\bigl(q^2\bigr)_s}
\bigl(-\e^{\alpha_8}q^{n'_1-n'_2+n'_3}\bigr)^s B^r
\bigl|n'_1-r-s, n'_2, n'_3+r-s\bigr\rangle,
\\
\langle n_1, n_2, n_3|P\bigl| n'_1, n'_2, n'_3\bigr\rangle
= q^{g_{23}(n_1-n_2+n_3)}\delta^{n_1+\ell_1}_{n'_1}
\delta^{n_1+n_3+\ell_3}_{n'_2}\delta^{n_2-n_1+\ell_2}_{n'_3},
\end{gather*}
where
$A=\e^{2C_5}q^{n_1-n_2+n_3+2k}$ and $B=-\e^{\alpha_6}q^{n'_1-n'_2+n'_3-2s}$.
They lead to
\begin{align}
\langle n_1, n_2, n_3| R\bigl| n'_1, n'_2, n'_3\bigr\rangle
 ={}& \frac{q^{2k^2-s^2+s}}{\bigl(q^2\bigr)_k\bigl(q^2\bigr)_s}q^{g_{23}(n_1-n_2+n_3)}
 \bigl(\e^{2C_7}q^{n_1-n_2+n_3}\bigr)^k\bigl(-\e^{\alpha_8}q^{n'_1-n'_2+n'_3}\bigr)^s
 \nonumber \\
 & \times \sum_{l+r=M}\frac{q^{l^2+r}}{\bigl(q^2\bigr)_r\bigl(q^2\bigr)_l}A^lB^r
 \label{kono}
\end{align}
with $k$, $s$ and $M$ fixed as\footnote{There is a parity condition on
 $n_i$, $n'_i$, $\ell_i$ in order to ensure $k, s \in \Z$ in
 \eqref{pari}. However the final formula \eqref{Sf} makes sense for
 generic $C_i$'s.}
\begin{gather}
k=\frac{1}{2}\bigl(n'_2-n_1-n_3-\ell_3\bigr)=-\frac{\nu_1}{2},
\qquad
s = \frac{1}{2}\bigl(n'_1+n'_3-n_2-\ell_1-\ell_2\bigr) = -\frac{\nu_2}{2},
\nonumber
\\
M = n'_1-n_1-k-s-\ell_1 \nonumber\\
\phantom{M }{} = \frac{1}{2}\bigl(-n_1+n_2+n_3+n'_1-n'_2-n'_3-\ell_1+\ell_2+\ell_3\bigr)
= -\frac{\nu_4}{2}.\label{pari}
\end{gather}
The last line of \eqref{kono} is
$(qB)^M\bigl(-AB^{-1}\bigr)_M/\bigl(q^2\bigr)_M$ with $AB^{-1}= -q^{\nu_3+\nu_4}$.
Thus \eqref{kono} is equal~to
\[
\frac{q^{\psi'}\bigl(q^{\nu_3+\nu_4}\bigr)_{-\nu_4/2}}
{\bigl(q^2\bigr)_{-\nu_1/2}\bigl(q^2\bigr)_{-\nu_2/2}\bigl(q^2\bigr)_{-\nu_4/2}}
= \frac{q^{\psi'}\bigl(q^{\nu_3+\nu_4}\bigr)_\infty}
{\bigl(q^2\bigr)_{-\nu_1/2}\bigl(q^2\bigr)_{-\nu_2/2}\bigl(q^2\bigr)_{-\nu_4/2}(q^{\nu_3})_\infty}
\]
for some power $\psi'$.
Rewriting \smash{$\bigl(q^2\bigr)_{-\nu_i/2}$} in the denominator $(i=1,2,4)$ as
$(-1)^{-\frac{\nu_i}{2}}q^{\frac{\nu_i}{2}(\frac{\nu_i}{2}-1)}\allowbreak\times(q^{\nu_i})_\infty/\bigl(q^2\bigr)_\infty$,
we obtain \eqref{Sf}--\eqref{Spf}.
\end{proof}

It is easily confirmed that the result \eqref{Sf} fulfills the symmetry \eqref{alphauw-R}.

A slightly more explicit form of \eqref{Sf} is
\begin{gather}
S^{n_1,n_2, n_3}_{n'_1, n'_2, n'_3} \nonumber\\
\qquad\equiv q^{\psi_0}
\bigl(-\e^{-2C_7}\bigr)^{\frac{m_1}{2}}
\bigl(\e^{2C_8-2C_3}\bigr)^{\frac{m_2}{2}}
\bigl(\e^{-C_1-C_2-2C_3-C_4}\bigr)^{\frac{m_3}{2}}
\bigl(\e^{C_1-C_2-2C_3-C_4}\bigr)^{\frac{m_4}{2}}
\nonumber\\
\phantom{\qquad\equiv}{} \times
\frac{\bigl(\e^{-4C_3+2C_5+2C_6}q^{m_3+m_4}\bigr)_\infty\bigl(q^2\bigr)^3_\infty}
{\bigl(\e^{2C_3}q^{m_1}\bigr)_\infty
\bigl(\e^{-2C_2-2C_8}q^{m_2}\bigr)_\infty
\bigl(\e^{2C_1-2C_3+2C_5}q^{m_3}\bigr)_\infty
\bigl(\e^{-2C_1-2C_3+2C_6}q^{m_4}\bigr)_\infty},
\label{Sf2}\\
\begin{pmatrix}m_1 \\ m_2\end{pmatrix} =
\begin{pmatrix}n_1+n_3-n'_2 \\ n_2-n'_1-n'_3 \end{pmatrix},
\qquad
\begin{pmatrix}m_3 \\ m_4\end{pmatrix} =
\begin{pmatrix}-n_1-n_2+n_3+n'_1+n'_2-n'_3 \\ n_1-n_2-n_3-n'_1+n'_2+n'_3 \end{pmatrix},
\label{mdef}
\\
\psi_0 = \frac{1}{4}\bigl(-(m_1-m_2)(m_3+m_4)+m_3m_4-m_1^2+2m_1\bigr),
\label{p0def}
\end{gather}
up to an overall factor depending on $C_1, \dots, C_8$. The formula
\eqref{mo} is obtained, up to an overall factor, by replacing the
infinite products $(zq^m)_\infty$ appearing here with
\smash{$\bigl(z;q^2\bigr)_\infty/\bigl(z;q^2\bigr)_{\tfrac{m}{2}}$}.\footnote{For a proper
 treatment of indices with both parities, see \cite[equation~(49)]{KMY23}
 and also \cite[Proposition~2]{KMY23}.}

Let us compare the above \smash{$S^{n_1,n_2, n_3}_{n'_1, n'_2, n'_3}$} with
the $R$-matrix $R^{ZZZ}$ obtained in \cite[equation~(45)]{KMY23}. We
write the element \smash{$R^{a,b,c}_{i,j,k}$} therein as \smash{$X^{a,b,c}_{i,j,k}$}
here for distinction. It contains twelve parameters~${(r_j ,s_j, t_j, w_j)}$ ($j=1, 2, 3$). Apply \cite[equation~(51)]{KMY23} to rewrite the first factor in its denominator and
replace the parameters as
\smash{$(t_j,w_j) \rightarrow \bigl(-i q^{-\frac{1}{2}} t_j, t_j^{-1}w_j\bigr)$}. The
result reads
\begin{gather}
X^{n_1,n_2,n_3}_{n'_1,n'_2,n'_3}=
q^{\varphi} \left(\frac{r_1r_3}{t_3w_1}\right)^{\frac{m_1}{2}}
\left(-\frac{s_2}{t_1w_3}\right)^{\frac{m_2}{2}}
\left( \frac{t_2}{s_1t_3}\right)^{\frac{m_3}{2}}
\left(\frac{w_2}{s_3w_1}\right)^{\frac{m_4}{2}}\nonumber
\\
\phantom{X^{n_1,n_2,n_3}_{n'_1,n'_2,n'_3}=}{}\times
\frac{
\Theta_{m_1}\bigl(\frac{r_2}{r_1r_3}\bigr)
\Theta_{m_2}\bigl(\frac{s_1s_3}{s_2}\bigr)
\Theta_{m_3}\bigl(\frac{r_3t_1w_2}{s_3t_2w_1}\bigr)
\Theta_{m_4}\bigl(\frac{r_1t_2w_3}{s_1t_3w_2}\bigr)}
{\Theta_{m_3+m_4}\bigl(\frac{r_1r_3t_1w_3}{s_1s_3t_3w_1}\bigr)},\nonumber
\\
\varphi = \frac{1}{4}\bigl((m_1-m_2)(m_3+m_4)+m_3m_4-m^2_2+2m_2\bigr),\label{Xn}
\end{gather}
where $m_j$'s are those in \eqref{mdef}.
The function $\Theta_m(z)$ is defined up to normalization by
$\Theta_{m+2}(z) = (1-zq^m)\Theta_m(z)$.

\begin{Remark}\label{re:kmy0}
With the choice $\Theta_m(z)=1/\bigl(zq^m;q^2\bigr)_\infty$ and the identification of parameters as
\begin{gather}
\e^{C_1} = \sqrt{\frac{r_1 t_2 w_3}{r_3 t_1 w_2}}, \qquad
\e^{C_2} = \sqrt{\frac{r_2 t_1 w_3}{r_1 r_3 s_2}}, \qquad
\e^{C_3} = \sqrt{\frac{r_1 r_3}{r_2}}, \qquad
\e^{C_4} = \sqrt{\frac{r_2 s_2}{t_2 w_2}},\nonumber
\\
\e^{C_5} = \sqrt{\frac{r_3 s_3 w_1}{r_2 w_3}}, \qquad
\e^{C_6} = \sqrt{\frac{r_1 s_1 t_3}{r_2 t_1}}, \qquad
\e^{C_7} = \sqrt{\frac{t_3 w_1}{r_2}}, \qquad
\e^{C_8} = \sqrt{\frac{r_1 r_3 s_1 s_3}{r_2 t_1 w_3}},\label{Css}
\end{gather}
the elements \eqref{Sf2} and \eqref{Xn} are related as
\[
S^{n_1,n_2, n_3}_{n'_1, n'_2, n'_3} \equiv
X^{-n'_3,-n'_2,-n'_1}_{-n_3,-n_2, -n_1}\big|_{(r_i,s_i,t_i,w_i) \rightarrow
(s_{4-i},r_{4-i}, t_{4-i}, w_{4-i})} .
\]
The replacement $\bigl(n_1,n_2,n_3,n'_1,n'_2,n'_3\bigr) \rightarrow
\bigl(-n'_3,-n'_2,-n'_1,-n_3,-n_2,-n_1\bigr)$ in the right-hand side induces
the exchange $m_1 \leftrightarrow m_2$ and $m_3 \leftrightarrow m_4$,
converting $\varphi$ into $\psi_0$.
Thus we have elucidated a quantum cluster algebra theoretic origin of
the $R$-matrix $R^{ZZZ}$ \cite{KMY23}.
\end{Remark}

\begin{Remark}
Apart from trivial rescaling of generators,
a $q$-Weyl algebra $\bigl\langle \e^{\pm \uu}, \e^{\pm \ww}\bigr\rangle$
with the commutation relation
$\e^\uu \e^\ww = q \e^\ww \e^\uu$
has the automorphism labeled with ${\rm SL}(2,\Z)$
\[
\iota_f\colon\
\begin{cases}\e^\uu \mapsto \e^{\alpha \uu} \e^{\beta \ww},
\\
\e^\ww \mapsto \e^{\gamma \uu} \e^{\delta \ww},
\end{cases}
\qquad
f= \begin{pmatrix} \alpha & \beta \\ \gamma & \delta \end{pmatrix}
\in {\rm SL}(2,\Z).
\]
Recall the $n$-fold direct product of the $q$-Weyl algebra $\mathcal{W}_n$ introduced in Section \ref{s:qw}.
Given a~representation $\rho_1 \otimes \rho_2 \otimes \rho_3\colon \mathcal{W}_3
\rightarrow \mathrm{End}(V_1 \otimes V_2 \otimes V_3)$ of $\mathcal{W}_3$,
one generates an infinite family of representations via the pullback
\begin{align*}
\rho^{f_1}_1 \otimes \rho^{f_2}_2 \otimes \rho^{f_3}_3\colon\
\mathcal{W}_3\xrightarrow{\iota_{f_1} \otimes \iota_{f_2} \otimes \iota_{f_3}}
\mathcal{W}_3
\xrightarrow{\rho_1 \otimes \rho_2 \otimes \rho_3}
\mathrm{End}(V_1 \otimes V_2 \otimes V_3).
\end{align*}
The $u$-diagonal representation and the $w$-diagonal representation
considered in this section are essentially transformed by the above
automorphism. They are just two special ``homogeneous'' cases
$\rho_1=\rho_2=\rho_3$, where the computation of the elements of
$\e^{(\uu_3-\uu_2)\ww_1/\hbar}$ in \eqref{pp123} is simple. A~similar
remark applies also to the treatment in the next section. In the
context of the $RLLL$ relation (cf.\ Section \ref{s:6v}), a study of
the case $\rho^{f_1} \otimes \rho^{f_2} \otimes \rho^{f_3}$ with
non-identical $f_1$, $f_2$, $f_3$ has been undertaken in \cite{KMY23}.
\end{Remark}

\section[Modular R and its elements]{Modular $\boldsymbol{\mathcal{R}}$ and its elements}\label{s:md}

\subsection[Modular R]{Modular $\boldsymbol{\mathcal{R}}$}
We use a parameter $\bb$ and set
\begin{align}\label{qq}
\hbar = {\rm i}\pi \bb^2,\qquad q= \e^{ {\rm i} \pi \bb^2},
\qquad q^\vee = \e^{ {\rm i} \pi \bb^{-2}},\qquad \bar{q} = \e^{- {\rm i} \pi \bb^{-2}},
\qquad \eta = \frac{\bb+\bb^{-1}}{2}.
\end{align}
The non-compact quantum dilogarithm is defined by
\begin{align}\label{ncq}
\Phi_\bb(z) = \exp\left(
\frac{1}{4}\int_{\R+{\rm i} 0} \frac{\e^{-2{\rm i}zw}}{\sinh(w \bb)\sinh(w/\bb)}\frac{{\rm d}w}{w}\right)
= \frac{\bigl(\e^{2\pi(z+{\rm i}\eta)\bb};q^2\bigr)_\infty}{\bigl(\e^{2\pi(z-{\rm i}\eta)\bb^{-1}};\bar{q}^2\bigr)_\infty},
\end{align}
where the integral avoids the singularity at $w=0$ from above.
The infinite product formula is valid in the so-called strong coupling regime $0<\eta < 1$
with $0<\operatorname{Im}\bb <\frac{\pi}{2}$.
It enjoys the symmetry
$\Phi_\bb(z) = \Phi_{\bb^{-1}}(z)$, and
has the following properties (see also \cite{FKV01})
\begin{gather}
\Phi_\bb(z)\Phi_\bb(-z) = \e^{ {\rm i} \pi z^2- {\rm i} \pi (1-2\eta^2)/6},
\label{pinv}\\
\frac{\Phi_\bb\bigl(z-{\rm i} \bb^{\pm 1}/2\bigr)}{\Phi_\bb\bigl(z+{\rm i} \bb^{\pm 1}/2\bigr)}
= 1+\e^{2\pi z \bb^{\pm1}},
\label{prec}
\\
\Phi_\bb(z) \rightarrow
\begin{cases} 1, & \operatorname{Re} z \rightarrow -\infty,
\\
\e^{ {\rm i} \pi z^2- {\rm i} \pi (1-2\eta^2)/6}, & \operatorname{Re} z \rightarrow \infty,
\end{cases}
\label{pa}
\\
\text{poles of $\Phi_\bb(z)^{\pm 1}$} = \bigl\{ \pm \bigl({\rm i}\eta + {\rm i} m \bb + {\rm i}n \bb^{-1}\bigr)\mid m,n \in \Z_{\ge 0}\bigr\}.
\label{pz}
\end{gather}

Recall that $\uu_k$, $\ww_k$ ($k=1, 2, 3$) are canonical variables
obeying \eqref{uwh}. In this section, we work with the
$\pi \bb$-rescaled canonical variables $\hat{x}_k$, $\hat{p}_k$ and
the parameters defined as follows:
\begin{gather}
\uu_k =\pi \bb \hat{x}_k,\qquad
\ww_k = \pi \bb \hat{p}_k,
\qquad [\hat{x}_j, \hat{p}_k] = \frac{\rm i}{\pi} \delta_{j,k},\nonumber
\\
(a_k,b_k, c_k, d_k,e_k) =
\pi \bb \bigl(\tilde{a}_k, \tilde{b}_k, \tilde{c}_k, \tilde{d}_k, \tilde{e}_k\bigr),
\qquad \lambda_k = \pi \bb \tilde{\lambda}_k,\qquad C_k =\pi \bb \mathcal{C}_k.
\label{cbc}
\end{gather}
From \eqref{Prec} and \eqref{prec}, we have
\[
\frac{\Psi_q\bigl(\e^{2\pi \bb(z+ {\rm i}\bb/2)}\bigr)}
{\Psi_q\bigl(\e^{2\pi \bb(z- {\rm i}\bb/2)}\bigr)}
= \frac{\Phi_\bb(z-{\rm i} \bb/2)}{\Phi_\bb(z+{\rm i} \bb/2)}.
\]
It indicates the formal correspondence
\begin{equation}\label{cor}
\Psi_q\bigl(\e^{2\pi \bb z}\bigr) \leftrightarrow
\Phi_\bb(z)^{-1}.
\end{equation}
Applying it to \eqref{RL12} and \eqref{pp123}, we define
\begin{align}
\mathcal{R}
={}& f\bigl({\tilde a}, {\tilde b}, {\tilde c}, {\tilde d}\bigr)
\Phi_\bb\left(\frac{1}{2}(\hat{x}_1 + \hat{x}_3 + \hat{p}_1 - \hat{p}_2 + \hat{p}_3 +2\mathcal{C}_7)\right)\nonumber\\
& \times
\Phi_\bb\left(\frac{1}{2}(\hat{x}_1 - \hat{x}_3 + \hat{p}_1 - \hat{p}_2 + \hat{p}_3 +2\mathcal{C}_5)\right)
\nonumber\\
& \times \mathcal{P}
\Phi_\bb\left(\frac{1}{2}(\hat{x}_1 - \hat{x}_3 + \hat{p}_1 - \hat{p}_2 + \hat{p}_3 +\tilde{\alpha}_6)\right)^{-1}\nonumber\\
& \times
\Phi_\bb\left(\frac{1}{2}(\hat{x}_1 + \hat{x}_3 + \hat{p}_1 - \hat{p}_2 + \hat{p}_3 +\tilde{\alpha}_8)\right)^{-1},
\nonumber
\\
\mathcal{P}={}& \e^{\pi {\rm i}(\hat{x}_2-\hat{x}_3)\hat{p}_1}
\e^{\pi {\rm i}\tilde{\lambda}_0(\hat{p}_1+\hat{p}_2-\hat{p}_3)}
\e^{-\pi {\rm i}(\tilde{\lambda}_1\hat{x}_1+\tilde{\lambda}_2\hat{x}_2+\tilde{\lambda}_3\hat{x}_3)}
\rho_{23},
\label{RRf}
\end{align}
where
$\tilde{\alpha}_6= -\tilde{b}_1-\tilde{a}_2-\tilde{d}_3-\tilde{e}_3$ and
$\tilde{\alpha}_8 = -\tilde{b}_1-\tilde{a}_2-\tilde{d}_3$.
They are also determined by
$\tilde{\alpha}_6+\tilde{\lambda}_1+\tilde{\lambda}_2-\tilde{\lambda}_3=-2\mathcal{C}_6$ and
$\tilde{\alpha}_8+\tilde{\lambda}_1+\tilde{\lambda}_2-\tilde{\lambda}_3=-2\mathcal{C}_8$
as in \eqref{C68}.

The normalization of $\mathcal{R}$ remains inherently undetermined from the postulate on
$\mathrm{Ad}(\mathcal{R})$.
Following the symmetry argument in Section \ref{ss:symR} with the rescaling \eqref{cbc} of parameters,
we choose the prefactor $f\bigl({\tilde a}, {\tilde b}, {\tilde c}, {\tilde d}\bigr)$ as
\begin{align}
f\bigl(\tilde{a}, \tilde{b}, \tilde{c}, \tilde{d}\bigr)
& = \exp\left( {\rm i} \pi (\mathcal{C}_4-\mathcal{C}_5+\mathcal{C}_7)(\mathcal{C}_5+\mathcal{C}_6)\right)\nonumber
\\
&=\exp\left( -\frac{{\rm i}\pi }{4}(\tilde{e}_2-\tilde{e}_3)\bigl(\tilde{a}_1+\tilde{a}_3+\tilde{c}_1-2\tilde{c}_2+\tilde{c}_3
+\tilde{b}_1-\tilde{b}_3-\tilde{d}_1+\tilde{d}_3\bigr)\right).\label{fc}
\end{align}
Then $\mathrm{Ad}(\mathcal{R})$ is invariant under the simultaneous
exchange $1 \leftrightarrow 3$ and
${\tilde b}_i \leftrightarrow {\tilde d}_i$ of indices and parameters.
Furthermore, $\mathrm{Ad}(\mathcal{R})$ becomes
$\mathrm{Ad}(\mathcal{R})^{-1}$ under the exchange
${\tilde a}_i \leftrightarrow {\tilde c}_i$,
${\tilde b}_i \leftrightarrow {\tilde d}_i$ of parameters. We can
multiply $f$ by any function
$h\bigl(\tilde{a}, \tilde{b}, \tilde{c}, \tilde{d}\bigr)$ such that
\begin{align*}
&h\bigl(\tilde{a}, \tilde{b}, \tilde{c}, \tilde{d}\bigr)h\bigl(\tilde{c}, \tilde{d}, \tilde{a}, \tilde{b}\bigr)=1,
\\
&h\bigl(\tilde{a}_1, \tilde{b}_1, \tilde{c}_1, \tilde{d}_1,
\tilde{a}_2, \tilde{b}_2, \tilde{c}_2, \tilde{d}_2,
\tilde{a}_3, \tilde{b}_3, \tilde{c}_3, \tilde{d}_3\bigr)
=
h\bigl(\tilde{a}_3, \tilde{d}_3, \tilde{c}_3, \tilde{b}_3,
\tilde{a}_2, \tilde{d}_2, \tilde{c}_2, \tilde{b}_2,
\tilde{a}_1, \tilde{d}_1, \tilde{c}_1, \tilde{b}_1\bigr),
\end{align*}
and the result still preserves the symmetries.

\subsection[Elements of R in coordinate representation]{Elements of $\boldsymbol{\mathcal{R}}$ in coordinate representation}\label{ss:cr}

We consider the representation of canonical variables on a space of
functions $G(\mathbf{x})$ of $\mathbf{x}=(x_1,x_2,x_3)$, where the
``coordinate'' $\hat{x}_k$ acts diagonally, i.e.,
$(\hat{x}_kG)(\mathbf{x}) = x_k G(\mathbf{x})$,
$(\hat{p}_k G)(\mathbf{x}) = -\frac{\rm i}{\pi}\frac{\partial G}{\partial
 x_k}(\mathbf{x})$. The functions $G(\mathbf{x})$ are actually
supposed to belong to a subspace of $L^2\bigl(\R^3\bigr)$. See \cite[Section~5.2]{FG09b} for details. We find it convenient to employ the bracket
notation as ${G(\mathbf{x}) = \langle \mathbf{x}| G\rangle}$. Then the
coordinate representation along with its dual can be summarized
formally in a~difference (exponential) form as follows:
\begin{gather}
\e^{\pi \bb \hat{x}_k}|\mathbf{x}\rangle = \e^{\pi \bb x_k}|\mathbf{x}\rangle,
\qquad
\e^{\pi \bb \hat{p}_k}|\mathbf{x}\rangle = |\mathbf{x} + {\rm i} \bb \mathbf{e}_k\rangle,\nonumber
\\
\langle \mathbf{x}|\e^{\pi \bb \hat{x}_k} = \langle \mathbf{x}|\e^{\pi \bb x_k},
\qquad
\langle \mathbf{x}|\e^{\pi \bb \hat{p}_k} = \langle \mathbf{x}-{\rm i} \bb \mathbf{e}_k|\label{xrep}
\end{gather}
for $k=1,2,3$, where $|x_1,x_2,x_3\rangle$ (resp.\
$\langle x_1,x_2,x_3|$) is denoted by $|\mathbf{x}\rangle$ (resp.\
$\langle \mathbf{x}|$). The dual pairing is specified by
$\bigl\langle \mathbf{x} | \mathbf{x}'\bigr\rangle =
\delta\bigl(x_1-x'_1\bigr)\delta\bigl(x_2-x'_2\bigr)\delta\bigl(x_3-x'_3\bigr)$.

\begin{Theorem}\label{th:Rm}
The matrix element
\smash{${\mathcal R}^{x_1,x_2,x_3}_{x'_1,x'_2, x'_3} = \langle x_1, x_2, x_3 |{\mathcal R}\bigl|x'_1, x'_2, x'_3 \bigr\rangle$}
of \eqref{RRf} with $f$ specified in \eqref{fc} is given, up to an overall numerical factor, by
\begin{gather}
{\mathcal R}^{x_1,x_2,x_3}_{x'_1,x'_2, x'_3}
\equiv g\bigl(\tilde{a}, \tilde{b}, \tilde{c}, \tilde{d}\bigr)
\delta\bigl(x_1+x_2-x'_1-x'_2\bigr)\delta\bigl(x_2+x_3-x'_2-x'_3\bigr)\e^{ {\rm i} \pi \phi} I^{x_1,x_2,x_3}_{x'_1,x'_2, x'_3}
\label{Rm0}
\\
 I^{x_1,x_2,x_3}_{x'_1,x'_2, x'_3} =
\int_{-\infty}^\infty {\rm d}z \e^{2\pi {\rm i}z(-x_2-{\rm i}\eta+{\mathcal C}_4)}
\nonumber \\
\phantom{ I^{x_1,x_2,x_3}_{x'_1,x'_2, x'_3} =}{} \times \frac{\Phi_\bb\bigl(z+\frac{1}{2}(x_1-x_3+{\rm i}\eta)+{\mathcal C}_5\bigr)
\Phi_\bb\bigl(z+\frac{1}{2}(-x_1+x_3+{\rm i}\eta)+{\mathcal C}_6\bigr)}
{\Phi_\bb\bigl(z+\frac{1}{2}(x_1+x_3-{\rm i}\eta)+{\mathcal C}_7\bigr)
\Phi_\bb\bigl(z+\frac{1}{2}\bigl(-x'_1-x'_3-{\rm i}\eta\bigr)+{\mathcal C}_8\bigr)},
\label{Rm1}
\\
\phi = x'_1x'_3+{\rm i}\eta\bigl(x'_1+x'_3-x_2\bigr)+{\mathcal C}_1(x_1-x_3)+{\mathcal C}_2\bigl(x'_1+x'_3\bigr)-2{\mathcal C}_3x'_2,
\label{Rm2}\\
g\bigl(\tilde{a}, \tilde{b}, \tilde{c}, \tilde{d}\bigr) =
\exp\left({\rm i}\pi(\mathcal{C}_4(\mathcal{C}_7+\mathcal{C}_8)
+(\mathcal{C}_5-\mathcal{C}_7)(\mathcal{C}_6-\mathcal{C}_7)
+ {\rm i}\eta(\mathcal{C}_4-2\mathcal{C}_8))\right).\nonumber
\end{gather}
\end{Theorem}
\begin{proof}
 In order to calculate the matrix elements of $\mathcal{R}$, we
 insert appropriate complete bases between each factor in the
 expression \eqref{RRf} and use quantum dilogarithm identities.

 Let us consider the matrix elements of the first quantum
 dilogarithm. Noting that $\hat{x}_1 + \hat{p}_3$, $\hat{p_2}$ and
 $\hat{x}_3 + \hat{p_1}$ in the argument commute with one another, we
 expand this quantum dilogarithm in the powers of these combinations
 of coordinates and momenta, sandwich the resulting series between
 $\bigl\langle x''_1, p_2, p_3\bigr|$ and $\bigl|p_1, p'_2, x''_3 \bigr\rangle$, and
 resum the series back to a quantum dilogarithm to get
 \begin{gather*}
 \bigl\langle x''_1, p_2, p_3\bigr|
 \Phi_\bb\left(\frac{1}{2}(\hat{x}_1 + \hat{x}_3 + \hat{p}_1 - \hat{p}_2 + \hat{p}_3 +2\mathcal{C}_7)\right)
 \bigl|p_1, p'_2, x''_3 \bigr\rangle
 \\
 \qquad =
 \Phi_\bb\left(\frac{1}{2}\bigl(x''_1 + x''_3 + p_1 - p_2 + p_3 +2\mathcal{C}_7\bigr)\right)
 \bigl\langle x''_1, p_2, p_3|p_1, p'_2, x''_3 \bigr\rangle.
 \end{gather*}
 Thus, the matrix elements are given, up to an overall numerical
 factor, by
 \begin{gather*}
 \langle x_1, x_2, x_3|
 \Phi_\bb\left(\frac{1}{2}(\hat{x}_1 + \hat{x}_3
 + \hat{p}_1 - \hat{p}_2 + \hat{p}_3 +2\mathcal{C}_7)\right)
 \bigl|x'_1, x'_2, x'_3 \bigr\rangle
 \\
\qquad
 \equiv \langle x_1, x_2, x_3| \int {\rm d}x''_1 {\rm d}p_2 {\rm d}p_3
 \bigl|x''_1, p_2, p_3 \bigr\rangle \bigl\langle x''_1, p_2, p_3 \bigr|
 \\
\phantom{\qquad
 \equiv}{} \times
 \Phi_\bb\left(\frac{1}{2}(\hat{x}_1 + \hat{x}_3
 + \hat{p}_1 - \hat{p}_2 + \hat{p}_3 +2\mathcal{C}_7)\right)\\
\phantom{\qquad
 \equiv}{} \times
 \int {\rm d}p_1 {\rm d}p'_2 {\rm d}x''_3 \bigl|p_1 , p'_2 , x''_3 \bigr\rangle
 \bigl\langle p_1 , p'_2 , x''_3\bigr|x'_1, x'_2, x'_3 \bigr\rangle
 \\
 \qquad
 \equiv
 \int {\rm d}p_1 {\rm d}p_2 {\rm d}p_3
 \Phi_\bb\left(\frac{1}{2}\bigl(x_1 + x'_3 + p_1 - p_2 + p_3 +2\mathcal{C}_7\bigr)\right)\\
 \phantom{\qquad
 \equiv}{} \times
 \e^{{\rm i}\pi(x_2 p_2 +x_3 p_3 + x_1 p_1 - x'_3 p_3 - x'_1 p_1 -x'_2 p_2)}
 .
 \end{gather*}
 Introducing $z_1 = p_2 + p_3 - p_1$, $z_2 = p_3 + p_1 - p_2$,
 $z_3 = p_1 + p_2 - p_3$ and performing the integration over $z_1$
 and $z_3$, we are left with
 \begin{gather*}
 \langle x_1, x_2, x_3|
 \Phi_\bb\left(\frac{1}{2}(\hat{x}_1 + \hat{x}_3 + \hat{p}_1 - \hat{p}_2 + \hat{p}_3 +2\mathcal{C}_7)\right)
 \bigl|x'_1, x'_2, x'_3 \bigr\rangle
 \\
 \qquad\equiv
 \delta\bigl(x_1 - x'_1 + x_2 - x'_2\bigr) \delta\bigl(x_2 - x'_2 + x_3 - x'_3\bigr)\\
 \phantom{ \qquad\equiv}{}\times
 \int {\rm d}z_2 \Phi_\bb\left(\frac{1}{2}\bigl(x_1 + z_2 + x'_3 +2\mathcal{C}_7\bigr)\right)
 \e^{{\rm i} \tfrac{\pi}{2} z_2(x_1 - x'_1 + x_3 - x'_3)}.
 \end{gather*}

 The matrix elements of the second quantum dilogarithm can be
 calculated in a similar manner. This time, $\hat{x}_1 - \hat{x}_3$,
 $\hat{p}_2$ and $\hat{p}_1 + \hat{p}_3$ in the argument mutually
 commute, so we can insert the completeness relation in the basis
 $\{|p_1, p_2, p_3 \rangle\}$ and get
 \begin{gather*}
 \langle x_1, x_2, x_3|
 \Phi_\bb\left(\frac{1}{2}(\hat{x}_1 - \hat{x}_3
 + \hat{p}_1 - \hat{p}_2 + \hat{p}_3 +2\mathcal{C}_5)\right)
 \bigl|x'_1, x'_2, x'_3 \bigr\rangle
 \\
 \qquad \equiv
 \delta\bigl(x_1 - x'_1 + x_2 - x'_2\bigr) \delta\bigl(x_2 - x'_2 + x_3 - x'_3\bigr)\\
 \phantom{ \qquad\equiv}{}\times
 \int {\rm d}z_2 \Phi_\bb\left(\frac{1}{2}(x_1 - x_3 + z_2 + 2\mathcal{C}_5)\right)
 \e^{{\rm i}\tfrac{\pi}{2} z_2(x_1 - x'_1 + x_3 - x'_3)}.
 \end{gather*}

 To calculate the product of the above two matrices, we use the
 Fourier transform identity
 \[
 \int {\rm d}x \Phi_\bb(x)^{\pm1}\e^{2{\rm i}\pi wx}
 =
 \e^{\mp {\rm i}\pi w^2 \pm {\rm i} \tfrac{\pi}{12} (1+4\eta^{2})}
 \Phi_\bb(\pm w\pm {\rm i}\eta)^{\pm1},
 \]
 which is a special case of \eqref{ram1} and \eqref{ram2}.
 We find
 \begin{gather*}
 \langle x_1, x_2, x_3|\Phi_\bb\left(\frac{1}{2}(\hat{x}_1 + \hat{x}_3 + \hat{p}_1 - \hat{p}_2 + \hat{p}_3 +2\mathcal{C}_7)\right)\\
 \qquad
 \times \Phi_\bb\left(\frac{1}{2}(\hat{x}_1 - \hat{x}_3 + \hat{p}_1 - \hat{p}_2 + \hat{p}_3 +2\mathcal{C}_5)\right)\bigl|x'_1, x'_2, x'_3 \bigr\rangle
 \\
 \phantom{ \qquad
 \times }{}\equiv
 \delta\bigl(x_1 + x_2 -x'_1 - x'_2\bigr) \delta\bigl(x_2 + x_3 -x'_2 - x'_3\bigr)
 \\
 \phantom{ \qquad\times\equiv}{} \times
 \int {\rm d}z_2
 \Phi_\bb\left(-\frac{1}{2}(z_2 + x_1 + x_3 + 2\mathcal{C}_7)+{\rm i}\eta\right) \Phi_\bb\left(\frac{1}{2}(z_2 + x_1 - x_3 + 2\mathcal{C}_5)\right)
 \\
 \phantom{ \qquad\times\equiv}{} \times
 \e^{-{\rm i} \tfrac{\pi}{4} (z_2+x_1+x_3+2\mathcal{C}_7)(z_2+x_1+x_3+2\mathcal{C}_7-4{\rm i}\eta)
 - {\rm i} \tfrac{\pi}{2} z_2(x'_3-x_3+x'_1-x_1)}.
 \end{gather*}

 Calculation of the matrix elements of the last two quantum
 dilogarithms can be done analogously. A quick way to write down the
 result is to consider the case $\operatorname{Im} \bb = 0$ or
 $|\bb| = 1$, which allows us to make use of the unitarity
\smash{$\overline{\Phi_{\bb}(z)} = \Phi_{\bb}(\bar{z})^{-1}
$}
 and deduce
 \begin{gather*}
\langle x_1,x_2,x_3|
 \Phi_\bb\left(\frac{1}{2}(\hat{x}_1 - \hat{x}_3 + \hat{p}_1 - \hat{p}_2 + \hat{p}_3 +\tilde{\alpha}_6)\right)^{-1}\\
 \qquad\times
\Phi_\bb\left(\frac{1}{2}(\hat{x}_1 + \hat{x}_3 + \hat{p}_1 - \hat{p}_2 + \hat{p}_3 +\tilde{\alpha}_8)\right)^{-1}
\bigl|x'_1,x'_2,x'_3\bigr\rangle
\\
\phantom{ \qquad\times}{}\equiv \delta\bigl(x_1 + x_2 -x'_1 - x'_2\bigr) \delta\bigl(x_2 + x_3 -x'_2 - x'_3\bigr)
\\
\phantom{ \qquad\times\equiv}{} \times \int{\rm d}z_2 \Phi_\bb\left(\frac{1}{2}(x_1-x_3+z_2+{\tilde{\alpha}_6})\right)^{-1}
\Phi_\bb\left(-\frac{1}{2}\bigl(x'_1+x'_3+z_2+\tilde{\alpha}_8\bigr)-{\rm i}\eta\right)^{-1}
\\
\phantom{ \qquad\times\equiv}{} \times \e^{{\rm i}\tfrac{\pi}{4}(z_2+x'_1+x'_3+\tilde{\alpha}_8)(z_2+x'_1+x'_3+\tilde{\alpha}_8+4{\rm i} \eta)
+{\rm i} \tfrac{\pi}{2}z_2(x_1-x'_1+x_3-x'_3)}.
 \end{gather*}
 Finally, we can also easily calculate
 \begin{gather*}
 \langle x_1, x_2, x_3|\mathcal{P}\bigl|x'_1, x'_2, x'_3 \bigr\rangle
 \\
\qquad \equiv
 \delta\bigl(x_1+x_2-x_3-x'_1+\tilde{\lambda}_{0}\bigr)
 \delta\bigl(x_2-x'_3+\tilde{\lambda}_{0}\bigr)
 \delta\bigl(x_3-x'_2-\tilde{\lambda}_{0}\bigr)
 \e^{-{\rm i}\pi(\tilde{\lambda}_{1}x'_1+\tilde{\lambda}_{3}x'_2+\tilde{\lambda}_{2}x'_3)}.
 \end{gather*}

 From the various matrix elements calculated above, we obtain
 \begin{align*}
 \mathcal{R}^{x_1,x_2,x_3}_{x'_1,x'_2, x'_3}\equiv{}&
 \e^{{\rm i}\pi(\mathcal{C}_4 - \mathcal{C}_5 + \mathcal{C}_7)(\mathcal{C}_5 + \mathcal{C}_6)}
 \delta\bigl(x_1 + x_2 - x'_1 - x'_2\bigr)\delta\bigl(x_2 + x_3 - x'_2 - x'_3\bigr)
 \\
&\times
 \e^{-{\rm i}\tfrac{\pi}{4}(x_1 + x_3 - x'_1 - x'_3 + 2\mathcal{C}_{7} + 2\mathcal{C}_{8} - 4{\rm i}\eta)(x_1 + x_3 + x'_1 + x'_3 + 2\mathcal{C}_{7} - 2\mathcal{C}_{8})
 - {\rm i}\pi(\tilde{\lambda}_{1} x'_1
 + \tilde{\lambda}_{3} x'_2
 + \tilde{\lambda}_{2} x'_3)}
 \\
 & \times
 \int {\rm d}z_2
 \frac{
 \Phi_\bb\bigl(\tfrac{1}{2}(z_2 + x_1 - x_3 + 2\mathcal{C}_5)\bigr)
 \Phi_\bb\bigl(-\tfrac{1}{2}(z_2 + x_1 + x_3 + 2\mathcal{C}_7)+{\rm i}\eta\bigr)
 }{
 \Phi_\bb\bigl(-\tfrac{1}{2}\bigl(z_2 - x'_1 + x'_3 + 2\mathcal{C}_6\bigr)\bigr)
 }\\
  & \times\frac{\e^{-{\rm i}\pi z_2 (\tilde{\lambda}_0 + x'_1 + x'_2 + \mathcal{C}_7 - \mathcal{C}_8)}}{\Phi_\bb\bigl(\tfrac{1}{2}\bigl(z_2 - x'_1 - x'_3 + 2\mathcal{C}_8\bigr) - {\rm i}\eta\bigr)},
\end{align*}
where the first exponential factor comes from the function $f$ in \eqref{fc}.
Changing the integration variable to $z = (z_2 - {\rm i}\eta)/2$ and using
the identity \eqref{pinv}, we arrive at the desired formula.
\end{proof}

Under the transformation \eqref{uwal},
$\mathcal{C}_k = (\pi \bb)^{-1}C_k$ has the same symmetry as that for $C_k$
mentioned in \eqref{alC}.
Therefore, $\mathcal{R}$ is indeed invariant.

\begin{Remark}\label{re:xn}
Comparison of \eqref{qrep} and \eqref{xrep} indicates the correspondence
\begin{align}\label{xn}
x_k = {\rm i}\bb n_k + {\rm i} \eta
\end{align}
between the indices of $R$ and $\mathcal{R}$.
In fact, by using \eqref{lac}, \eqref{pinv} and \eqref{cbc},
one can check that~\smash{$R^{n_1,n_2, n_3}_{n'_1, n'_2, n'_3}$} in Theorem \ref{th:re} is transformed to
\smash{$\mathcal{R}^{x_1,x_2,x_3}_{x'_1,x'_2,x'_3}$} in Theorem \ref{th:Rm} up to normalization by
replacing $\Psi_q\bigl(\e^{2\pi \bb z}\bigr)$ by $\Phi_\bb(z)^{-1}$ according to \eqref{cor}
and substituting \eqref{xn}.
The strange normalization of $\e^{\uu_k}$ in \eqref{qrep} is attributed
to the second term of \eqref{xn}, which may be viewed as a~modular
double analogue of the ``zero point energy''.
\end{Remark}

\subsection[Elements of R in momentum representation]{Elements of $\boldsymbol{\mathcal{R}}$ in momentum representation}

Let us consider the modular ${\mathcal R}$ \eqref{RRf} in the
``momentum representation'' in which $\hat{p}_j$ becomes the diagonal
operator of multiplying $p_j$ as
\begin{gather}
\e^{\pi \bb \hat{x}_k}|\mathbf{p}\rangle = |\mathbf{p}-{\rm i} \bb \mathbf{e}_k\rangle,
\qquad
\e^{\pi \bb \hat{p}_k}|\mathbf{p}\rangle = \e^{\pi \bb p_k}|\mathbf{p}\rangle,\nonumber
\\
\langle \mathbf{p}|\e^{\pi \bb \hat{x}_k} = \langle \mathbf{p}+{\rm i} \bb \mathbf{e}_k|,
\qquad
\langle \mathbf{p}|\e^{\pi \bb \hat{p}_k} = \langle \mathbf{p}|\e^{\pi \bb p_k},\label{prep}
\end{gather}
where $k=1, 2, 3$ and $|p_1,p_2, p_3\rangle$ (resp.\
$\langle p_1,p_2,p_3|$) is denoted by $|\mathbf{p}\rangle$ (resp.\
$\langle \mathbf{p}|$). The dual pairing is specified by
$\bigl\langle \mathbf{p} | \mathbf{p}'\bigr\rangle =
\delta\bigl(p_1-p'_1\bigr)\delta\bigl(p_2-p'_2\bigr)\delta\bigl(p_3-p'_3\bigr)$.

From $\langle \mathbf{x}|\mathbf{p}\rangle \equiv \e^{\pi {\rm i}(p_1x_1+p_2x_2+p_3x_3)}$,
its matrix element
\smash{${\mathcal S}^{p_1,p_2,p_3}_{p'_1,p'_2,p'_3} := \langle p_1,p_2,p_3| \mathcal{R} \bigl|p'_1,p'_2,p'_3\bigr\rangle$}
is obtained by taking the Fourier transformation
\begin{align}\label{ft}
{\mathcal S}^{p_1,p_2,p_3}_{p'_1,p'_2,p'_3} = \int_{\R^6} {\rm d}x_1{\rm d}x_2{\rm d}x_3
{\rm d}x'_1{\rm d}x'_2{\rm d}x'_3
\e^{\pi {\rm i}(p'_1x'_1+p'_2x'_2+p'_3x'_3-p_1x_1-p_2x_2-p_3x_3)}
{\mathcal R}^{x_1,x_2,x_3}_{x'_1,x'_2, x'_3},
\end{align}
where
\smash{${\mathcal R}^{x_1,x_2,x_3}_{x'_1,x'_2, x'_3}$}
is the coordinate representation given in Theorem \ref{th:Rm}.

\begin{Theorem}\label{th:ft}
 Up to an overall factor depending on $\mathcal{C}_1, \dots,
 \mathcal{C}_8$, the following formula is valid:
\begin{gather}
{\mathcal S}^{p_1,p_2,p_3}_{p'_1,p'_2,p'_3} \equiv
\e^{ {\rm i} \pi(\alpha+\beta)}
\frac{\Phi_{\bb}(z_1+{\rm i}\eta)\Phi_{\bb}(z_2+{\rm i}\eta)\Phi_{\bb}(z_3+{\rm i}\eta)\Phi_{\bb}(z_4+{\rm i}\eta)}
{\Phi_{\bb}(z_3+z_4+{\rm i}\eta)},
\label{S1}\\
z_1 = -{\mathcal C}_3-\frac{1}{2}\bigl(p_1+p_3-p'_2\bigr),\qquad
z_3 = -{\mathcal C}_1+{\mathcal C}_3-{\mathcal C}_5-\frac{1}{2}\bigl(-p_1-p_2+p_3+p'_1+p'_2-p'_3\bigr),\nonumber
\\
z_2= {\mathcal C}_2+{\mathcal C}_8-\frac{1}{2}\bigl(p_2-p'_1-p'_3\bigr),\qquad
z_4 = {\mathcal C}_1+{\mathcal C}_3-{\mathcal C}_6-\frac{1}{2}\bigl(p_1-p_2-p_3-p'_1+p'_2+p'_3\bigr),\nonumber
\\
\alpha = (z_1-z_2)(z_3+z_4)+z_3z_4-z_2^2-2{\rm i}\eta z_2,\nonumber
\\
\beta = ({\mathcal C}_8 - {\mathcal C}_7) (z_1+z_2) + ({\mathcal C}_8-{\mathcal C}_6-{\mathcal C}_4) z_3
+ ({\mathcal C}_8-{\mathcal C}_5-{\mathcal C}_4) z_4.\nonumber
 \end{gather}
\end{Theorem}

\begin{proof}
Substitute \eqref{Rm0} into \eqref{ft} and eliminate $x'_1$ and $x'_3$ by the delta functions.
With the shift $x_2 \rightarrow x_2+x'_2$,
the exponent of $\e$ in the result becomes {\em linear} in $x'_2$.
In fact, up to an overall factor, \eqref{ft} is equal to
\begin{gather*}
\int {\rm d}x_1{\rm d}x_2{\rm d}x_3{\rm d}x'_2{\rm d}z
\e^{{\rm i}\pi x'_2(p'_2-p_2-2z-{\rm i}\eta-2{\mathcal C}_3)+{\rm i}\pi\alpha_1}
\nonumber \\
\qquad \times
\frac{\Phi_{\bb}\bigl(z+\frac{x_1-x_3+{\rm i}\eta}{2}+{\mathcal C}_5\bigr)
\Phi_{\bb}\bigl(z+\frac{x_3-x_1+{\rm i}\eta}{2}+{\mathcal C}_6\bigr)}
{\Phi_{\bb}\bigl(z+\frac{x_1+x_3-{\rm i}\eta}{2}+{\mathcal C}_7\bigr)
\Phi_{\bb}\bigl(z+\frac{-x_1-2x_2-x_3-{\rm i}\eta}{2}+{\mathcal C}_8\bigr)},
\\
\alpha_1 =-p_1x_1-p_2x_2-p_3x_3
	+p'_1(x_1+x_2)+p'_3(x_2+x_3)
+(x_1+x_2)(x_2+x_3)
\nonumber\\
 \phantom{\alpha_1 =}{} +{\rm i}\eta(x_1+x_2+x_3)-2({\rm i}\eta+x_2)z+{\mathcal C}_1(x_1-x_3)
+{\mathcal C}_2(x_1+2x_2+x_3)
\nonumber \\
 \phantom{\alpha_1 =}{} +{\rm i}\eta({\mathcal C}_4-2{\mathcal C}_8)+2{\mathcal C}_4z.
\end{gather*}
The integral over $x_2'$ yields $2\delta(p'_2-p_2-2z-{\rm i}\eta-2{\mathcal C}_3)$.
Further integral over $z$ after shifting the contour leads, up to an overall factor, to
\begin{gather*}
\int {\rm d}x_1{\rm d}x_2{\rm d}x_3
\e^{{\rm i}\pi \alpha_2}
\frac{\Phi_{\bb}\bigl(\frac{x_1-x_3+p'_2-p_2}{2}-{\mathcal C}_3+{\mathcal C}_5\bigr)
\Phi_{\bb}\bigl(\frac{x_3-x_1+p'_2-p_2}{2}-{\mathcal C}_3+{\mathcal C}_6\bigr)}
{\Phi_{\bb}\bigl(\frac{x_1+x_3+p'_2-p_2}{2}-{\rm i}\eta-{\mathcal C}_3+{\mathcal C}_7\bigr)
\Phi_{\bb}\bigl(\frac{-x_1-2x_2-x_3+p'_2-p_2}{2}-{\rm i}\eta-{\mathcal C}_3+{\mathcal C}_8\bigr)},
\\
\alpha_2 =\alpha_1\vert_{z=\frac{p'_2-p_2-{\rm i}\eta}{2}-{\mathcal C}_3}.
\end{gather*}
Set $x_2 \rightarrow x_2 - (x_1+x_3)/2$ and
apply \eqref{pinv}
to the second (right) $\Phi_\bb$ in the numerator and the denominator,
which makes the power of $\e$ linear in all the integration variables.
Up to an overall factor, the result reads
\begin{gather*}
\int {\rm d}x_1{\rm d}x_2{\rm d}x_3
\e^{{\rm i}\pi \alpha_3}
\frac{\Phi_{\bb}\bigl(\frac{x_1-x_3+p'_2-p_2}{2}-{\mathcal C}_3+{\mathcal C}_5\bigr)
\Phi_{\bb}\bigl(x_2+\frac{-p'_2+p_2}{2}+{\rm i}\eta+{\mathcal C}_3-{\mathcal C}_8\bigr)}
{\Phi_{\bb}\bigl(\frac{x_1-x_3-p'_2+p_2}{2}+{\mathcal C}_3-{\mathcal C}_6\bigr)
\Phi_{\bb}\bigl(\frac{x_1+x_3+p'_2-p_2}{2}-{\rm i}\eta-{\mathcal C}_3+{\mathcal C}_7\bigr)},
\\
\alpha_3 = x_1\left({\mathcal C}_1-{\mathcal C}_6-p_1+\frac{p_2+p'_1-p'_3}{2}\right)
+ x_2\bigl(2{\mathcal C}_2+2{\mathcal C}_8-p_2+p'_1+p'_3\bigr)
 \\
 \phantom{\alpha_3 =}{} +x_3\left(-{\mathcal C}_1-2{\mathcal C}_3+{\mathcal C}_6-p_3+p'_2+\frac{-p_2-p'_1+p'_3}{2}\right).
\end{gather*}
By setting $x_1=s+t$, $x_3=s-t$, this can be separated into three
independent integrals as
\begin{gather*}
\int {\rm d}s \frac{\e^{{\rm i}\pi s(-2{\mathcal C}_3-p_1-p_3+p'_2)}}
{\Phi_{\bb}\bigl(-{\mathcal C}_3+{\mathcal C}_7+s+\frac{p'_2-p_2}{2}-{\rm i}\eta\bigr)}
\nonumber\\
\qquad \times
\int {\rm d}x_2 \e^{{\rm i}\pi x_2(2{\mathcal C}_2+2{\mathcal C}_8-p_2+p'_1+p'_3)}
\Phi_{\bb}\left({\mathcal C}_3-{\mathcal C}_8+x_2+\frac{p_2-p'_2}{2}+{\rm i}\eta\right)
\nonumber \\
\qquad\times
\int {\rm d}t
\e^{\pi {\rm i}t(2{\mathcal C}_1+2{\mathcal C}_3-2{\mathcal C}_6-p_1+p_2+p_3+p'_1-p'_2-p'_3)}
\frac{\Phi_{\bb}\bigl(-{\mathcal C}_3+{\mathcal C}_5+t+\frac{p'_2-p_2}{2}\bigr)}
{\Phi_{\bb}\bigl({\mathcal C}_3-{\mathcal C}_6+t+\frac{p_2-p'_2}{2}\bigr)},
\end{gather*}
where the Jacobian value 2 has not been included.
They can be evaluated by the formulas in Appendix~\ref{ap:nc}.
After applying \eqref{pinv} again in the result, we obtain~\eqref{S1}.
\end{proof}

By construction, the $R$-matrix in the momentum representation
\smash{$\bigl(\mathcal{S}^{p_1,p_2,p_3}_{p'_1,p'_2,p'_3}\bigr)$} in Theorem \ref{th:ft}
also satisfies the tetrahedron equation.

\begin{Remark}\label{re:kon}
From \eqref{wrep} and \eqref{prep},
one sees the correspondence $q^{n_k} \leftrightarrow \e^{\pi \bb p_k}$, i.e.,
$p_k \leftrightarrow {\rm i} \bb n_k$ in the $w$-diagonal/momentum representation.
In fact, in the formula \eqref{S1}, replace $\Phi_\bb(z+{\rm i}\eta)$ according to
\[
\Phi_\bb(z+{\rm i}\eta) \overset{\eqref{pinv}}{\equiv}
 \frac{\e^{\pi {\rm i}(z+{\rm i}\eta)^2}}{\Phi_\bb(-z-{\rm i}\eta)}
\overset{\eqref{cor}}{\longrightarrow}
\e^{\pi {\rm i}(z+{\rm i}\eta)^2}\Psi_q\bigl(\e^{-2\pi \bb(z+{\rm i}\eta)}\bigr)
\overset{\eqref{Psiq}}{=} \frac{\e^{\pi {\rm i}(z+{\rm i}\eta)^2}}{\bigl(\e^{-2\pi \bb z}\bigr)_\infty}.
\]
Then under the identification $p_k ={\rm i} \bb n_k$, $p'_k = {\rm i} \bb n'_k$
($k=1, 2, 3$) and $C_j = \pi \bb \mathcal{C}_j$ \eqref{cbc}, the
modular $R$-matrix \smash{${\mathcal S}^{p_1,p_2,p_3}_{p'_1,p'_2,p'_3}$} in
\eqref{S1} is transformed to \smash{$S^{n_1,n_2,n_3}_{n'_1,n'_2,n'_3}$} in
\eqref{Sf2} up to normalization. Taking Remark~\ref{re:kmy0} into
account, \smash{$\bigl(\mathcal{S}^{p_1,p_2,p_3}_{p'_1,p'_2,p'_3}\bigr)$} may be
regarded as a modular version of~$R^{ZZZ}$ in~\cite{KMY23}.
\end{Remark}

\section{Relation to quantized six-vertex model}\label{s:6v}

In this section, we show that the $R$-matrix obtained in Section
\ref{s:qw} satisfies the $RLLL = LLLR$ relation ($RLLL$ relation for
short) for the quantized six-vertex model with full parameters
\cite{KMY23}. This result is a quantum version of the observation
made in \cite{GSZ21} that the classical limit of the $RLLL$ relation
arises from a mutation sequence of a symmetric butterfly quiver
associated with a perfect network. We also provide a separate proof
of the $RLLL$ relation for the modular $\mathcal{R}$ based on
properties of the non-compact quantum dilogarithm.

\subsection[3D L operator]{3D $\boldsymbol{L}$ operator}

Let $V = \C v_0 \oplus \C v_1$ be a two-dimensional vector space and
$\mathcal{W}(p)$ be the $p$-Weyl algebra generated by
$Z^{\pm 1}$, $X^{\pm 1}$ with the relation
\begin{equation}\label{zx}
 ZX = pXZ.
\end{equation}
We consider a $\mathcal{W}(p)$-valued operator
\begin{gather}
\mathscr{L}^p = \sum_{a,b,i,j=0,1} E_{ai}\otimes E_{bj} \otimes \mathscr{L}^{ab}_{ij}
\in \mathrm{End}(V \otimes V) \otimes \mathcal{W}(p),
\nonumber
\\
\mathscr{L}^{ab}_{ij}=0\qquad \text{unless}\qquad a+b=i+j,
\label{L2}
\\
\mathscr{L}^{00}_{00} = r,\qquad \mathscr{L}^{11}_{11} = s,\qquad \mathscr{L}^{10}_{10} = w X,\qquad
\mathscr{L}^{01}_{01} = tX,
\qquad \mathscr{L}^{10}_{01} = Z,\nonumber
\\ \mathscr{L}^{01}_{10} = rsZ^{-1}+ tw XZ^{-1}X.
\label{L3}
\end{gather}
Here $r$, $s$, $t$, $w$ are parameters. Note that
$\mathscr{L}^{ab}_{ij}=\mathscr{L}(r,s,t,w)^{ab}_{ij}$ depends also on
$p$ via \eqref{zx}. The symbol $E_{ij}$ denotes the matrix unit on
$V$ acting on the basis as $E_{ij}v_k = \delta_{jk}v_i$. The operator~$\mathscr{L}^p$ may be viewed as a quantized six-vertex model where
the Boltzmann weights are $\mathcal{W}(p)$-valued. It is obtained
from \cite[Figure~1]{KMY23} by (i) gauge transformation
$\mathscr{L}^{ab}_{ij} \rightarrow \alpha^{a-j}\mathscr{L}^{ab}_{ij}$
\cite[Remark~3.23]{K22}, preserving the $RLLL$ relation \eqref{rlll}
described below, with \smash{$\alpha= {\rm i} q^{\frac{1}{2}}$}, (ii)
\smash{$t \rightarrow -{\rm i} q^{-\frac{1}{2}}t$}, (iii) $w \rightarrow t^{-1}w$.%
\footnote{The last term in \eqref{L3} originates from $-t^2wZ^{-1}X^2$
 in \cite[equation~(15)]{KMY23}.
 It is transformed to $q^{-1}twZ^{-1}X^2$
 by (i)--(iii) and further to $twXZ^{-1}X$ by $XZ=qZX$ in
 \cite[equation~(6)]{KMY23} to eliminate the explicit $q$-dependence. The
 relation $XZ = qZX$ \cite[equation~(15)]{KMY23} corresponds to
 $p = q^{-1}$ and differs from the choice $p = q$ \big(or $q^\vee$\big) made
 in this paper.}
See Figure \ref{fig:6v} for a graphical representation.

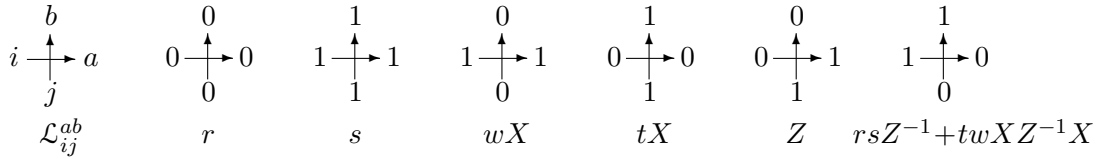
\begin{figure}[h]
\begin{picture}(330,60)(-40,30)
{\unitlength 0.011in
\put(6,80){
\put(-11,0){\vector(1,0){23}}\put(0,-10){\vector(0,1){22}}
}
\multiput(81,80.5)(70,0){6}{
\put(-11,0){\vector(1,0){23}}\put(0,-10){\vector(0,1){22}}
}
\put(-74,0){
\put(60.5,77){$i$}\put(77.5,60){$j$}
\put(96,77){$a$}\put(77.5,96.5){$b$}
}
\put(61,77){0}\put(78,60){0}\put(96,77){0}\put(78,96.5){0}
\put(70,0){
\put(61,77){1}\put(78,60){1}\put(96,77){1}\put(78,96.5){1}
}
\put(140,0){
\put(61,77){1}\put(78,60){0}\put(96,77){1}\put(78,96.5){0}
}
\put(210,0){
\put(61,77){0}\put(78,60){1}\put(96,77){0}\put(78,96.5){1}
}
\put(280,0){
\put(61,77){0}\put(78,60){1}\put(96,77){1}\put(78,96.5){0}
}
\put(350,0){
\put(61,77){1}\put(78,60){0}\put(96,77){0}\put(78,96.5){1}
}
\put(78,40){
\put(-77,0){$\mathscr{L}^{ab}_{ij}$}
\put(0,0){$r$} \put(70,0){$s$} \put(134,0){$wX$}
\put(207,0){$tX$} \put(278,0){$Z$} \put(310,0){$rsZ^{-1}\!+\! tw XZ^{-1}X$}
}}
\end{picture}
\caption{The operator $\mathscr{L}^p = \mathscr{L}^p(r,s,t,w)$ as a
 $\mathcal{W}(p)$-valued six-vertex model.}
\label{fig:6v}
\end{figure}

\subsection[RLLL relation]{$\boldsymbol{RLLL}$ relation}

Consider the $RLLL$ relation, which takes the form of the Yang--Baxter equation up to
conjugation \cite{BMS10,K22}
\begin{equation}
 \mathscr{R}_{456} \mathscr{L}^p_{236} \mathscr{L}^p_{135} \mathscr{L}^p_{124}
 = \mathscr{L}^p_{124} \mathscr{L}^p_{135} \mathscr{L}^p_{236} \mathscr{R}_{456}.
\label{rlll}
\end{equation}
Here $\mathscr{R}_{456}$ is supposed to be an element of a group of operators
whose adjoint action yield linear maps on $\mathcal{W}(p)^{\otimes 3}$.
The indices denote the tensor components on
which the operators act nontrivially. In terms of the components
$\mathscr{L}^{ab}_{ij}$, the $RLLL$ relation reads
\begin{equation}\label{qybe}
 \mathscr{R} \sum_{\alpha, \beta, \gamma=0,1}\bigl(\mathscr{L}^{\alpha \beta}_{ij}
 \otimes \mathscr{L}^{a \gamma}_{\alpha k} \otimes \mathscr{L}^{bc}_{\beta\gamma}\bigr)
 = \sum_{\alpha, \beta, \gamma=0,1}
 \bigl(\mathscr{L}^{ab}_{\alpha\beta} \otimes
 \mathscr{L}^{\alpha c}_{i\gamma} \otimes \mathscr{L}^{\beta\gamma}_{jk}\bigr) \mathscr{R}
\end{equation}
for arbitrary $a, b, c, i, j, k \in \{0,1\}$. See Figure
\ref{fig:qybe}.

\begin{figure}[h]
\hspace*{1.1cm}
\includegraphics[clip,scale=0.4]{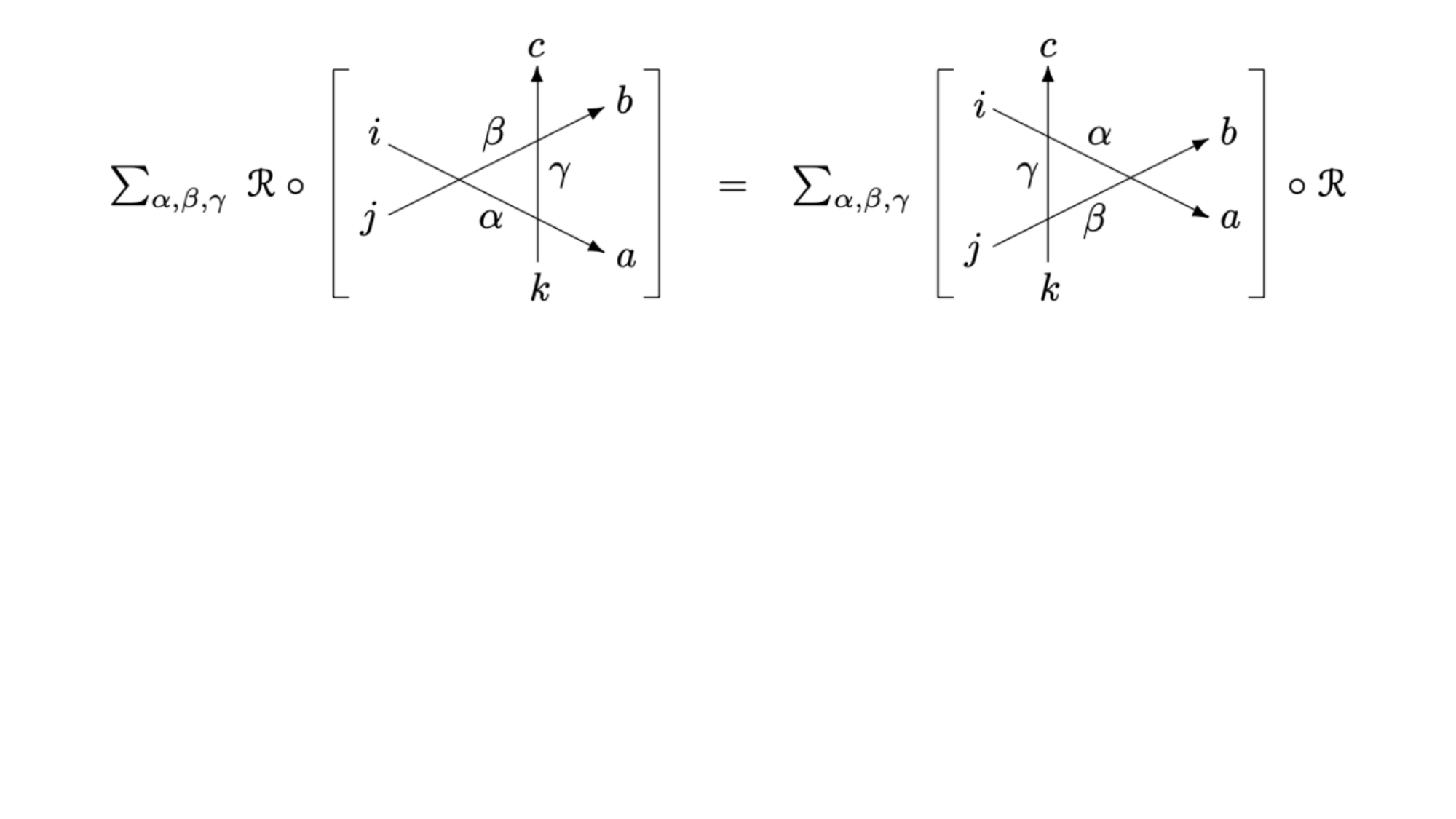}
\vspace*{-5cm}
\caption{A pictorial representation of
the quantized Yang--Baxter equation \eqref{qybe}. }
\label{fig:qybe}
\end{figure}

We take the parameters of $\mathscr{L}^p_{124}$,
$\mathscr{L}^p_{135}$, $\mathscr{L}^p_{236}$ on both sides of
\eqref{rlll} to be $(r_1,s_1,t_1,w_1)$, $(r_2,s_2,t_2,w_2)$,
$(r_3,s_3,t_3,w_3)$, respectively. From the conservation condition
\eqref{L2}, the equation~\eqref{qybe} becomes $0=0$ unless
$a+b+c=i+j+k$. There are 20 choices of $(a,b,c,i,j,k) \in \{0,1\}^6$
satisfying this condition. Among them, the cases $(0,0,0,0,0,0)$ and
$(1,1,1,1,1,1)$ yield the trivial relation
$\mathscr{R}(1 \otimes 1 \otimes 1) = (1 \otimes 1 \otimes
1)\mathscr{R}$. Thus, there are 18 nontrivial equations for
\eqref{rlll}. They are listed in Appendix \ref{app:rlll}.

\begin{Theorem}
 \label{th:rlll}
 The $RLLL$ relation \eqref{rlll} with $p=q$ holds for $\mathscr{R} = R$ in \eqref{rijk} under the identification
 \begin{equation}
 \label{pXZrstw}
 X_i = \e^{\uu_i}, \qquad
 Z_i = \e^{-\ww_i}, \qquad
 r_i = \e^{c_i}, \qquad
 s_i = \e^{a_i}, \qquad
 t_i = \e^{-b_i}, \qquad
 w_i = \e^{-d_i},
 \end{equation}
 where $X_1=X\otimes 1 \otimes 1$, $X_2=1 \otimes X \otimes 1$,
 $X_3=1 \otimes 1 \otimes X$, and $Z_i$ is defined similarly.\footnote{The parameter $w_i$ should not be confused with the
 canonical variable $\ww_i$.}

\end{Theorem}

\begin{proof}
 Write the $RLLL$ relation as
 $\widehat{R}^{\uu\ww}_{456} \bigl(\mathscr{L}^q_{236} \mathscr{L}^q_{135}
 \mathscr{L}^q_{124}\bigr) = \mathscr{L}^q_{124} \mathscr{L}^q_{135}
 \mathscr{L}^q_{236}$. The symmetries
 \eqref{alphauw-Rhat}--\eqref{gammauw-Rhat} of \smash{$\widehat{R}^{\uu\ww}$}
 relate the component equations \eqref{d1}--\eqref{d18} by the
 following three transformations:
 \begin{itemize}\itemsep=0pt
 \item $r_1 \leftrightarrow r_3$, $s_1 \leftrightarrow s_3$,
 $t_1 \leftrightarrow w_3$, $w_1 \leftrightarrow t_3$,
 $t_2 \leftrightarrow w_2$, $X_1 \leftrightarrow X_3$,
 $Z_1 \leftrightarrow Z_3$,

 \item $r_i \leftrightarrow s_i$, $t_i \leftrightarrow w_i$,
 $\mathscr{R} \leftrightarrow \mathscr{R}^{-1}$,

 \item $q \mapsto q^{-1}$, $r_i \leftrightarrow s_i$,
 $t_i \leftrightarrow w_i$,
 $Z_i \leftrightarrow Y_i$,
 \end{itemize}
 where $Y_i$ is defined in Appendix \ref{app:rlll}.
 Accordingly, it suffices to check one equation in each of the
 following four groups:
 \begin{enumerate}\itemsep=0pt
 \item[1)] \eqref{d1}, \eqref{d7}, \eqref{d12}, \eqref{d18},

 \item[2)] \eqref{d3}, \eqref{d9}, \eqref{d10}, \eqref{d16},

 \item[3)] \eqref{d2}, \eqref{d4}, \eqref{d6}, \eqref{d8}, \eqref{d11}, \eqref{d13}, \eqref{d15}, \eqref{d17},

 \item[4)] \eqref{d5}, \eqref{d14}.
 \end{enumerate}

 The equations in (1) follow from \eqref{Y3Y8-uw} and
 \eqref{Y5Y8-uw}. Equation \eqref{d10} is equivalent to
 $\widehat{R}^{\uu\ww}_{123} \circ \smash{\phi'_\mathrm{SB}\bigl(Y_1^{\prime-1}\bigr) =
 \phi_\mathrm{SB} \circ \widehat{R}_{123}\bigl(Y_1^{\prime-1}\bigr)}$, as can be
 seen from \eqref{Yw}, \eqref{Ypw} and Proposition \ref{pr:RY}. One
 can reduce \eqref{d4} to
 \smash{$\widehat{R}^{\uu\ww}_{123} \circ \phi'_\mathrm{SB}\bigl(Y_6^{\prime-1}\bigr) =
 \phi_\mathrm{SB} \circ \widehat{R}_{123}\bigl(Y_6^{\prime-1}\bigr)$} by
 multiplying it by \eqref{d10}. Finally, to verify \eqref{d5}, one
 can check that the relation
 $\eqref{d10} \eqref{d5} = r_1 r_2 r_3 \eqref{d16} + q \eqref{d4}
 \eqref{d11}$ holds whether the left-hand sides or the right-hand
 sides of the equations are used.
\end{proof}

The relation \eqref{pXZrstw} between parameters agrees with \eqref{Css}.

In \cite{KMY23}, it has been shown that
the solutions $R$ to the $RLLL$ relation for the present $L$
are unique up to normalization within appropriate parity sectors \cite{KMY23}.
Thus, Theorem \ref{th:rlll} effectively identifies the concrete $R$-matrices obtained in \cite{KMY23}
with the images of \eqref{RL1} in the corresponding representations of the canonical variables.
Moreover, Theorem \ref{th:main} verifies the validity of the various tetrahedron equations
of the form $RRRR=RRRR$ for these $R$-matrices as conjectured in \cite[Section~6.2]{KMY23}.

\subsection[RLLL relation for the modular R]{$\boldsymbol{RLLL}$ relation for the modular $\boldsymbol{\mathcal{R}}$}

Let $\mathcal{V}$ be the space of ket vectors $|x\rangle (x \in \C)$
(cf. Section \ref{ss:cr}) and
consider the joint representations of $\mathcal{W}(q)$ and
$\mathcal{W}\bigl({q^\vee}\bigr)$ on $\mathcal{V}$ given by
\begin{gather*}
\pi_q\colon\ \mathcal{W}(q) \rightarrow \mathrm{End}(\mathcal{V})\colon\
 X|x\rangle = \e^{\pi \bb x}|x\rangle, \qquad Z|x\rangle = |x-{\rm i}\bb\rangle,
\\
\pi_{q^\vee}\colon\ \mathcal{W}\bigl({q^\vee}\bigr) \rightarrow \mathrm{End}(\mathcal{V})\colon\
 X|x\rangle = \e^{\pi \bb^{-1} x}|x\rangle, \quad Z|x\rangle = \bigl|x-{\rm i}\bb^{-1}\bigr\rangle.
\end{gather*}
See \eqref{qq} for the relations between the parameters $q$, $q^\vee$
and $\bb$. We introduce two $L$-operators that are modular dual to
each other as follows:
\begin{gather*}
\mathcal{L}^q = (1 \otimes 1 \otimes \pi_q)(\mathscr{L}^q) \in \mathrm{End}(V \otimes V \otimes \mathcal{V}),
\\
\mathcal{L}^{q^\vee} =(1 \otimes 1 \otimes \pi_{q^\vee})
\bigl(\mathscr{L}^{q^\vee}\bigr) \in \mathrm{End}(V \otimes V \otimes \mathcal{V}).
\end{gather*}
Since there is no explicit dependence on $p$ in \eqref{L3} or in Figure \ref{fig:6v},
Theorem \ref{th:rlll} implies the following.

\begin{Corollary}\label{co:mod}
 The $R$ matrix $\mathcal{R}$ in Theorem {\rm\ref{th:Rm}} satisfies the RLLL relation
 \begin{equation}
 \mathcal{R}_{456} \mathcal{L}^p_{236} \mathcal{L}^p_{135} \mathcal{L}^p_{124}
 = \mathcal{L}^p_{124} \mathcal{L}^p_{135} \mathcal{L}^p_{236}\mathcal{R}_{456}
 \label{rlllm}
 \end{equation}
 for
 $p = \exp\bigl({\rm i}\pi \bb^{\pm 2}\bigr) = \left(\begin{smallmatrix}q \\
 q^\vee \end{smallmatrix}\right)$ and the parameters $\mathcal{C}_1,
 \dotsc, \mathcal{C}_8$ given by
 \begin{alignat}{4}
& \e^{\pi \bb^{\pm 1} {\mathcal C}_1} = \sqrt{\frac{r_1 t_2 w_3}{r_3 t_1 w_2}}, \qquad&&
 \e^{\pi \bb^{\pm 1} {\mathcal C}_2} = \sqrt{\frac{r_2 t_1 w_3}{r_1 r_3 s_2}}, \qquad&&
 \e^{\pi \bb^{\pm 1} {\mathcal C}_3} = \sqrt{\frac{r_1 r_3}{r_2}}, &\nonumber
 \\
& \e^{\pi \bb^{\pm 1} {\mathcal C}_4} = \sqrt{\frac{r_2 s_2}{t_2 w_2}},\qquad&&
 \e^{\pi \bb^{\pm 1} {\mathcal C}_5} = \sqrt{\frac{r_3 s_3 w_1}{r_2 w_3}}, \qquad&&
 \e^{\pi \bb^{\pm 1} {\mathcal C}_6} = \sqrt{\frac{r_1 s_1 t_3}{r_2 t_1}},&\nonumber
 \\
 &\e^{\pi \bb^{\pm 1} {\mathcal C}_7} = \sqrt{\frac{t_3 w_1}{r_2}}, \qquad&&
 \e^{\pi \bb^{\pm 1} {\mathcal C}_8} = \sqrt{\frac{r_1 r_3 s_1 s_3}{r_2 t_1 w_3}}.\qquad &&&\label{Ckmy}
 \end{alignat}
\end{Corollary}
The upper choice of parameters in \eqref{Ckmy} is consistent with
\eqref{pXZrstw} and \eqref{Css}. (Recall the rescaling \eqref{cbc}.)
The parameters \eqref{Ckmy} satisfy the constraint \eqref{cz}.

 In the rest of this subsection, we illustrate an independent check of \eqref{rlllm} at the level of matrix elements
 in the strong coupling regime
 assuming that ${\mathcal C}_1,\ldots, {\mathcal C}_8$ are all real.
 Thanks to the modular duality, it suffices to consider the $p=q$ case.
 Note also $\eta = \mathrm{Im}(i \bb)$.
 There are three cases (i), (ii), (iii).

 (i) Trivial: \eqref{d1}, \eqref{d7}, \eqref{d12} and \eqref{d18}.
 They are satisfied due to the presence of the two delta functions in
 \eqref{Rm0}.

 (ii) Easy: \eqref{d4}, \eqref{d10}, \eqref{d11}, \eqref{d13},
 \eqref{d16}, \eqref{d17}. They are shown by direct substitution and
 the recursion relation \eqref{prec}. Let us illustrate the
 calculation along the example of \eqref{d17}. The matrix elements
 of $\mathrm{LHS} - \mathrm{RHS}$ for the transition
 $\bigl|x'_1,x'_2,x'_3+{\rm i}\bb\bigr\rangle \mapsto |x_1,x_2,x_3\rangle$ is
\begin{gather*}
w_3\e^{\pi \bb (x'_3+{\rm i}\bb)}\bigl(r_1s_1+t_1w_1\e^{\pi \bb(2x'_1+{\rm i}\bb)}\bigr)
R^{x_1,x_2,x_3}_{x'_1+{\rm i}\bb,x'_2-{\rm i}\bb,x'_3+{\rm i}\bb}
+s_2w_1\e^{\pi\bb x'_1}R^{x_1,x_2,x_3}_{x'_1,x'_2, x'_3}\\
\qquad
-s_1w_2\e^{\pi \bb x_2}
R^{x_1,x_2,x_3+{\rm i}\bb}_{x'_1,x'_2,x'_3+{\rm i}\bb}.
\end{gather*}
Upon substitution of \eqref{Rm0}--\eqref{Rm2} and \eqref{prec}, this
is equal, up to an overall factor, to
\begin{gather*}
\int_{-\infty}^\infty\!\! {\rm d}z \e^{2\pi {\rm i} z(-x_2-{\rm i}\eta+{\mathcal C}_4)}
\frac{\Phi_\bb\bigl(z+\frac{1}{2}(x_1-x_3+{\rm i}\eta)+{\mathcal C}_5\bigr)
\Phi_\bb\bigl(z+\frac{1}{2}(-x_1+x_3+{\rm i}\eta)+{\mathcal C}_6\bigr)}
{\Phi_\bb\bigl(z+\frac{1}{2}(x_1+x_3-{\rm i}\eta)+{\mathcal C}_7\bigr)
\Phi_\bb\bigl(z+\frac{1}{2}\bigl(-x'_1-x'_3-{\rm i}\eta\bigr)-{\rm i} \bb+{\mathcal C}_8\bigr)}\mathcal{D},
\\
\mathcal{D} = \e^{\pi \bb(-2{\mathcal C}_2-2{\mathcal C}_3-x'_1-2{\rm i}\eta)}\bigl(r_1s_1+t_1w_1\e^{\pi \bb(2x'_1+{\rm i}\bb)}\bigr)w_3\\
\phantom{\mathcal{D} =}{}
+\e^{\pi \bb x'_1}\bigl(1+\e^{\pi \bb(2{\mathcal C}_8-{\rm i}\bb -x'_1-x'_3-{\rm i}\eta+2z)}\bigr)s_2w_1
\\
\phantom{\mathcal{D} =}{} -\e^{\pi \bb({\mathcal C}_1-{\mathcal C}_2+{\mathcal C}_4-x'_1-2{\rm i}\eta)}
\bigl(1+\e^{\pi \bb(2{\mathcal C}_5-{\rm i}\bb+x_1-x_3+{\rm i}\eta+2z)}\bigr)s_1w_2.
\end{gather*}
Under the constraint $x_1-x_3=x'_1-x'_3$ deduced from the two delta
functions, $\mathcal{D}=0$ amounts to three equalities
\begin{gather*}
\e^{\pi \bb({\mathcal C}_2+2{\mathcal C}_8)}s_2w_1 = \e^{\pi \bb
 ({\mathcal C}_1+{\mathcal C}_4+2{\mathcal C}_5)}s_1w_2 ,\qquad
\e^{\pi \bb({\mathcal C}_1+{\mathcal C}_2+2{\mathcal C}_3+{\mathcal
 C}_4)}w_2=r_1w_3,\\
\e^{\pi \bb(2{\mathcal C}_2)} \e^{\pi\bb(2{\mathcal C}_3+2{\rm i}\eta)}s_2+\e^{\pi {\rm i}
 \bb^2}t_1w_3=0.
\end{gather*}
 They can be confirmed by using \eqref{Ckmy}.

(iii) The remaining cases \eqref{d2}, \eqref{d3}, \eqref{d5}, \eqref{d6}, \eqref{d8}, \eqref{d9}, \eqref{d14}, \eqref{d15}.
Direct substitution of the formula \eqref{Rm0} with
appropriate shift of the integration variable $z$ and the application
of \eqref{prec} lead to
\smash{$\mathrm{LHS}-\mathrm{RHS} = \int_{\R + {\rm i} f} {\rm d}z \Xi(z)$},
where $f \in \R$ represents a freedom to shift the integration contour.
Although~$\Xi(z)$ is not identically vanishing, one can always find~${\tilde \Xi}(z)$ such that
$\Xi(z) = {\tilde \Xi}(z+{\rm i} \bb)-{\tilde \Xi}(z)$.\footnote{An analogous
treatment can also be found in the proof of \cite[Theorem~3.18]{K22}.}
Thus, the claim reduces to the analyticity of~${\tilde \Xi}(z)$ in
$f< \operatorname{Im} z < f+\eta$ and the
damping in
$\operatorname{Re}(z) \rightarrow \pm \infty$ in this strip.
As an example, consider~\eqref{d6}, whose elements of $\text{LHS}-\text{RHS}$ for the transition
$\bigl|x'_1 - {\rm i} \bb,x'_2,x'_3\bigr\rangle \mapsto |x_1,x_2,x_3\rangle$ read
\begin{gather}
w_1 \e^{\pi \bb (x'_1-{\rm i} \bb)}\bigl(r_2s_2+t_2w_2\e^{\pi \bb(2x'_2+{\rm i}\bb)}\bigr)
 R^{x_1,x_2,x_3}_{x'_1-{\rm i} \bb, x'_2+{\rm i}\bb, x'_3-{\rm i}\bb} \nonumber
 \\
 \qquad{}+ r_2w_3 \e^{\pi \bb x'_3}\bigl(r_1s_1+t_1w_1\e^{\pi \bb(2x'_1-{\rm i}\bb)}\bigr)
 R^{x_1,x_2,x_3}_{x'_1,x'_2,x'_3}\nonumber
 \\
 \qquad{}- r_3w_2\e^{\pi \bb x_2}\bigl(r_1s_1+t_1w_1\e^{\pi \bb(2x_1-{\rm i}\bb)}\bigr)
 R^{x_1-{\rm i}\bb,x_2,x_3}_{x'_1-{\rm i}\bb,x'_2,x'_3},
 \label{dd6}
\end{gather}
where $x_1+x_2=x'_1+x'_2$ and $x_2+x_3=x'_2+x'_3$ are assumed in view of the two delta functions
in \eqref{Rm0}.
After shifting the integration variable $z$ for the last term to
$z+{\rm i} \bb/2$, the integrand for this expression can be shown to be
proportional to ${\tilde \Xi}(z+{\rm i} \bb)-{\tilde \Xi}(z)$ with
\[
{\tilde \Xi}(z) =
\e^{2\pi {\rm i}z(-x_2-{\rm i}\eta+{\mathcal C}_4)}
\frac{\Phi_\bb\bigl(z+\frac{1}{2}(x_1-x_3+{\rm i}\eta)+{\mathcal C}_5-{\rm i} \bb\bigr)
\Phi_\bb\bigl(z+\frac{1}{2}(-x_1+x_3+{\rm i}\eta)+{\mathcal C}_6\bigr)}
{\Phi_\bb\bigl(z+\frac{1}{2}(x_1+x_3-{\rm i}\eta)+{\mathcal C}_7\bigr)
\Phi_\bb\bigl(z+\frac{1}{2}\bigl(-x'_1-x'_3-{\rm i}\eta\bigr)+{\mathcal C}_8\bigr)}.
\]
Consider the region $-\eta < \operatorname{Im} x_2 = \operatorname{Im}x'_2<0$
with $x_1,x_3, x'_1, x'_3 \in \R$, which is compatible with the above mentioned condition.
From \eqref{pz}, ${\tilde \Xi}(z)$ is analytic in the strip ${-\eta/2 < \operatorname{Im}z < \eta/2}$.
Moreover, from \eqref{pa} and $\mathcal{C}_5+\mathcal{C}_6=\mathcal{C}_7+\mathcal{C}_8$,
${\tilde \Xi}(z)$ asymptotically tends to
$\e^{2\pi {\rm i} z(-x_2-{\rm i}\eta+\mathcal{C}_4)}$ for~${\mathop{\operatorname{Re}}z \rightarrow -\infty}$
and to
$\e^{2\pi {\rm i} z(-x'_2+{\rm i}\eta-{\rm i} \bb +\mathcal{C}_4)}$ for $\mathop{\operatorname{Re}}z \rightarrow +\infty$.
Therefore, from $0<\eta<1$, they are both decaying if~${-\eta < \operatorname{Im} x_2 = \operatorname{Im}x'_2<0}$ as long as the limit
$|\mathop{\operatorname{Re}}z| \rightarrow \infty$ is taken in the strip
$-\eta/2 < \operatorname{Im}z < \eta/2$.
(This corresponds to the choice $f=-\eta/2$.)
This verifies \eqref{d2}.

As seen in this example,
the precise region of $x_i$, $x'_i$ which validates such a check
is sensitive to how they appear shifted as the indices of $R$'s
as in \eqref{dd6}, and they indeed vary case by case.
We have checked, for all the equations of (iii), that
there is ${\tilde \Xi}(z)$ having a similar `factorized' form, and
there is a subregion of $-\eta < \operatorname{Im} x_2 = \operatorname{Im}x'_2< \eta$
with $x_1,x_3, x'_1, x'_3 \in \R$ which assures that
${\tilde \Xi}(z)$ possesses a strip $f< \operatorname{Im} z < f+\eta$
where it behaves in the same manner as the above example.

\begin{Remark}\label{re:bms}
 As a corollary of Theorem \ref{th:rlll}, it is evident that
 postulating the $RLLL$ relation~\eqref{rlllm} for $p=q$ and
 $q^\vee$ \emph{simultaneously} compels
 ${\mathcal C}_1 = \dotsb = {\mathcal C}_8=0$. The resulting
 parameter-free (except $\bb$) $R$-matrix \eqref{Rm0}--\eqref{Rm2}
 exactly reproduces \cite[equation~(51)]{BMS10}. We refer to this
 particular case as the {\em modular double} $\mathcal{R}$.
\end{Remark}

In Section \ref{s:md}, the term ``modular $\mathcal{R}$'' (without
``double'') is deliberately used to distinctively describe the results
with full parameters.
Note on the other hand that the condition
${\mathcal C}_1 = \dotsb = {\mathcal C}_8=0$ still leaves five free
parameters among $(r_j,s_j,t_j,w_j)_{j=1,2,3}$.

\section{Reduction to the Fock--Goncharov quiver}\label{s:red}

In this section, we explain that our $R$-matrix \eqref{RL0} for the symmetric butterfly (SB) quiver
reduces to that for the Fock--Goncharov (FG) quiver \cite{IKT1} in a certain limit of parameters.

\subsection[R-matrix for the FG quiver]{$\boldsymbol{R}$-matrix for the FG quiver}

Let $\mathrm{p}_i$, $\mathrm{u}_i$ ($i=1, 2, 3$) be canonical
variables obeying $[\mathrm{p}_i, \mathrm{u}_j] = \delta_{ij}\hbar$,
$[\mathrm{p}_i, \mathrm{p}_j] = [\mathrm{u}_i, \mathrm{u}_j] =0$.
Recall the $R$-matrix in \cite[equation~(4.14)]{IKT1}\footnote{The
 parameter $\lambda_i$ in \cite[equation~(3.5)]{IKT1} is denoted by
 $\theta_i$ here.} given as
\begin{align}\label{rfg}
R_\mathrm{FG} =
\Psi_q\bigl(\e^{\theta_1-\theta_3+\mathrm{p}_1+\mathrm{u}_1+\mathrm{p}_3
-\mathrm{u}_3-\mathrm{p}_2}\bigr) P_\mathrm{FG},
\qquad
P_\mathrm{FG} = \rho_{23}\e^{\tfrac{1}{\hbar}\mathrm{p}_1(\mathrm{u}_3-\mathrm{u}_2)}
\e^{\tfrac{\theta_2-\theta_3}{\hbar}(\mathrm{u}_3-\mathrm{u}_1)},
\end{align}
where $\theta_i$' s are parameters.
It has been deduced from
$\hat{R}_\mathrm{FG} = \mathrm{Ad}(R_\mathrm{FG})$,
where $\hat{R}_\mathrm{FG} $ is the cluster transformation
corresponding to $\mu^\ast_4$ in the
FG quivers depicted as follows:
\begin{align}
\label{FG}
\begin{split}&
\begin{tikzpicture}
\begin{scope}[>=latex,xshift=0pt]
{\color{red}
\fill (1,0.5) circle(2pt) coordinate(A) node[below]{$1$};
\fill (2,1.5) circle(2pt) coordinate(B) node[above]{$2$};
\fill (3,0.5) circle(2pt) coordinate(C) node[below]{$3$};
}
\coordinate (P1) at (4.5,1);
\coordinate (P2) at (5.5,1);
{\color{red}
\draw[<-] (4.5,1) -- (5.5,1);
\draw (5,1) circle(0pt) node[below]{$\hat{R}_\mathrm{FG}$};
}
{\color{blue}
\draw (0,0.5) circle(2pt) coordinate(B1) node[below]{$3$};
\draw (2,0.5) circle(2pt) coordinate(C1) node[below]{$4$};
\draw (4,0.5) circle(2pt) coordinate(F1) node[below]{$5$};
\draw (1,1.5) circle(2pt) coordinate(D1) node[above]{$1$};
\draw (3,1.5) circle(2pt) coordinate(E1) node[above]{$2$};
\qarrow{B1}{C1}
\qarrow{C1}{D1}
\qarrow{D1}{E1}
\qarrow{E1}{C1}
\qarrow{C1}{F1}
\qdarrow{D1}{B1}
\qdarrow{F1}{E1}
}
\end{scope}
\begin{scope}[>=latex,xshift=175pt]
{\color{red}
\fill (3,1.5) circle(2pt) coordinate(A) node[above]{$1$};
\fill (2,0.5) circle(2pt) coordinate(B) node[below]{$2$};
\fill (1,1.5) circle(2pt) coordinate(C) node[above]{$3$};
}
{\color{blue}
\draw (0,1.5) circle(2pt) coordinate(B1) node[above]{$1$};
\draw (2,1.5) circle(2pt) coordinate(C1) node[above]{$4$};
\draw (4,1.5) circle(2pt) coordinate(D1) node[above]{$2$};
\draw (1,0.5) circle(2pt) coordinate(E1) node[below]{$3$};
\draw (3,0.5) circle(2pt) coordinate(F1) node[below]{$5$};
\qarrow{B1}{C1}
\qarrow{C1}{D1}
\qarrow{E1}{F1}
\qarrow{F1}{C1}
\qarrow{C1}{E1}
\qdarrow{E1}{B1}
\qdarrow{D1}{F1}
}
\end{scope}
\end{tikzpicture}
\end{split}
\end{align}
As in Figure \ref{fig:R123}, the dots marked $1$, $2$, $3$ in red
signify the crossings of the associated wiring diagram, to which the
canonical variables are attached.

Let $B_{\mathrm{FG}}$ and $B'_{\mathrm{FG}}$ be (the exchange matrices
of) the left and the right quivers in \eqref{FG}, respectively. Let
$\mathcal{Y}(B_{\mathrm{FG}})$ be the skew field generated by
$\mathscr{Y}_1, \dots, \mathscr{Y}_5$ which are attached to the
vertices of $B_{\mathrm{FG}}$ and obey the commutation relation
\eqref{yyq} where $b_{ij}$ is taken to be the elements of
$B_{\mathrm{FG}}$. Define $\mathcal{Y}(B'_{\mathrm{FG}})$ generated
by $\mathscr{Y}'_1, \dots, \mathscr{Y}'_5$ from $B'_{\mathrm{FG}}$
similarly.

Let $\mathcal{W}^{\mathrm{pu}}_3$ be the direct product of the
$q$-Weyl algebras generated by $\e^{\pm \mathrm{p}_i}$,
$\e^{\pm \mathrm{u}_i}$ for $i=1, 2, 3$. The fractional field of
$\mathcal{W}^{\mathrm{pu}}_3$ is denoted by
$\mathcal{A}^{\mathrm{pu}}_3$. We denote the isomorphism $\pi_{123}$
of $\mathcal{W}^{\mathrm{pu}}_3$ \cite[equation~(3.7)]{IKT1} (in the sense
of exponentials) by $\pi_\mathrm{FG}$ and recall the embeddings
$\phi_\mathrm{FG}\colon \mathcal{Y}(B_\mathrm{FG}) \hookrightarrow
\mathcal{A}^{\mathrm{pu}}_3$ and
$\phi'_\mathrm{FG}\colon \mathcal{Y}(B'_\mathrm{FG}) \hookrightarrow
\mathcal{A}^{\mathrm{pu}}_3$ \cite[equation~(3.6)]{IKT1} given as
\begin{align*}
&\pi_\mathrm{FG}\colon\
\begin{cases}
\mathrm{p}_1 \mapsto \mathrm{p}_1+ \theta_2-\theta_3,
\qquad
\mathrm{p}_2 \mapsto \mathrm{p}_1+\mathrm{p}_3,
\qquad
\mathrm{p}_3 \mapsto \mathrm{p}_2-\mathrm{p}_1- \theta_2+\theta_3,
\\
\mathrm{u}_1 \mapsto \mathrm{u}_1+\mathrm{u}_2-\mathrm{u}_3,
\qquad
\mathrm{u}_2 \mapsto \mathrm{u}_3,
\qquad
\mathrm{u}_3 \mapsto \mathrm{u}_2,
\end{cases}
\\
&\phi_\mathrm{FG}\colon\
\begin{cases}
\mathscr{Y}_1 \mapsto \e^{-\theta_2+\mathrm{p}_2-\mathrm{u}_2-\mathrm{p}_1},
\\
\mathscr{Y}_2 \mapsto \e^{\theta_2+\mathrm{p}_2+\mathrm{u}_2-\mathrm{p}_3},
\\
\mathscr{Y}_3 \mapsto \e^{-\theta_1+\mathrm{p}_1-\mathrm{u}_1},
\\
\mathscr{Y}_4 \mapsto \e^{\theta_1-\theta_3+\mathrm{p}_1
+\mathrm{u}_1+\mathrm{p}_3-\mathrm{u}_3-\mathrm{p}_2},
\\
\mathscr{Y}_5 \mapsto \e^{\theta_3+\mathrm{p}_3+\mathrm{u}_3},
\end{cases}
\qquad
\phi'_\mathrm{FG}\colon\
\begin{cases}
\mathscr{Y}'_1 \mapsto \e^{-\theta_3 +\mathrm{p}_3-\mathrm{u}_3},
\\
\mathscr{Y}'_2 \mapsto \e^{\theta_1+\mathrm{p}_1+\mathrm{u}_1},
\\
\mathscr{Y}'_3 \mapsto \e^{-\theta_2+\mathrm{p}_2-\mathrm{u}_2-\mathrm{p}_3},
\\
\mathscr{Y}'_4 \mapsto \e^{-\theta_1+\theta_3+\mathrm{p}_3
+\mathrm{u}_3+\mathrm{p}_1-\mathrm{u}_1-\mathrm{p}_2},
\\
\mathscr{Y}'_5 \mapsto \e^{\theta_2+\mathrm{p}_2+\mathrm{u}_2-\mathrm{p}_1}.
\end{cases}
\end{align*}
One has $\pi_\mathrm{FG} = \mathrm{Ad}(P_\mathrm{FG})$.
With these notation, the $R$-matrix \eqref{rfg} is rephrased as
\[
R_\mathrm{FG} =
\Psi_q(\phi_\mathrm{FG}(\mathscr{Y}_4))P_\mathrm{FG}.
\]

\subsection{Embedding FG into SB}

We employ a parallel notation
$\mathcal{Y}(B_{\mathrm{SB}})$ and
$\mathcal{Y}(B'_{\mathrm{SB}})$ to signify the
skew fields corresponding to the left and the right quivers in Figure \ref{fig:R123}.
It is easy to see that the following maps yield morphisms of the skew
fields
$\alpha\colon \mathcal{Y}(B_{\mathrm{FG}}) \rightarrow
\mathcal{Y}(B_{\mathrm{SB}})$ and
$\alpha'\colon \mathcal{Y}(B'_{\mathrm{FG}}) \rightarrow
\mathcal{Y}(B'_{\mathrm{SB}})$
\begin{align}\label{Ymor}
\alpha\colon\ \begin{cases}
\mathscr{Y}_1 \mapsto Y_9,
\\
\mathscr{Y}_2 \mapsto q Y_8Y_7,
\\
\mathscr{Y}_3 \mapsto Y_6,
\\
\mathscr{Y}_4 \mapsto q Y_5Y_4,
\\
\mathscr{Y}_5 \mapsto qY_3Y_2,
\end{cases}
\qquad
\alpha'\colon\ \begin{cases}
\mathscr{Y}'_1 \mapsto Y'_9,
\\
\mathscr{Y}'_2 \mapsto q Y'_5Y'_7,
\\
\mathscr{Y}'_3 \mapsto Y'_6,
\\
\mathscr{Y}'_4 \mapsto q Y'_3Y'_4,
\\
\mathscr{Y}'_5 \mapsto qY'_8Y'_2.
\end{cases}
\end{align}

Recall that $\mathcal{A}_3$ defined after Figure \ref{fig:para} for
the SB quiver is a fractional field of $\mathcal{W}_3$ in which
$\e^{\uu_i}\e^{\ww_j}=q^{\delta_{ij}\hbar}\e^{\ww_j}\e^{\uu_i}$. On the other
hand, $\mathcal{A}^{\mathrm{pu}}_3$ for the FG quiver in the previous
subsection is a~fractional field of $\mathcal{W}^{\mathrm{pu}}_3$ in
which
$\e^{\mathrm{p}_i}\e^{\mathrm{u}_i}
=q^{\delta_{ij}\hbar}\e^{\mathrm{u}_i}\e^{\mathrm{p}_i}$ . Thus there
is an isomorphism
$\beta\colon \mathcal{W}^{\mathrm{pu}}_3 \rightarrow \mathcal{W}_3$
given by
\begin{align}\label{bed}
\beta\colon\ \mathrm{p}_i \mapsto -\ww_i,\qquad \mathrm{u}_i \mapsto \uu_i, \qquad i=1,2,3,
\end{align}
in the sense of exponentials.
We consider the diagram
\begin{align}\label{cd3}
\begin{CD}
\mathcal{Y}(B_\mathrm{FG}) @> {\alpha}>> \mathcal{Y}(B_\mathrm{SB}) \\
@V{\phi_\mathrm{FG}}VV @VV{\phi_\mathrm{SB}}V \\
\mathcal{W}^{\mathrm{pu}}_3 @>{\beta}>> \mathcal{W}_3\\
@V{\pi_\mathrm{FG}}VV @VV{\tau^{\uu\ww}_{--++}}V \\
\mathcal{W}^{\mathrm{pu}}_3 @>{\beta}>> \mathcal{W}_3\\
@A{\phi'_\mathrm{FG}}AA @AA{\phi'_\mathrm{SB}}A \\
\mathcal{Y}(B'_\mathrm{FG}) @> {\alpha'}>> \mathcal{Y}(B'_\mathrm{SB}),
\end{CD}
\end{align}
where $\phi_\mathrm{SB}$, $\phi'_\mathrm{SB}$ and $\tau^{\uu\ww}_{--++}$ are defined
in \eqref{Yw}, \eqref{Ypw} and \eqref{uwL1}, respectively.

\begin{Proposition}\label{pr:cd3}
 The diagram \eqref{cd3} is commutative if and only if the parameters
 $\theta_i$ $(i=1, 2, 3)$ and $(a_i,b_i, c_i, d_i, e_i)$ subject
 to \eqref{ae0} satisfy the relations
\begin{align}
&e_2 = e_3, \qquad a_1=-a_3=c_3=-c_1, \qquad a_2=c_2=0,
\label{pare1}
\\
&\theta_1 = -b_1, \qquad \theta_2 = -a_1-b_2, \qquad \theta_3= d_3+e_3.
\label{pare2}
\end{align}
\end{Proposition}
\begin{proof}
 For the top square, it suffices to consider the image of
 $\mathscr{Y}_i \in \mathcal{Y}(B_\mathrm{FG})$ ($i=1, \dots, 5$).
 For instance, one has
 $\phi_\mathrm{SB}\circ \alpha(\mathscr{Y}_1) =
 \phi_\mathrm{SB}(Y_9)= \e^{a_1+b_2+\ww_1-\uu_2-\ww_2}$ and
 $\beta \circ \phi_\mathrm{FG}(\mathscr{Y}_1) =
 \beta\bigl(\e^{-\theta_2+\mathrm{p}_2-\mathrm{u}_2-\mathrm{p}_1}\bigr) =
 \e^{-\theta_2-\ww_2-\uu_2+\ww_1}$, hence the commutativity requires
 $\theta_2 = -a_1-b_2$. A similar calculation leads to
\begin{gather*}
\theta_2 = -a_1-b_2 = e_2+d_2+a_3, \qquad \theta_1 = -b_1, \qquad
\theta_1-\theta_3 = e_1+d_1+c_2+b_3, \\ \theta_3 = e_3+d_3.
\end{gather*}
For the middle square, it suffices to consider the image of
$\mathrm{p}_i$, $\mathrm{u}_i$, ($i=1, 2, 3$). The commutativity
leads to
\begin{align*}
\theta_2-\theta_3 = -\lambda_1 = \lambda_2, \qquad
\lambda_0=\lambda_3=0,
\end{align*}
where $\lambda_i$'s are specified in \eqref{lad}.
These nine relations are equivalent to \eqref{pare1} and \eqref{pare2}.
The commutativity of the bottom square follows from them.
\end{proof}

\subsection[R\_FG as a limit of R\_SB]{$\boldsymbol{R_\mathrm{FG}}$ as a limit of $\boldsymbol{R_\mathrm{SB}}$}

Let $R_\mathrm{SB}$ be the $R$-matrix \eqref{RL0} for the SB quiver under the specialization of the
parameters~\eqref{pare1} and \eqref{pare2}.
Explicitly, we have
\begin{align*}
R_\mathrm{SB} ={}& \Psi_q\bigl(\e^{-\Lambda+e_1+\uu_1+\uu_3+\ww_1-\ww_2+\ww_3}\bigr)^{-1}
\Psi_q\bigl(\e^{-\Lambda-e_3+e_1+\uu_1-\uu_3+\ww_1-\ww_2+\ww_3}\bigr)^{-1}
 \\
& \times
\Psi_q\bigl(\e^{\theta_1-\theta_3+\uu_1-\uu_3-\ww_1+\ww_2-\ww_3}\bigr)
\Psi_q\bigl(\e^{\Lambda+e_2+\uu_1+2\uu_2-\uu_3-\ww_1+\ww_2-\ww_3}\bigr)P_\mathrm{SB},
\\
P_\mathrm{SB} ={}&
\e^{\tfrac{1}{\hbar}(\uu_3-\uu_2)\ww_1}\e^{\tfrac{\theta_3-\theta_2}{\hbar}(\uu_1-\uu_2)}\rho_{23}
=
\rho_{23}\e^{\tfrac{1}{\hbar}(\uu_2-\uu_3)\ww_1}\e^{\tfrac{\theta_3-\theta_2}{\hbar}(\uu_1-\uu_3)},
\end{align*}
where $\Lambda =\theta_1-\theta_3= d_1+e_1+c_2+b_3$.
The above formula for $P_\mathrm{SB}$ follows from \eqref{Pbch1} under the specialization.

\begin{Theorem}\label{th:fg}
The $R$-matrix $R_\mathrm{FG}$ is reproduced from the specialized $R$-matrix
$R_\mathrm{SB}$ as
\[
\lim R_\mathrm{SB} = \beta(R_\mathrm{FG}),
\]
where the limit is taken as
\begin{gather}
e_1 \rightarrow -\infty, \qquad
e_2=e_3 \rightarrow -\infty, \qquad
e_1-e_3 \rightarrow -\infty,\nonumber
\\
e_i+d_i= \mathrm{finite}, \qquad i=1,2,3.\label{elim}
\end{gather}
\end{Theorem}
\begin{proof}
Since $\Lambda$ remains finite in the limit,
one has
$\lim R_\mathrm{SB} = \Psi_q\bigl(\e^{\theta_1-\theta_3+\uu_1-\uu_3-\ww_1+\ww_2-\ww_3}\bigr)
P_\mathrm{SB}$.
By comparing this with \eqref{rfg}, the claim is checked easily.
\end{proof}

\begin{Remark}
Parallel results which fit the formula \eqref{RL3} can also be formulated.
One replaces~\eqref{Ymor} with
\begin{align*}
\alpha\colon\ \begin{cases}
\mathscr{Y}_1 \mapsto Y_7,
\\
\mathscr{Y}_2 \mapsto q Y_8Y_9,
\\
\mathscr{Y}_3 \mapsto Y_2,
\\
\mathscr{Y}_4 \mapsto q Y_3Y_4,
\\
\mathscr{Y}_5 \mapsto qY_5Y_6,
\end{cases}
\qquad
\alpha'\colon\ \begin{cases}
\mathscr{Y}'_1 \mapsto Y'_7,
\\
\mathscr{Y}'_2 \mapsto q Y'_3Y'_9,
\\
\mathscr{Y}'_3 \mapsto Y'_2,
\\
\mathscr{Y}'_4 \mapsto q Y'_5Y'_4,
\\
\mathscr{Y}'_5 \mapsto qY'_8Y'_6,
\end{cases}
\end{align*}
and \eqref{bed} with
$\beta\colon
(\mathrm{p}_1,\mathrm{p}_2,\mathrm{p}_3,\mathrm{u}_1,\mathrm{u}_2,\mathrm{u}_3)
\mapsto (-\ww_3,-\ww_2,-\ww_1,\uu_3,\uu_2,\uu_1)$. Then the diagram \eqref{cd3}
in which $\tau^{\uu\ww}_{--++}$ is replaced with $\tau^{\uu\ww}_{-+-+}$
\eqref{uwL2} becomes commutative if and only if \eqref{pare1} and $(\theta_1,\theta_2, \theta_3)=(-d_3, -d_2-a_3, b_1+e_1)$ hold.
\end{Remark}

\subsection[A limiting procedure for modular R]{A limiting procedure for modular $\boldsymbol{R}$}\label{ss:fg}

Let us demonstrate an explicit limiting procedure that essentially corresponds to
Theorem \ref{th:fg} in the context of the modular $R$ in the coordinate representation.
In the strong coupling regime~${0<\eta<1}$, one has $|\bb|=1$ and
$\mathop{\operatorname{Re}} \bb >0$. Thus the product representation
\eqref{ncq} indicates
\begin{equation}\label{plim}
\lim_{z \rightarrow -\infty}\Phi_\bb(z) = 1.
\end{equation}

\begin{Proposition}\label{pr:lim}
Specialize the parameters in $\mathcal{C}_1,\dotsc, \mathcal{C}_8$ in
Theorem {\rm\ref{th:Rm}} as
\begin{gather*}
{\mathcal C}_1 = \frac{1}{2}(\zeta-\xi)\ell_{21}-\frac{1}{2}\zeta\ell_{31},
\qquad
{\mathcal C}_2 = -{\mathcal C}_8 = -T - \frac{1}{2}\zeta\ell_{31}, \qquad
{\mathcal C}_3=0,
\\
{\mathcal C}_4 = T + \frac{1}{2}(\xi+\zeta)\ell_{21},\qquad
{\mathcal C}_5 = \frac{1}{2}\zeta\ell_{31},
\qquad
{\mathcal C}_7 = {\mathcal C}_5+{\mathcal C}_6-{\mathcal C}_8,
\end{gather*}
where $\ell_{ij}=\ell_i-\ell_j$.
Then the elements of the $R$-matrix associated with the FG quiver in {\rm\cite[\emph{Proposition}~7.4]{IKT1}},
with the exchange of components $1 \leftrightarrow 3$,
are reproduced as a limit of \eqref{Rm0} as follows:
\begin{gather*}
\lim_{\substack{ T \rightarrow \infty\\{\mathcal C}_6,{\mathcal C}_7\rightarrow -\infty}}
\e^{-{\rm i}\pi\gamma}g\bigl(\tilde{a}, \tilde{b}, \tilde{c}, \tilde{d}\bigr)^{-1}{\mathcal R}^{x_1,x_2,x_3}_{x'_1,x'_2, x'_3}
\\
\qquad= \delta\bigl(x_1+x_2-x'_1-x'_2\bigr)\delta\bigl(x_2+x_3-x'_2-x'_3\bigr)
\\
 \phantom{\qquad= }{} \times
\Phi_\bb\left(x_2-x'_1+{\rm i}\eta-\frac{1}{2}(\xi+\zeta)\ell_{21}\right)
\e^{{\rm i}\pi(x_2-x'_1-\xi\ell_{21})(x_1-x_3+\zeta\ell_{31}-2{\rm i}\eta)},
\end{gather*}
where
$\gamma= -T^2+({\rm i}\eta -\zeta \ell_{31})T -\tfrac{1}{2}{\rm i}\eta \ell_{21}(3\xi-\zeta)
+\tfrac{1}{2}\ell_{31}\zeta(\ell_{21}(\xi-\zeta)+2{\rm i}\eta)+\tfrac{1+4\eta^2}{12}
$.
\end{Proposition}
\begin{proof}
Due to \eqref{plim},
in the limit ${\mathcal C}_6, {\mathcal C}_7 \rightarrow -\infty$,
the integral \smash{$ I^{x_1,x_2,x_3}_{x'_1,x'_2, x'_3} $} in \eqref{Rm1} simplifies to
\begin{gather*}
 I^{x_1,x_2,x_3}_{x'_1,x'_2, x'_3} \rightarrow
\int_{-\infty}^\infty {\rm d}z \e^{2\pi {\rm i}z(-x_2-{\rm i}\eta+{\mathcal C}_4)}
\frac{\Phi_\bb\bigl(z+\frac{1}{2}(x_1-x_3+{\rm i}\eta\bigr)+{\mathcal C}_5\bigr)}
{\Phi_\bb\bigl(z+\frac{1}{2}\bigl(-x'_1-x'_3-{\rm i}\eta\bigr)+{\mathcal C}_8\bigr)}
\\
\qquad= \e^{ {\rm i} \pi \nu}\frac{\Phi_\bb\bigl(x_2-x'_1+{\rm i}\eta-{\mathcal C}_4-{\mathcal C}_5+{\mathcal C}_8\bigr)
\Phi_\bb\bigl(x'_1+{\mathcal C}_5-{\mathcal C}_8\bigr)}{\Phi_\bb(x_2-{\mathcal C}_4)},
\end{gather*}
where we have evaluated the integral by \eqref{rama2}, set
$x_1=x'_1-x'_3+x_3$ and then applied \eqref{pinv} in the result.
We omit the messy explicit form of the power $\nu$.
Noting that $-{\mathcal C}_4-{\mathcal C}_5+{\mathcal C}_8 = -\frac{1}{2}(\xi+\zeta)\ell_{21}$, the rest is straightforward.
\end{proof}

In view of the symmetry \eqref{alphauw-R} and the comment after Theorem \ref{th:Rm},
one can also reproduce the original form of \cite[equation~(7.12)]{IKT1}
without the exchange of components $1\leftrightarrow 3$ by
specializing the parameters in \eqref{Rm1} as
\begin{gather}
{\mathcal C}_1 = \frac{1}{2}(\xi-\zeta)\ell_{23}+\frac{1}{2}\zeta\ell_{13},
\qquad
{\mathcal C}_2 = -{\mathcal C}_8 = -T - \frac{1}{2}\zeta\ell_{13}, \qquad
{\mathcal C}_3=0,\nonumber
\\
{\mathcal C}_4 = T + \frac{1}{2}(\xi+\zeta)\ell_{23},\qquad
{\mathcal C}_6 = \frac{1}{2}\zeta\ell_{13},
\qquad
{\mathcal C}_7 = {\mathcal C}_5+{\mathcal C}_6-{\mathcal C}_8,\label{CT2}
\end{gather}
and taking the limit $-T$, $\mathcal{C}_5$,
$\mathcal{C}_7 \rightarrow -\infty$.

These results may be regarded as modular $R$ versions of Theorem
\ref{th:fg} at the level of matrix elements. In fact, under the
specialization \eqref{pare1}, one has
\begin{gather*}
{\mathcal C}_1 =\frac{1}{2}\bigl(\tilde{b}_1-\tilde{b}_2 + 2\tilde{c}_1+\tilde{d}_2-\tilde{d}_3\bigr),
\qquad
{\mathcal C}_2 = -\frac{1}{2}\bigl(\tilde{b}_1+\tilde{d}_3\bigr),
\qquad
{\mathcal C}_3 =0,
\qquad
{\mathcal C}_4 = \frac{1}{2}\bigl(\tilde{b}_2+\tilde{d}_2\bigr),
\\
{\mathcal C}_5 = \frac{1}{2}\bigl(-\tilde{d}_1+\tilde{d}_3\bigr),
\qquad
 {\mathcal C}_6 = \frac{1}{2}\bigl(\tilde{b}_1-\tilde{b}_3\bigr),
\qquad
{\mathcal C}_7 = \frac{1}{2}\bigl(-\tilde{d}_1-\tilde{b}_3\bigr),
\qquad
{\mathcal C}_8 = \frac{1}{2}\bigl(\tilde{b}_1+\tilde{d}_3\bigr).
\end{gather*}
See \eqref{C8} and \eqref{cbc}. Then the limit \eqref{elim} with
$e_i$, $d_i$ replaced with $\tilde{e}_i$, $\tilde{d}_i$ can be
identified with~$-T, \mathcal{C}_5,
\mathcal{C}_7 \rightarrow -\infty$ for \eqref{CT2}.

\appendix

\section[Supplement to Section 3.4]{Supplement to Section \ref{ss:ms}}\label{ap:sup}

The two sides of the inhomogeneous twisted tetrahedron equation \eqref{ihte} yield the
following monomial transformation:
\begin{gather}
Y^{(22)}_1 \mapsto Y_1, \qquad
Y^{(22)}_2 \mapsto Y_2, \qquad
Y^{(22)}_3 \mapsto qY^{-1}_{7}Y^{-1}_{12}Y^{-1}_{13}Y^{-1}_{14},\qquad
Y^{(22)}_4 \mapsto Y_4,\nonumber
\\
Y^{(22)}_5 \mapsto q^{-1}Y_5Y_6, \qquad
Y^{(22)}_6 \mapsto q Y_{13}Y_{14}, \qquad
Y^{(22)}_7 \mapsto Y^{-1}_{13},\qquad
Y^{(22)}_{8} \mapsto q^{-1}Y_{12}Y_{13},\nonumber\\
Y^{(22)}_{9} \mapsto qY_{8}Y_{9},
\qquad
Y^{(22)}_{10} \mapsto q^{-1}Y_{10}Y_{11}, \qquad
Y^{(22)}_{11} \mapsto Y_3Y_6Y_7Y_{11}Y_{12},\nonumber\\
Y^{(22)}_{12} \mapsto q^{-3}Y^{-1}_{3}Y^{-1}_{6}Y^{-1}_{7}Y^{-1}_{8}Y^{-1}_{14}Y^{-1}_{15},
\qquad
Y^{(22)}_ {13} \mapsto q^{-1}Y_3Y_6Y_7Y_8,\nonumber\\
Y^{(22)}_{14} \mapsto q^{-1}Y^{-1}_{3}Y^{-1}_{6}Y^{-1}_{7}Y^{-1}_{8}Y^{-1}_{11}Y^{-1}_{12},\qquad
Y^{(22)}_ {15} \mapsto q^{-4}Y_3Y_7Y_8Y_{14}Y_{15},\nonumber\\
Y^{(22)}_{16} \mapsto qY_{15}Y_{16},\qquad
Y^{(22)}_{17} \mapsto q^{2}Y_7Y_{12}Y_{13}Y_{14}Y_{17}.\label{y22}
\end{gather}

Let us describe the monomial parts $\tau_{-+++}$ and $\tau_{---+}$ in \eqref{taue} mentioned
in Proposition \ref{pr:mono}
\begin{gather*}
\tau_{-+++}\colon\
\begin{cases}
Y'_1 \mapsto Y_1,\qquad
Y'_2 \mapsto Y_2Y_3Y_4, \qquad
Y'_3 \mapsto q^2Y_3Y_4Y_8, \\ \qquad
Y'_4 \mapsto q^2Y^{-1}_3Y^{-2}_4Y^{-1}_5Y^{-1}_8,
\\
Y'_5 \mapsto Y_4Y_5Y_8, \qquad
Y'_6 \mapsto Y_4 Y_5 Y_6,\qquad
Y'_7 \mapsto Y_7,\qquad
Y'_8 \mapsto Y^{-1}_4,
\\
Y'_9 \mapsto Y_9, \qquad
Y'_{10} \mapsto Y_{10},
\end{cases}
\\
 \tau_{---+}\colon\
\begin{cases}
Y'_1 \mapsto Y_1,\qquad
Y'_2 \mapsto Y_2, \qquad
Y'_3 \mapsto q^{-2}Y^{-1}_4Y^{-1}_5Y_8, \qquad
Y'_4 \mapsto Y^{-1}_8, \qquad
\\
Y'_5 \mapsto Y^{-1}_3Y^{-1}_4Y_8, \qquad
Y'_6 \mapsto Y_6,\qquad
Y'_7 \mapsto Y_3Y_4Y_7,\qquad
Y'_8 \mapsto Y_3Y_4Y_5,
\\
Y'_9 \mapsto q^{-2} Y_4Y_5Y_9, \qquad
Y'_{10} \mapsto Y_{10}.
\end{cases}
\end{gather*}
Note that the image is not necessarily sign coherent.

\section[Formulas for tau\^uw\_-+-+, P\_-+-+ and R\_123 for
 (ve\_1,ve\_2,ve\_3,ve\_4)=(-,+,-,+)]{Formulas for $\boldsymbol{\tau^{\uu\ww}_{-+-+}}$, $\boldsymbol{P_{-+-+}}$ and $\boldsymbol{R_{123}}$\\ for
 $\boldsymbol{(\ve_1,\ve_2,\ve_3,\ve_4)=(-,+,-,+)}$}\label{app:ruw}

Let $\tau_{-+-+}$ be the one in Example \ref{ex:2}.
Under the parametrization by $q$-Weyl algebra generators \eqref{Yw}--\eqref{Ypw},
it is translated into the transformation of the canonical variables as
\begin{align}
\tau^{\uu\ww}_{-+-+}\colon\
\begin{cases}
\uu_1 \mapsto \uu_2 +\kappa_0,
&
\ww_1 \mapsto \ww_2-\ww_3+\kappa_2,
\\
\uu_2 \mapsto \uu_1-\kappa_0,
&
\ww_2 \mapsto \ww_1+\ww_3+\kappa_1,
\\
\uu_3 \mapsto -\uu_1+\uu_2+\uu_3+\kappa_0,
&
\ww_3 \mapsto \ww_3+\kappa_3,
\end{cases}
\label{uwL2}
\end{align}
where $\kappa_r=\kappa_r(\mathscr{P}_1,\mathscr{P}_2,\mathscr{P}_3)$
for $r=0, 1, 2, 3$ is defined, under the condition \eqref{ae0},
by
\begin{gather*}
\kappa_0 = \frac{e_2-e_1}{2},\qquad
\kappa_1 = c_1-c_2+c_3,\qquad
\kappa_2 =d_1-d_2-a_3 - \kappa_0,\\
\kappa_3 = b_1+c_1-b_2-c_2-\kappa_0.
\end{gather*}
This is realized as an adjoint action as
\begin{gather*}
\tau^{\uu\ww}_{-+-+} = \mathrm{Ad}(P_{-+-+}),
\\
P_{-+-+}=
\e^{\tfrac{1}{\hbar}(\uu_1-\uu_2)\ww_3}
\e^{\tfrac{\kappa_0}{\hbar}(\ww_1-\ww_2-\ww_3)}
\e^{\tfrac{1}{\hbar}(\kappa_1\uu_1+\kappa_2\uu_2+\kappa_3\uu_3)}\rho_{12}
\in N_3 \rtimes \mathfrak{S}_3.
\end{gather*}
From \eqref{dode1} and \eqref{dode2},
the formulas analogous to \eqref{RL0} and \eqref{RL1} become as follows:
\begin{gather}
\Psi_q\bigl(\e^{-d_1-c_2-b_3+\uu_1+\uu_3+\ww_1-\ww_2+\ww_3}\bigr)^{-1}
\Psi_q\bigl(\e^{d_1+c_2+b_3+e_3-\uu_1+\uu_3-\ww_1+\ww_2-\ww_3}\bigr)
\nonumber \\
\qquad\times
\Psi_q\bigl(\e^{-d_1-e_1-c_2-b_3-\uu_1+\uu_3+\ww_1-\ww_2+\ww_3}\bigr)^{-1}\nonumber \\
\qquad\times
\Psi_q\bigl(\e^{d_1+c_2+e_2+b_3+e_3-\uu_1+2\uu_2+\uu_3-\ww_1+\ww_2-\ww_3}\bigr)P_{-+-+}\nonumber
\\
\phantom{\qquad\times}{}=
\Psi_q\bigl(\e^{-d_1-c_2-b_3+\uu_1+\uu_3+\ww_1-\ww_2+\ww_3}\bigr)^{-1}
\Psi_q\bigl(\e^{d_1+c_2+b_3+e_3-\uu_1+\uu_3-\ww_1+\ww_2-\ww_3}\bigr)P_{-+-+}
\nonumber\\
\phantom{\qquad\times=}{}\times
\Psi_q\bigl(\e^{b_1+a_2+d_3+e_3-\uu_1+\uu_3-\ww_1+\ww_2-\ww_3}\bigr)^{-1}
\Psi_q\bigl(\e^{-b_1-a_2-d_3+\uu_1+\uu_3+\ww_1-\ww_2+\ww_3}\bigr).
\label{RL3}
\end{gather}

\section{Integral formula involving non-compact quantum dilogarithm}\label{ap:nc}

The following is known as a modular double analogue of the Ramanujan ${}_1\Psi_1$-sum
\begin{align}
\int {\rm d}t \frac{\Phi_\bb(t+u)}{\Phi_\bb(t+v)} \e^{2\pi {\rm i} wt}
&= \frac{\Phi_\bb(u-v-{\rm i}\eta)\Phi_\bb(w+{\rm i}\eta)}{K\Phi_\bb(u-v+w-{\rm i}\eta)}
\e^{-2\pi {\rm i} w(v+{\rm i}\eta)}
\nonumber\\
&= \frac{K\Phi_\bb(v-u-w+{\rm i}\eta)}{\Phi_\bb(v-u+{\rm i}\eta)\Phi_\bb(-w-{\rm i}\eta)}
\e^{-2\pi {\rm i} w(u-{\rm i}\eta)},
\label{rama2}
\end{align}
where $K = \e^{-{\rm i}\pi(4\eta^2+1)/12}$. See \cite[Section~6.3]{FKV01} for
the condition concerning the validity of the integrals. From
$\Phi_\bb(u)\vert_{u \rightarrow -\infty} \rightarrow 1$, their limit
$u, v \rightarrow -\infty$ reduces to
\begin{gather}
\int{\rm d}t \frac{\e^{2\pi {\rm i} w t}}{\Phi_\bb(t+v)} = \frac{\Phi_\bb(w+{\rm i}\eta)}{K}\e^{-2\pi {\rm i}w(v+{\rm i}\eta)},
\label{ram1}
\\
\int {\rm d}t \Phi_\bb(t+u)\e^{2\pi {\rm i} w t}
= \frac{K}{\Phi_\bb(-w-{\rm i}\eta)} \e^{-2\pi {\rm i} w(u-{\rm i}\eta)}.
\label{ram2}
\end{gather}

\section[Explicit form of RLLL relation (7.4)]{Explicit form of $\boldsymbol{RLLL}$ relation (\ref{rlll})}\label{app:rlll}

We write down the explicit form of \eqref{rlll} together with
the corresponding choice of $(abcijk)$ in~\eqref{qybe} or in Figure \ref{fig:qybe}.
As mentioned after Figure \ref{fig:qybe}, there are 18 non-trivial cases.
To save the space, we write
$Y_\alpha = r_\alpha s_\alpha Z^{-1}+ t_\alpha w_\alpha XZ^{-1}X$,
\begin{gather}
(001001)\colon\
 \mathscr{R}(1\otimes X\otimes X) = (1\otimes X \otimes X)\mathscr{R},
\label{d1}\\
(001010)\colon\ \mathscr{R}\bigl(
r_2t_1X \otimes 1 \otimes Y_3
+t_3Z\otimes Y_2 \otimes X\bigr) = r_1t_2\bigl( 1 \otimes X \otimes Y_3\bigr)\mathscr{R},
\label{d2}\\
(001100)\colon\ \mathscr{R}\bigl(
t_3w_1X \otimes Y_2\otimes X
+r_2 Y_1 \otimes 1 \otimes Y_3\bigr)
= r_1r_3\bigl(1 \otimes Y_2 \otimes 1\bigr)\mathscr{R},
\label{d3}\\
(010001)\colon\
r_1t_2\mathscr{R}(1 \otimes X \otimes Z) =
\bigl(r_2t_1X \otimes 1 \otimes Z
+ t_3Y_1 \otimes Z\otimes X \bigr)\mathscr{R},
\label{d4}\\
(010010)\colon\ \mathscr{R}\bigl(
r_2t_1w_3 X \otimes 1 \otimes X + Z \otimes Y_2 \otimes Z\bigr)\nonumber\\
\phantom{(101101)\colon\ }{}\qquad
= \bigl(
r_2t_1w_3 X \otimes 1 \otimes X
+Y_1 \otimes Z \otimes Y_3
\bigr)\mathscr{R},
\label{d5}\\
(010100)\colon\ \mathscr{R}\bigl(
w_1 X \otimes Y_2 \otimes Z
+ r_2w_3Y_1\otimes 1 \otimes X\bigr)
= r_3w_2 \bigl(Y_1 \otimes X \otimes 1\bigr) \mathscr{R},
\label{d6}\\
(011011)\colon\
 \mathscr{R}(X \otimes X \otimes 1) = (X \otimes X \otimes 1)\mathscr{R},
\label{d7}\\
(011101)\colon\ s_3t_2 \mathscr{R}\bigl(
Y_1\otimes X \otimes 1\bigr)
= \bigl(
t_1X \otimes Y_2 \otimes Z
+ s_2t_3 Y_1\otimes 1 \otimes X
\bigr)\mathscr{R},
\label{d8}\\
(011110)\colon\
 s_1s_3 \mathscr{R}\bigl(
1 \otimes Y_2 \otimes 1\bigr)
= \bigl(t_1w_3 X \otimes Y_2 \otimes X
+s_2Y_1\otimes 1 \otimes Y_3
\bigr)\mathscr{R},
\label{d9}\\
(100001)\colon\
 r_1r_3 \mathscr{R}(1\otimes Z \otimes 1) = (t_3w_1 X \otimes Z \otimes X +r_2 Z \otimes 1 \otimes Z)\mathscr{R},
\label{d10}\\
(100010)\colon\
 r_3w_2\mathscr{R}(Z \otimes X \otimes 1)
=\bigl( w_1 X \otimes Z \otimes Y_3
+ r_2w_3 Z \otimes 1 \otimes X\bigr)\mathscr{R},
\label{d11}\\
(100100)\colon\
 \mathscr{R}(X \otimes X \otimes 1) = (X \otimes X \otimes 1)\mathscr{R},
\label{d12}\\
(101011)\colon\
 \mathscr{R} (t_1X \otimes Z \otimes Y_3 + s_2 t_3 Z \otimes 1 \otimes X)
= s_3 t_2 (Z \otimes X \otimes 1)\mathscr{R},
\label{d13}\\
(101101)\colon\ \mathscr{R}\bigl(
s_2t_3w_1 X \otimes 1 \otimes X
+ Y_1 \otimes Z \otimes Y_3\bigr)\nonumber\\
\phantom{(101101)\colon\ }{}\qquad
=\bigl(s_2 t_3w_1X \otimes 1 \otimes X
+ Z \otimes Y_2 \otimes Z\bigr)\mathscr{R},
\label{d14}\\
(101110)\colon\ s_1w_2 \mathscr{R}\bigl(1 \otimes X \otimes Y_3)
=
\bigl(s_2w_1X \otimes 1 \otimes Y_3
+ w_3 Z \otimes Y_2 \otimes X\bigr)\mathscr{R},
\label{d15}\\
(110011)\colon\
 \mathscr{R}(t_1w_3 X \otimes Z \otimes X + s_2 Z \otimes 1 \otimes Z) = s_1s_3 (1 \otimes Z \otimes 1)\mathscr{R},
\label{d16}\\
(110101)\colon\
 \mathscr{R}\bigl (w_3 Y_1 \otimes Z \otimes X + s_2 w_1 X \otimes 1 \otimes Z\bigr)
= s_1 w_2 (1 \otimes X \otimes Z)\mathscr{R},
\label{d17}\\
(110110)\colon\
 \mathscr{R}(1\otimes X\otimes X) = (1\otimes X \otimes X)\mathscr{R}.
\label{d18}
\end{gather}

\subsection*{Acknowledgments}

The authors would like to thank Vladimir Bazhanov, Vladimir Mangazeev,
Sergey Sergeev and Akihito Yoneyama for stimulating discussions. The
authors also thank the anonymous referees for their careful reading
and valuable comments. RI is supported by JSPS KAKENHI Grant Number
19K03440 and 23K03048. AK is supported by JSPS KAKENHI Grant Number 24K06882.
YT is supported by JSPS KAKENHI Grant Number
JP21K03240 and 22H01117. JY and XS are supported by NSFC Grant Number
12375064.

\pdfbookmark[1]{References}{ref}
\LastPageEnding

\end{document}